\documentclass[A4,11pt]{article}

\usepackage[english]{babel}
\usepackage{amssymb}
\usepackage{amsmath}
\usepackage{graphics}
\usepackage{graphicx}
\usepackage{enumitem}
\newtheorem{DE}{Definition}[section]

\newcommand {\sm} {\setminus}

\usepackage{latexsym} %Des symboles de maths en plus.
\usepackage{theorem} %pour changer les styles dans les theoremes.
 \newcommand{\qed}{\relax\ifmmode\hskip2em\Box\else\unskip\nobreak\hfill$\Box$\fi}

 \newtheorem{definition}[DE]{Definition}
\newtheorem{theorem}[DE]{Theorem}
\newtheorem{lemma}[DE]{Lemma}

{\theoremstyle{break}\theorembodyfont{\rmfamily}}
{\theoremstyle{break}\theorembodyfont{\rmfamily}}

\newcounter{claim}
\newenvironment{proof}[1][]%
	{\noindent {\setcounter{claim}{0}\it Proof. }{#1}{}}{\qed\vspace{2ex}}
\newenvironment{claim}[1][]%
	{\refstepcounter{claim}\vspace{1ex}\noindent {(\it\arabic{claim}) {#1}{}}\it}{\vspace{1ex}}
\newenvironment{proofclaim}[1][]%
	{\noindent {}{#1}{}}{This proves~(\arabic{claim}).\vspace{1ex}}

\renewcommand{\emptyset}{\varnothing}
        
\begin{document}

\title{When all holes have the same length}

\author{Jake Horsfield\thanks{School of Computing, University of Leeds, UK and
    Faculty of Computer Science (RAF), Union University, Belgrade,
    Serbia.  Partially supported by EPSRC grant EP/N0196660/1, and
    Serbian Ministry of Education and Science projects 174033 and
    III44006. sc15jh@leeds.ac.uk (Jake Horsfield),
    k.vuskovic@leeds.ac.uk}~, Myriam Preissmann\thanks{Univ. Grenoble
    Alpes, CNRS, Grenoble INP, G-SCOP, 38000 Grenoble, France.
    myriam.preissmann@grenoble-inp.fr, cleophee.robin@grenoble-inp.fr}~, Cl\'eoph\'ee  Robin\footnotemark[2]~\thanks{Univ Lyon, EnsL, UCBL, CNRS, LIP,
    F-69342, LYON Cedex 07, France. Partially
    supported by the LABEX MILYON (ANR-10-LABX-0070) of Universit\'e
    de Lyon, within the program ‘‘Investissements d'Avenir’’
    (ANR-11-IDEX-0007) operated by the French National Research Agency
    (ANR) and by Agence Nationale de la Recherche (France) under
    research grant ANR DIGRAPHS
    ANR-19-CE48-0013-01. nld.sintiari@gmail.com, nicolas.trotignon@ens-lyon.fr\newline}~,\\Ni Luh Dewi Sintiari\footnotemark[3]~, Nicolas Trotignon\footnotemark[3]~~and Kristina Vu\v skovi\'c\footnotemark[1]}

\maketitle

\begin{abstract}
  For every integer $\ell \geq 7$, we give a structural
  description of the class of graphs whose chordless cycles of length
  at least~4 all have length $\ell$.
\end{abstract}

\section{Introduction}

A \emph{hole} in a graph is an induced cycle of length at least~4.
For an integer $k\geq 7$, we study the class $\mathcal C_k$ of graphs
where every hole has length $k$. Note that when $k$ is even, this is a
class of perfect graphs, and when $k$ is odd, this is a class of
even-hole-free graphs. Both these classes are well studied and we do
not recall their definition. They have celebrated decomposition
theorems (see~\cite{chudnovsky.r.s.t:spgt} and~\cite{dsv:ehf}), but no
full structural description.  This motivates studying $\mathcal C_k$.

In~\cite{DBLP:journals/dam/Penev20}, the class of ($4K_1$, $C_4$,
$C_6$, $C_7$)-free graphs is studied. It is a subclass of
$\mathcal C_5$.  In \cite{DBLP:journals/gc/FoleyFHHL20}, the class of
($4K_1$, $C_4$, $C_6$)-free graphs is studied. In this class, every
hole has length~5 or~7.  In~\cite{DBLP:journals/jgt/BoncompagniPV19},
the class of rings of length $k$ is defined for every integer
$k\geq 4$ (see Section~\ref{sec:truemper} for the definition), and it
is used as a basic class for several decompositions theorems.  Rings
of length $k$ form a subclass of $\mathcal C_k$.
In~\cite{maffray2019coloring}, a polynomial time algorithm that colors
every ring is given. In~\cite{DBLP:journals/corr/abs-2007-11513}, it
is proved that for every fixed integer $k$, there exist rings of length
$k$ of arbitrarily large rankwidth.

In~\cite{bergerSeymourSpirkl} and~\cite{chiuLu}, polynomial-time algorithms that, given a
graph $G$ and $u, v \in V (G)$, decide whether there exists a path
from $u$ to $v$ that is not a shortest path are described.  
%The running time is $O(|V(G)|^{18})$.  
It is easy to deduce from such algorithms an
algorithm to recognize $\mathcal C_k$ in polynomial time.

Here, we provide first a structural description of graphs in
$\mathcal C_{2\ell +1}$ for any $\ell \geq 3$. It says that every
graph in the class is constructed in some precise way or has a
universal vertex or has a clique cut.  The formal statement is given
in Theorem~\ref{th:struct}. 
This work appears in the PhD thesis of Cl\'eoph\'ee
Robin~\cite{robin:these}.  Part of it and algorithmic applications
will appear in the PhD thesis of Jake Horsfield.
In the second part of this work, we 
provide a similar description of graphs in $\mathcal C_{2\ell}$ for
any $\ell \geq 4$.   The formal statement is given
in Theorem~\ref{th:structEven}.

A similar description was obtained independently by Linda Cook and
Paul Seymour. Much of it forms part of the PhD
thesis~\cite{phdthesis:cook} of Linda Cook and both groups decided to
write a joint work based on this version,
see~\cite{cookEtAl:holes}.  Our statement is different but equivalent to the one in~\cite{cookEtAl:holes} as will be shown in the PhD thesis of Jake Horsfield. We publish the present version as a
preprint because the approach is not the same  and for later reference. 

\section{Definition and notation}

We denote by $\overline{G}$ the complement of a graph $G$. 

When $x$ is a vertex of a graph $G$ and $A$ is a subset of vertices of
$G$ or an induced subgraph of $G$, we denote by $N_A(x)$ the set of
neighbors of $x$ that are in $A$. Note that $x\notin N_A(x)$.  We set
$N_A[x] = \{x\} \cup N_A(x)$.  If $X\subseteq V(G)$, we set
$N_A(X) = (\bigcup_{x\in X}N_A(x)) \sm X$ and $N_A[X] = N_A(X) \cup
X$. We sometimes write $N$ instead of $N_{V(G)}$ (when there is no risk of
confusion).

A set $X\subseteq V(G)$ is \emph{complete} to a set $Y\subseteq V(G)$
if they are disjoint and every vertex of $X$ is adjacent to every
vertex of $Y$.  A set $X\subseteq V(G)$ is \emph{anticomplete} to a
set $Y\subseteq V(G)$ if they are disjoint and no vertex of $X$ is
adjacent to a vertex of $Y$.  We sometimes say that $x$ is complete
(resp.\ anticomplete) to $Y$ to mean that $\{x\}$ is complete (resp.\
anticomplete) to $Y$.

A vertex $v$ in a graph $G$ is \emph{isolated} if it has no neighbors
in $G$. It is \emph{universal} if it is adjacent to all vertices of
$G\sm v$.  A graph $G$ is \emph{connected} if for every pair of
vertices $u$, $v$ there exists a path from $u$ to $v$ in $G$.  A graph
is \emph{anticonnected} if its complement is connected.  A
\emph{connected component} of a graph $G$ is a subset $X$ of $V(G$
such that $G[X]$ is connected and $X$ is maximal w.r.t.\ this
property. An \emph{anticonnected component} of a graph $G$ is a subset
$X$ of $V(G)$ such that $G[X]$ is anticonnected and $X$ is maximal
w.r.t.\ this property.

We will use the notion of \emph{hypergraph}; that is, a structure
similar to graphs except that the edges (called \emph{hyperedges}) may
contain an arbitrary positive number of vertices.  While all the
graphs that we use are simple, in hypergraphs, we allow hyperedges
that contain a single vertex and multiple hyperedges (that is, there
can be different hyperedges on the same set of vertices). Observe that
we do not allow an empty hyperedge.

A \emph{cutset} in a graph $G$ is a set $S$ of vertices such that
$G\sm S$ is disconnected. A \emph{clique} in a graph is a set of
pairwise adjacent vertices. In a graph, we view the empty set as a
clique, and as a clique cutset of any disconnected graph. 
A \emph{stable set} in a graph is a set of pairwise non-adjacent vertices.

For $k\geq 1$, we denote by $P_k$ the path on $k$ vertices, that is,
the graph with vertex-set $\{p_1,\dots, p_k\}$ and edge-set
$\{p_1p_2, \dots, p_{k-1}p_k\}$.  We denote it by $p_1p_2\dots p_k$.
If $1\leq i\leq j\leq k$, we then denote by $p_iPp_j$ the path
$p_i p_{i+1}\dots p_j$.  For $k\geq 3$, we denote by $C_k$ the cycle
on $k$ vertices; that is, the graph with vertex-set
$\{p_1,\dots, p_k\}$ and edge-set
$\{p_1p_2, \dots, p_{k-1}p_k, p_kp_1\}$.  We denote it by
$p_1p_2 \dots p_kp_1$. We denote it by
$p_1p_2 \dots p_kp_1$.  When $C_k$ is a subgraph of a graph $G$ (possibly not induced), an edge with both ends in $\{p_1,...,p_k\}$ that is not an edge of $C_k$ is called a \emph{chord} of $C_k$. We denote by $2K_2$ the complement of $C_4$.

We say that \emph{$P$ is a path in a graph $G$} (or \emph{$P$ is a
  path of $G$}) to mean that $P$ is a path that is an induced subgraph
of $G$.  A \emph{hole} in a graph $G$ is a cycle of length at least~4
that is an induced subgraph of $G$.  The \emph{length} of a path,
cycle or hole is the number of its edges.  A hole is \emph{even} or
\emph{odd} depending on the parity of its length.

A graph $G$ \emph{contains} a graph $H$ if $H$ is isomorphic to an
induced subgraph of $G$ and $G$ is \emph{$H$-free} if $G$ does not
contain $H$.  For a class of graphs $\mathcal H$, we say that $G$ is
$\mathcal H$-free, if $G$ is $H$-free for all $H$ in $\mathcal H$.

\section{A survey of some classes of graphs}

Here we present several known classes of graphs and their
properties. We do not need all of them, but we believe that presenting
them all gives a better understanding of the class we work on.

\subsection{Classes of perfect graphs}

A graph is \emph{chordal} if it is hole-free.  A graph is a \emph{cograph} if
it is $P_4$-free. A graph is a \emph{split graph} if it is ($C_4$, $C_5$,
$2K_2$)-free.  A graph is a \emph{quasi-threshold graph} if it is ($P_4$,
$C_4$)-free (quasi-threshold graphs are sometimes called \emph{trivially
perfect graphs}, see~\cite{DBLP:journals/dm/Golumbic78}).  A graph is a
\emph{threshold graph} if it is ($P_4$, $C_4$, $2K_2$)-free (threshold graphs
are sometimes called \emph{graphs with Dilworth number~1}).  A graph is a
\emph{half  graph} if it is $(3K_1, C_4, C_5)$-free.

Observe that these six classes are all classes of perfect graphs.  The
classes of cographs, split graphs and threshold graphs are
self-complementary while the classes of chordal graphs,
quasi-threshold and half graphs are not.  In Figure~\ref{f:venn}, a
Venn diagram of seven graph classes is represented
($\overline{\rm chordal}$ and $\overline{\rm quasi-threshold}$ mean
complements of chordal and quasi-threshold graphs respectively).  In
every set, a typical example of the class is represented. The diagram
provides several alternative definitions of the classes we work on
(for instance, a threshold graph is a split cograph, a split graph is
a chordal graph whose complement is chordal, and so on).  All the
information given by Figure~\ref{f:venn} is easily recovered from the
definitions of the corresponding classes.

\begin{figure}
\begin{center}
 \includegraphics[width=12cm]{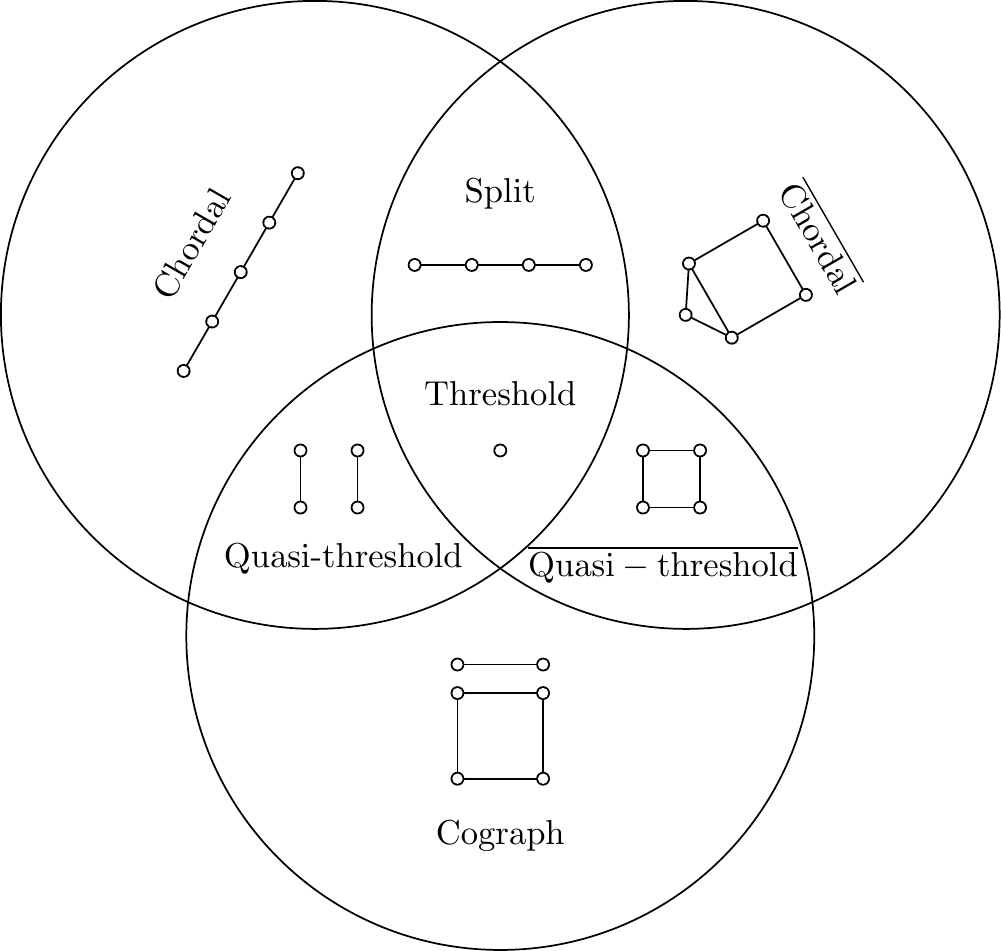}
\end{center}
\caption{Venn diagram of seven classes of graphs\label{f:venn}}
\end{figure}

\begin{theorem}[\cite{dirac:chordal}]
  \label{th:chordal}
  A graph $G$ is chordal if and only if every non-complete
  induced subgraph of $G$ has a clique cutset. 
\end{theorem}

\begin{theorem}[\cite{seinsche:P4}]
  \label{th:cograph}
  A graph $G$ is a cograph if and only if every induced subgraph of
  $G$ on at least two vertices is either not connected or not
  anticonnected.
\end{theorem}

\begin{theorem}[\cite{foldesHammer:split}]
  \label{th:split}
  A graph $G$ is a split graph if and only if $V(G)$ can be
  partitioned into a (possibly empty) clique and a (possibly empty)
  stable set.
\end{theorem}

The \emph{line graph} of a hypergraph $\mathcal H$ is the graph $G$ whose
vertex-set is $E(\mathcal H)$ and where two hyperedges of $\mathcal H$ are adjacent
vertices of $G$ whenever their intersection is non-empty.  Recall that
in this paper, hypergraphs may have multiple hyperedges (that are
distinct hyperedges with the same vertices in them).  A hypergraph is
\emph{laminar} is for every pair $X, Y$ of hyperedges, either
$X\subseteq Y$ or $Y\subseteq X$ or $X\cap Y = \emptyset$.

\begin{theorem}[\cite{wolk:Laminar}]
  \label{th:laminar}
  For all graphs $G$ the following statements are equivalent.
  \begin{enumerate}
  \item $G$ is a quasi-threshold graph.
  \item Every induced subgraph of $G$ is disconnected or has a
    universal vertex. 
  \item $G$ is the line graph of a laminar hypergraph.
  \end{enumerate}
\end{theorem}

\begin{definition}
  If $u$ and $v$ are vertices of a  graph $G$ we write $u\leq_G v$ if $N(u)\sm
  \{v\}\subseteq N(v)\sm \{u\}$ and $u <_G v$ if $N(u)\sm \{v\}\subsetneq
  N(v)\sm \{u\}$.
\end{definition}

\begin{lemma}
  The relations $\leq_G$ and $<_G$ are transitive, that is, for all
  vertices $u, v, w$ of some graph $G$, if $u\leq_G v$ and $v\leq_G w$,
  then $u\leq_G w$ (resp.\ if $u<_Gv$ and $v<_G w$, then $u<_G w$).
\end{lemma}

We define $\geq_G$ and $>_G$ accordingly (i.e.\ $x\geq_G y$ if and only if
$y\leq_G x$) and extend these relations to sets of vertices $X$ and $Y$ as
follows: $X\leq_G Y$ if and only if for every $x\in X$ and $y\in Y$, $x\leq_G
y$ and so on.

\begin{theorem}[\cite{chvatalHammer:77}]
  \label{th:threshold}
  For all graphs $G$ the following statements are equivalent.
  \begin{enumerate}
  \item $G$ is a threshold graph.
  \item\label{i:elimination} Every induced subgraph of $G$ has an
    isolated vertex or a universal vertex.
  \item For all vertices $u$ and $v$ of $G$, $u\leq_G v$ or $v\leq_G
    u$.
  \end{enumerate}
\end{theorem}

It is convenient to sort the vertices of a threshold graph.  Formally, an
ordering $v_1$, \dots, $v_k$ such that $v_i\leq_G v_j$ for all integers $i$ and
$j$ satisfying $1 \leq i\leq j \leq k$ is called a \emph{domination ordering}.
There is another convenient ordering of the vertices of a threshold graph.  By
characterization~\eqref{i:elimination} in Theorem~\ref{th:threshold}, every
threshold graph can be obtained by the following inductive process: start with
a vertex $u_1$, assume for some $k \geq 1$ that vertices $u_1, \dots, u_k$ are
already constructed, and then add a vertex $u_{k+1}$ that is either complete or
anticomplete to $\{u_1, \dots, u_k\}$.  The order $u_1, \dots, u_n$ is then
called an \emph{elimination ordering} of the threshold graph (and it is not a
domination ordering in general).

\begin{figure}

\includegraphics[height=5cm]{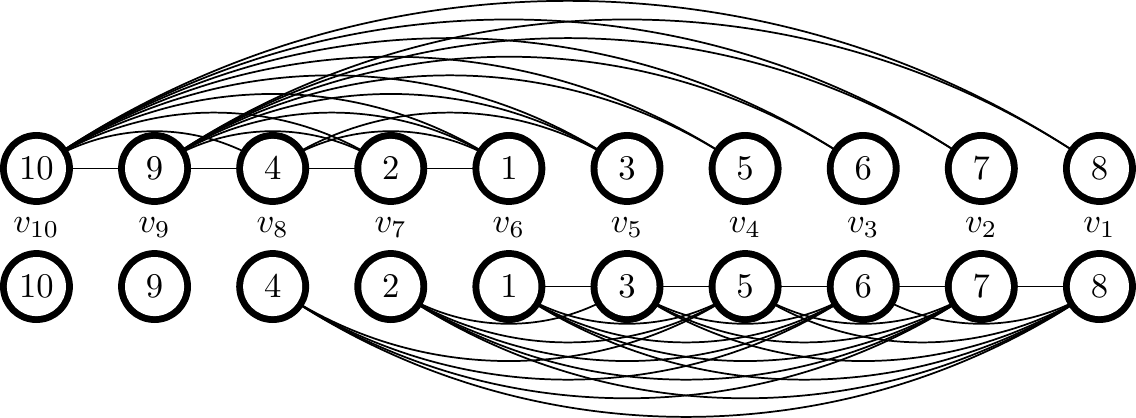}

\caption{A threshold graph and its complement\label{f:threshold}}
\end{figure}

An example is represented in Figure~\ref{f:threshold}.  On the top, a
threshold graph $J$ on $\{v_1, \dots, v_{10}\}$ is represented for which
$(v_1, \dots, v_{10})$ is a domination ordering. Vertices are circles
with a number in them that gives the place of the vertex in the
elimination ordering.  On the bottom, the complement $J'$ of $J$ is
represented.  It is also a threshold graph but the domination ordering
is reversed (it is $(v_{10}, \dots, v_{1})$), while the elimination
ordering remains the same.

\begin{theorem}[Folklore]\label{t:half-graph}
  A graph $G$ is a half graph if and only if $V(G)$ can be partitioned
  into two (possibly empty) cliques $K$ and $K'$ such that for all
  vertices $x$ and $y$ in $K$ (resp.\ in $K'$), either $x \leq_G y$ or
  $y \leq_G x$.
\end{theorem}

\begin{proof}
  If $G$ is a half graph, then the complement of $G$ contains (as a subgraph,
  not necessarily induced) no cycle of odd length because a shortest such cycle
  cannot have length~3 (it would yield a $3K_1$ in $G$), cannot have length~5
  (it would yield a $C_5$ in $G$) and cannot have length at least~7 (it
  would yield a $C_4$ in $G$).  It follows that the complement of $G$ is a
  bipartite graph, so $V(G)$ can be partitioned into two cliques as claimed.
  The condition on $\leq_G$ then follows from the fact that $G$ contains no
  $C_4$.

  The converse statement is clear. 
\end{proof}

\subsection{Classes defined by excluding Truemper configurations}
\label{sec:truemper}

\emph{Truemper configurations} are graphs that play a role in many
decomposition theorems, see~\cite{vuskovic:truemper}. 
They are the prisms, thetas, pyramids and
wheels. Let us define them.

A \emph{prism} is a graph made of three vertex-disjoint paths
$P_1 = a_1 \dots b_1$, $P_2 = a_2 \dots b_2$, $P_3 = a_3 \dots b_3$ of
length at least 1, such that $a_1a_2a_3$ and $b_1b_2b_3$ are triangles
and no edges exist between the paths except those of the two
triangles.  

A \emph{pyramid} is a graph made of three paths
$P_1 = a \dots b_1$, $P_2 = a \dots b_2$, $P_3 = a \dots b_3$ of
length at least~1, two of which have length at least 2, vertex-disjoint
except at $a$, and such that $b_1b_2b_3$ is a triangle and no edges
exist between the paths except those of the triangle and the three
edges incident to $a$.  The vertex $a$ is called the \emph{apex} of
the pyramid. 

A \emph{theta} is a graph made of three internally vertex-disjoint
 paths $P_1 = a \dots b$, $P_2 = a \dots b$,
$P_3 = a \dots b$ of length at least~2 and such that no edges exist
between the paths except the three edges incident to $a$ and the three
edges incident to $b$.  

Observe that the lengths of the paths used in the three definitions
above are designed so that the union of any two of the paths induce a
hole.  A prism, pyramid or theta is \emph{balanced} if the three paths
in the definition are of the same length. It is \emph{unbalanced}
otherwise.

\begin{figure}[h]
\begin{center}
 \includegraphics[height=2cm]{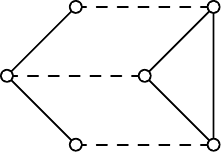}
 \hspace{.2em}
 \includegraphics[height=2cm]{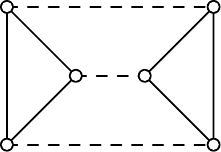}
 \hspace{.2em}
 \includegraphics[height=2cm]{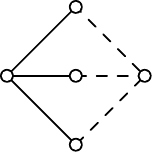}
 \hspace{.2em}
 \includegraphics[height=2cm]{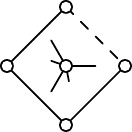}
\end{center}
\caption{Pyramid, prism, theta and wheel (dashed lines represent
 paths)\label{f:tc}}
\end{figure}

A \emph{wheel} $W= (H, c)$ is a graph formed by a hole $H$ (called the
\emph{rim}) together with a vertex $c$ (called the \emph{center}) that
has at least three neighbors in the hole.

A wheel is a \emph{universal wheel} if the center is adjacent to all vertices
of the rim.  A wheel is a \emph{twin wheel} if the center is adjacent to
exactly three vertices of the rim and they induce a $P_3$.  A wheel is
\emph{proper} if it is neither a twin wheel nor a universal wheel.

Truemper configurations are of interest here because of the following
easy observation.

\begin{lemma}
  \label{l:holeTruemper}
  Every unbalanced prism, every unbalanced pyramid, every unbalanced
  theta and every proper wheel contains holes of different lengths.

  Every pyramid contains an odd hole. Every prism and every theta
  contains an even hole.
\end{lemma}

\begin{proof}
  In a prism, pyramid or theta, the union of any two paths used in the
  definition induces a hole. Paths of different lengths are then
  easily used to provide holes of different lengths. In a proper
  wheel, the rim and a shortest hole are holes of different lengths. 

  In a pyramid, paths of the same parity, that exist since there are
  three paths, induce an odd hole. In thetas and prisms, they
  induce an even hole. 
\end{proof}

The following variant is more useful for our study.

\begin{lemma}
  \label{l:holeTruemperS}
  If $\ell\geq 2$ is an integer and $G\in \mathcal C_{2\ell+1}$, then
  every Truemper configuration of $G$ is a twin wheel, a universal
  wheel or a pyramid whose three paths all have length $\ell$.
  
  If $\ell\geq 2$ is an integer and $G\in \mathcal C_{2\ell}$, every Truemper configuration of $G$ is a twin wheel, a universal
  wheel, a theta whose three paths all have length $\ell$ or a prism whose three paths all have length $\ell-1$.
\end{lemma}

\begin{proof}
  Clear from Lemma~\ref{l:holeTruemper}. 
\end{proof}

A graph $G$ is \emph{universally signable} if $G$ is (prism, pyramid,
theta, wheel)-free.

\begin{theorem}[\cite{confortiCKV97}]
  \label{th:us}
  A graph $G$ is universally signable if and only if every induced
  subgraph of $G$ is a hole, a complete graph or has a clique cutset.
\end{theorem}

A graph $G$ is a \emph{ring} if its vertex-set can be partitioned into
$k\geq 4$ sets $K_1, \dots, K_k$ such that (with subscripts understood to be taken
modulo $k$):

\begin{enumerate}
\item $K_1$, \dots, $K_k$ are cliques;
\item for all $i\in \{1, \dots, k\}$, $K_i$ is anticomplete to
  $V(G) \sm (K_{i-1} \cup K_i \cup K_{i+1})$;
\item for all $i\in \{1, \dots, k\}$, some vertex of $K_i$ is complete
  to $K_{i-1} \cup K_{i+1}$;
\item for all $i\in \{1, \dots, k\}$ and all $x, x' \in K_i$, either
  $x \leq_G x'$ or $x' \leq_G x$.
\end{enumerate}

The integer $k$ in the definition above is the \emph{length of the ring}.
Observe that when $k\geq 4$, the hole $C_k$ is a ring of length $k$.  Observe
also that, by Theorem~\ref{t:half-graph}, for any integer $1\leq i \leq k$, the
graph $G[K_i \cup K_{i+1}]$ is a half graph.  
We refer to the cliques $K_1, \ldots, K_k$ as the \emph{cliques of the ring $G$}.

\begin{lemma}
  Every hole in a ring $G$ of length $k$ has length $k$. 
\end{lemma}

\begin{proof}
  We prove that a hole $C$ in $G$ contains at most one vertex in each clique of
  the ring. Suppose otherwise. Let $x, x' \in K_i$ be two vertices of $C$ and
  suppose up to symmetry that $x \leq_G x'$.  Hence, the neighbor of $x$ in
  $C\sm x'$ is also adjacent to $x'$, so $C$ contains a triangle, a
  contradiction.

  Hence, $C$ contains exactly one vertex in each clique of the ring.  So, it
  has length~$k$.
\end{proof}

The following is a corollary of Theorem~1.6
from~\cite{DBLP:journals/jgt/BoncompagniPV19}.

\begin{theorem}
  \label{th:ring}
  If $G$ is (prism, theta, pyramid, proper wheel, $C_4$, $C_5$)-free,
  then one of the following holds.
  \begin{enumerate}
  \item $G$ is a ring of length at least 6;
  \item $G$ has a clique cutset;
  \item $G$ has a universal vertex.
  \end{enumerate}
\end{theorem}

\section{Odd templates}

Here we define and study the main basic class of Theorem~\ref{th:struct}.

\subsection{Modules in threshold graphs}

Let $G$ be a graph.  A \emph{module} of $G$ is a set $X \subseteq V(G)$ such that
every vertex in $V(G)\sm X$ is either complete or anticomplete to $X$.  Observe
that all subsets of $V(G)$ of cardinality 0, 1 or $|V(G)|$ are modules of $G$.
We will use the notion of module only in the context of threshold graphs. The
reader can check that sets of vertices that are intervals for both elimination
and domination orderings are modules. We omit the proof since we do not
need this formally.  We now state three lemmas.

\begin{lemma}
  \label{l:ModuleThr}
  Let $J$ be a threshold graph and $X\subseteq V(J)$ such that
  $|X|\geq 2$. Then $\bar J$ is a threshold graph, $X$ is a
  module of $J$ if and only if it is a module of $\bar J$, and
  exactly one of $J[X]$ and $\bar J[X]$ is anticonnected.
\end{lemma}

\begin{proof}
  Being a threshold graph and module are properties that are closed
  under taking the complement.  By Theorem~\ref{th:threshold}, exactly
  one of $J[X]$ or $\bar J[X]$ contains an isolated vertex, and
  the other one contains a universal vertex.  Hence, since
  $|X|\geq 2$, exactly one of $J[X]$ or $\bar J[X]$ is connected
  and the other one is anticonnected.
\end{proof}

\begin{lemma}
  \label{l:nxcc}
  Let $J$ be a threshold graph. If $X$ is an anticonnected module of
  $J$ that contains at least two vertices, then $N(X)$ is a clique
  that is complete to $X$.  Moreover, $N(X) >_J X$.
\end{lemma}

\begin{proof}
  Since $X$ is a module, $N(X)$ is complete to $X$. Suppose that $N(X)$ is not
  a clique and let $u$ and $v$ be two non-adjacent vertices in $N(X)$.  Since
  $|X|\geq 2$ and $X$ is anticonnected, $X$ contains two non-adjacent vertices
  $u', v'$ that together with $u$ and $v$ form a $C_4$ in $J$.  This
  contradicts $J$ being a threshold graph.

  Suppose that $N(X)>_J X$ does not hold. So, there exists $u\in N(X)$ and
  $v\in X$ with $v\geq_J u$.  Since $u\in N(X)$, $u$ is complete $X$, so $v$ is
  complete to $X\sm \{v\}$.  This contradicts $X$ being anticonnected. 
\end{proof}

\begin{lemma}
  \label{l:ModuleIso}
  Let $J$ be a threshold graph and $X\subseteq V(J)$ a module of
  $J$.  If $X$ contains some isolated vertices of $J$, then either $X$
  contains only isolated vertices of $J$, or $X$ contains all
  non-isolated vertices of $J$.
\end{lemma}

\begin{proof}
  Let $S$ be the set of all isolated vertices of $J$ and
  $T= V(J) \sm S$.  By assumption, $X$ contains a vertex of $S$. If
  $X$ contains only vertices of $S$, then the conclusion holds, so
  suppose that $X$ contains at least one vertex of $T$.  Suppose for a
  contradiction that $X$ does not contain all of $T$.  Since $J[T]$ is
  connected (because it is a threshold graph with no isolated
  vertices and hence by Theorem~\ref{th:threshold} 
  it contains a universal vertex), 
  there exists an edge $uv$ of $J$ with $u\in T\cap X$ and
  $v\in T\sm X$.  Since $X$ contains isolated vertices, this
  contradicts $X$ being a module.
\end{proof}

\subsection{Templates}

For an integer $\ell\geq 2$, an \emph{odd $\ell$-template} is any
graph $G$ that can be built according to the following process.

\begin{enumerate}
\item Choose a threshold graph $J$ on vertex set $\{1, \ldots,k\}$,
  $k\geq 3$.

\item Choose a laminar hypergraph $\mathcal H$ on vertex set
  $\{1, \ldots,k\}$ such that:
  \begin{enumerate}
  \item\label{a:module} every hyperedge $X$ of $\mathcal H$ is a
    module of $J$ of cardinality at least~$2$ and
  \item\label{a:universal} at least one hyperedge $W$ of $\mathcal H$
    contains all vertices of $\mathcal H$.
  \end{enumerate}

\item\label{i:linkP} For each $i \in \{1, \ldots,k\}$, $G$ contains
  two vertices $v_i $ and $v'_i$ that are linked by a path of $G$ of
  length $\ell-1$. The $k$ paths built at this step are vertex
  disjoint and are called the \emph{principal paths} of the odd
  template.

\item\label{tempd} The set of vertices of $G$ is $V(G)= A \cup A' \cup B \cup B' \cup I$ where:
\begin{enumerate}
\item $I$ is the set of all internal vertices of the principal paths,
\item
 $A = \{v_1, \ldots, v_k\}$, 
\item
 $A' = \{v'_1, \ldots, v'_k\}$, 
\item\label{a:chooseB}
 $B = \{v_X : X \text{ hyperedge of } \mathcal H \text{ such that } J[X] \text{ is anticonnected}\}$, 
\item
 $B' = \{v'_X : X \text{ hyperedge of } \mathcal H \text{ such that }\bar{J}[X] \text{ is anticonnected}\}$.
\end{enumerate}

Note that by Lemma~\ref{l:ModuleThr}, for every hyperedge $X$ of
  $\mathcal H$, either $v_X\in B$ or $v'_X\in B'$ (and not both). 

\item\label{tempe} The set of edges of $G$ is defined as follows.
\begin{enumerate}
\item\label{i:BJ} for every $v_i, v_j \in A$, $v_iv_j  \in E(G)$ if and only if $ij \in E(J)$, 
\item\label{i:BJP} for every $v'_i, v'_j \in A'$, $v'_iv'_j  \in E(G)$ if and only if $ij \notin E(J)$,
\item\label{a:last} for every $v_X, v_Y \in B$, $v_X v_Y \in E(G)$ if and only if $X \cap Y \neq \emptyset$,
\item\label{tempe4} for every $v'_X, v'_Y \in B'$, $v'_X v'_Y \in E(G)$ if and only if $X \cap Y \neq \emptyset$,
\item\label{a:makeit}  for every $v_i \in A$, $v_X \in B$, $v_i v_X \in E(G)$ if and only if $i \in N_J[X]$,
\item\label{tempe6} for every $v'_i \in A'$, $v'_X \in B'$, $v'_i v'_X \in E(G)$ if and only if $i \in N_{\bar J}[X]$,
\item\label{tempe7} for every $v\in I$, $v$ is incident to exactly two edges (those
  in its principal path).
\end{enumerate}
\end{enumerate}

The following notation is convenient.

\medskip

\noindent{\bf Notation:}  For every vertex $x\in B$ such that $x=v_X$
  where $X$ is a hyperedge of $\mathcal H$, we set
  $H_x=\{v_i: i\in X\}$. Similarly, for every vertex $x\in B'$ such
  that $x=v'_X$ where $X$ is a hyperedge of $\mathcal H$, we set
  $H'_x=\{v'_i: i\in X\}$.  

\medskip

We now list some properties of templates that follow directly from the definition.

\begin{enumerate}[label=(\roman*)]
\item\label{t:threshold} $G[A]$ is a threshold graph isomorphic to $J$
  and $G[A']$ is a threshold graph isomorphic to $\bar J$ (and hence
  to the complement of $G[A]$).

\item\label{t:antiMod} For all $x\in B$, $H_x$ is a
  module of $G[A]$ and $G[H_x]$ is anticonnected. Also for all
  $x\in B'$, $H'_x$ is a module of $G[A']$ and $G[H'_x]$ is
  anticonnected.

\item\label{t:qTh} $G[B]$ is isomorphic to the line graph of the hypergraph
  $\mathcal H_B$ on vertex set $A$ and hyperedge set
  $\{H_x : x\in B\}$.  Also $G[B']$ is isomorphic to the line graph of
  the hypergraph $\mathcal H_{B'}$ on vertex set $A'$ and hyperedge
  set $\{H'_x : x\in B'\}$.  Hence $G[B]$ and $G[B']$ are a
  quasi-threshold graphs by Theorem~\ref{th:laminar}.

\item\label{t:makeIt} There is an edge between $v_i \in A$ and
  $x \in B$ if and only if $v_i \in N_A[H_x]$, and there is an edge
  between $v'_i \in A'$ and $x \in B'$ if and only if
  $v'_i \in N_{A'}[H'_x]$.
\end{enumerate}

\begin{lemma}
  \label{l:ABuniv}
  There exist vertices $w$ and $w'$ that are universal vertices in
  respectively $G[A \cup B]$ and $G[A' \cup B']$, and such that either
  $w\in A$ and $w'\in B'$, or $w\in B$ and $w'\in A'$.
\end{lemma}

\begin{proof}
  By Theorem~\ref{th:threshold}, $G[A]$ contains a vertex $u$ that is
  either universal or isolated.  If $u$ is universal, then for every
  $x\in B$, $u\in N(H_x)$ ($u$ cannot be in $H_x$ since $G[H_x]$ is
  anticonnected by property \ref{t:antiMod} of templates).  So, $u$ is adjacent to $x$ by
  property~\ref{t:makeIt} of templates. Hence, $u$ is a universal
  vertex of $G[ A\cup B]$.

  Otherwise, $u$ is an isolated vertex of $G[A]$.  So, $G[A]$ is
  anticonnected. Hence, the vertex $w$ corresponding to the hyperedge
  $W$ from condition~\eqref{a:universal} of templates is in $B$.  By
  property~\ref{t:makeIt} of templates, $w$ is a universal vertex of
  $G[A\cup B]$.
  
  The proof for $G[A' \cup B']$ is similar.  So, $w$ and $w'$ exist,
  and by the way we construct them, we see that either
  $w\in A$ and $w'\in B'$, or $w\in B$ and $w'\in A'$.
\end{proof}

Let $w$ and $w'$ be as in Lemma~\ref{l:ABuniv}. The 7-tuple
$(A, B, A', B', I, w, w')$ is then called an \emph{$\ell$-partition}
of $G$.

Let us give a simple example.  Consider an integer $\ell \geq 2$ and a
threshold graph $J$ on three vertices $\{1, 2, 3\}$ with no edges. So,
$G[A]$ has no edges, $G[A']$ is a triangle on three vertices
$v'_1, v_2', v'_3$, and for $i=1, 2, 3$, there is a path of length
$\ell - 1$ from $v_i$ to $v'_i$.  Consider $\mathcal H$ the hypergraph
on $\{1, 2, 3\}$ with a unique hyperedge that is $\{1, 2, 3\}$.  We
now see that $G$ is a balanced pyramid with apex $w$ and triangle
$v'_1v'_2v'_3$.  Under these circumstances, the sets
$A=\{v_1, v_2, v_3\}$, $B=\{w\}$, $A'= \{v'_1, v'_2, v'_3\}$,
$B'=\emptyset$, $I= V(G) \sm (A\cup B \cup A' \cup B')$, $w$ and
$v'_3$ form an $\ell$-partition of $G$.

It is worth noting that the $\ell$-partition above is not unique.
Here is another one. Call $u$ the neighbor of $v'_3$ in the path from
$v_3$ to $v'_3$ with interior in $I$ (possibly, $u=v_3$).  Set
$A_1 = \{w, v_1, v_2\}$, $B_1= \emptyset$, $A'_1= \{u, v'_1, v'_2\}$,
$B'_1=\{v'_3\}$ and $I_1= V(G) \sm (A_1\cup B_1 \cup A'_1 \cup B'_1)$.
It can be checked that $A_1$, $B_1$, $A'_1$, $B'_1$, $I_1$, $w$ and
$v'_3$ form another $\ell$-partition of $G$.  See
Figure~\ref{f:pyr}. Some edges are dashed in several ways, this will
be explained later, so far, they are just edges of $G$.

\begin{figure}
\begin{center}
 \includegraphics{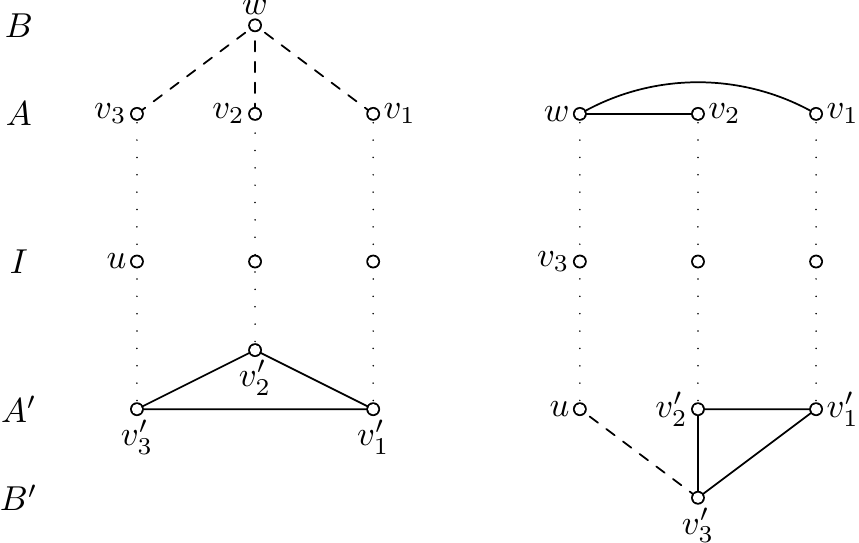}
\end{center}
\caption{Two $3$-partitions of a pyramid whose paths
 all have length 3\label{f:pyr}}
\end{figure}

\begin{lemma}
  \label{l:pyrTemplate}
  For all integers $\ell\geq 2$, every pyramid $\Pi$ such that $\Pi\in
  \mathcal C_{2\ell +1}$ is an odd $\ell$-template. 
\end{lemma}

\begin{proof}
  Since $\Pi\in \mathcal C_{2\ell +1}$, its three paths have
  length~$\ell$.  The explanations above show it is an odd
  $\ell$-template. 
\end{proof}

We now give a more complicated example represented in
Figure~\ref{f:template}. The
threshold graph $J$ has 10 vertices.  Each vertex of $G[A]$ and
$G[A']$ is represented with a number in a circle that represents the
elimination ordering of the threshold graph it belongs to.  The
hypergraph $\mathcal H$ has the following hyperedges:
$X_1 = \{1, 2\}$, $X_2 = \{1, 2, 3\}$, $X_3 = \{9, 10\}$,
$X_4 = \{5, 6, 7\}$, $X_5 = \{5, 6, 7, 8\}$, $X_6 = \{4, 5, 6, 7, 8\}$
and $X_7 = \{1, \dots, 10\}$.  The vertex of $B\cup B'$ corresponding
to a hyperedge $X_i$ is denoted by $x_i$.

\begin{figure}

 \hspace{-1em}\includegraphics[height=18cm]{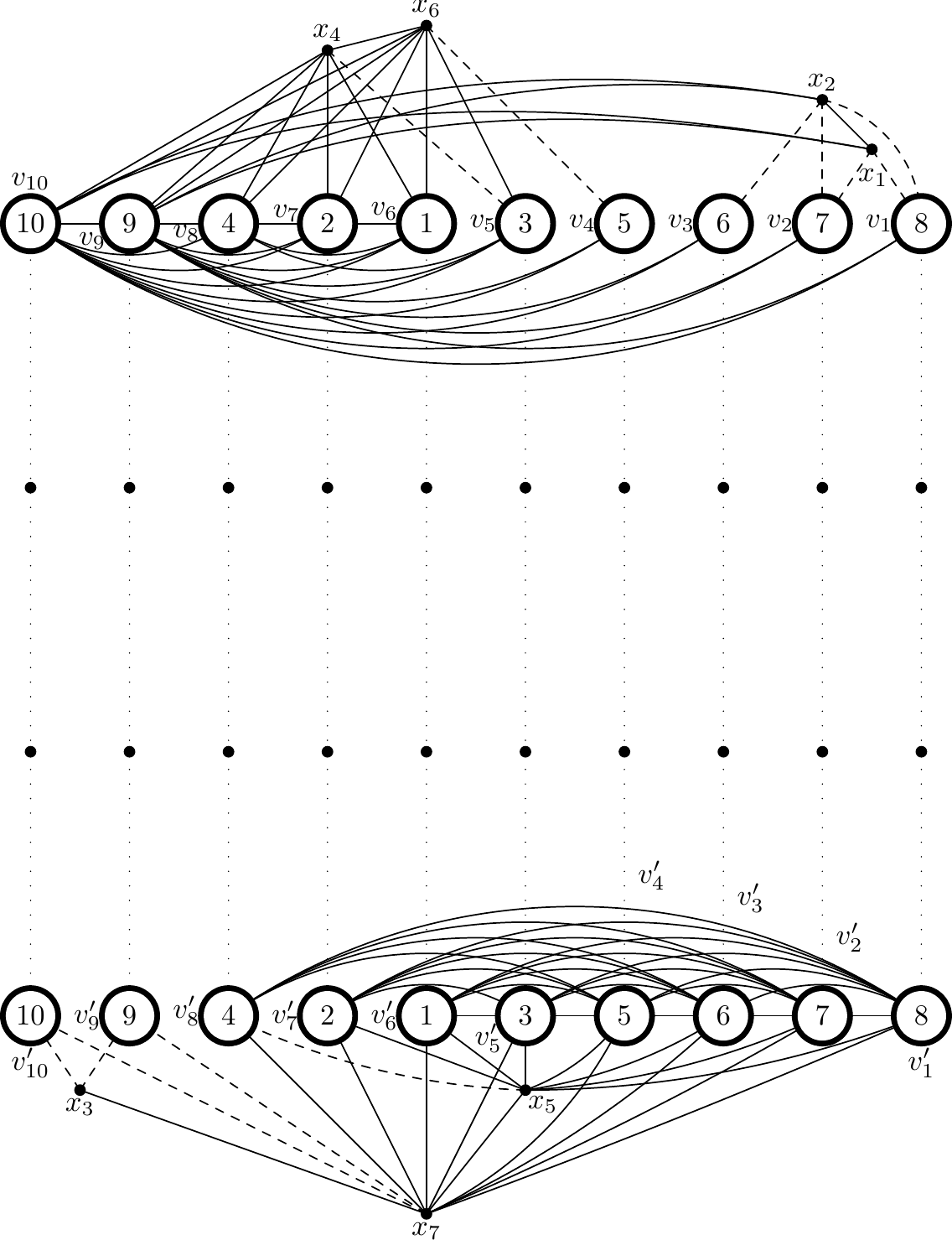}

 \caption{An odd $4$-template \label{f:template}}
\end{figure}

\subsection{Structure of odd templates}

Throughout this subsection, $\ell\geq 2$ is an integer and
$(A, B, A', B', I, w, w')$ is an $\ell$-partition of an odd
$\ell$-template $G$. 

\begin{lemma}
  \label{l:recoverHx}
  If $x\in B$ (resp. $x\in B'$), then $H_x$ (resp.\ $H'_x$) is the
  unique anticomponent of $G[N_A(x)]$ (resp.\ $G[N_{A'}(x)]$) that
  contains at least two vertices.
\end{lemma}

\begin{proof}
  Since $G$ is $C_4$-free and $N_A(x)$ contains at least two
  non-adjacent vertices, $G[N_A(x)]$ contains a unique anticomponent
  $X$ of size at least~2.  Since by property~\ref{t:antiMod} of
  templates, $H_x$ is an anticonnected module of $G[A]$, it is also an
  anticonnected module of $G[N_A(x)]$. Since every vertex of $N_A(x)$
  is either in $H_x$ or complete to $H_x$, $H_x$ must be an
  anticomponent of $G[N_A(x)]$, and since it contains at least two
  vertices, it is equal to $X$.
\end{proof}

\begin{lemma}
  \label{lt:deuxsommetsdeB}
  If $x, y \in B$ (resp.\ $\in B'$) are such that $xy\notin E(G)$,
  then $H_x \cup \{x\}$ (resp.\ $H'_x \cup \{x\}$)  is anticomplete to
  $ H_y\cup \{y\}$ (resp.\ $H'_y \cup \{y\}$).   
\end{lemma}

\begin{proof}
  Suppose $x,y\in B$ and $xy\notin E(G)$. Then by condition \eqref{a:last}
  of templates, $H_x$ and $H_y$ are disjoint.

  Suppose there is at least one edge from $H_x$ to $H_y$.  Since they
  are both modules of $G[A]$, it follows that $H_x$ is complete to
  $H_y$, so $H_y \subseteq N(H_x)$.  Hence, by Lemma~\ref{l:nxcc},
  $H_y$ is a clique. Since $H_y$ contains at least two vertices, this
  contradicts $H_y$ being anticonnected.  So, $H_x$ is anticomplete to
  $H_y$.  Hence, by property~\ref{t:makeIt} of templates, $x$ is
  anticomplete to $H_y$ and $y$ is anticomplete to $H_x$. So,
  $H_x \cup \{x\}$ is anticomplete to $ H_y\cup \{y\}$ because
  $xy\notin E(G)$ holds from our assumption.

  The proof for $x,y\in B'$ is similar. 
\end{proof}

\begin{lemma}
  \label{l:Bdeg3}
  Every vertex of $G$ has degree at least~2 and every vertex of
  $B\cup B'$ has degree at least~3.
\end{lemma}

\begin{proof}
  Vertices in $I$ are all in the interior of some path, so they have
  degree at least~2. 
  
  Vertex $w$ has degree at least~2 since $|A|\geq 3$.  A vertex
  $v\in A\sm\{w\}$ therefore has degree at least~2 (one neighbor in
  $I$, and $w$). So every vertex of $A$ has degree at least 2. The
  proof for $A'$ is similar.

  Let $x$ be a vertex of $B$.  If $x=w$, then $x$ has degree at
  least~3 (because $|A|\geq 3$), so we may assume $x\neq w$.  By,
  property~\ref{t:antiMod} of templates, $|H_x|\geq 2$ and
  $w\notin H_x$ because $H_x$ is anticonnected.  So $x$ has degree at
  least~3 as claimed (at least two neighbors in $H_x$, and $w$). The
  proof for $x\in B'$ is similar.
\end{proof}

The following shows that odd templates can be considered as a
generalization of balanced pyramids (we do not need it and include it
because we believe it helps understanding the structure of the class
we work on).

\begin{lemma}
  For every integer $\ell\geq 2$, every odd $\ell$-template $G$
  contains a pyramid.
\end{lemma}

\begin{proof}
  Consider three vertices $v_i, v_j$ and $v_h$ in $A$ and the corresponding
  vertices $v'_i, v'_j$ and $v'_h$ in $A'$.  Exactly one of $G[\{v_i, v_j,
  v_h\}]$ and $G'[\{v'_i, v'_j, v'_h\}]$ is connected (because they have three
  vertices and one is isomorphic to the complement of the other).  So, up to
  symmetry, we may assume that $G[\{v_i, v_j, v_h\}]$ is disconnected (and
  therefore contains at most one edge).

  Note that $w$ is distinct from $v_i, v_j$ and $v_h$ since $G[\{v_i, v_j, v_h\}]$ is
  disconnected.  We see that $w$ and the three principal paths linking $\{v_i,
  v_j, v_h\}$ to $\{v'_i, v'_j, v'_h\}$ form a pyramid (if $G[\{v_i, v_j,
  v_h\}]$ contains one edge $e$, then the triangle is formed by $e$ and $w$, and
  otherwise it is $v'_iv'_jv'_h$).
\end{proof}

From the definition of odd $\ell$-templates, every vertex $x\in B$
corresponds to a set $H_x\subseteq A$.  These sets form a hypergraph
$\mathcal H_B$ on the vertex-set $A$ (that is isomophic to a
sub-hypergraph of $\mathcal H$).  Let us build an extention
$\mathcal H_A$ of $\mathcal H_B$ by adding more hyperedges: for every
vertex $v\in A$, we add the hyperedge
$$H_v = N_A[v] \cap \{u \in A : u\leq_{G[A]} v\}.$$

Note that $v\in H_v$.

\begin{lemma}
  \label{l:abc}
  $\mathcal H_A$ is a laminar hypergraph and $G[A\cup B]$ is
  isomorphic to its line graph (in particular, $G[A\cup B]$ is a
  quasi-threshold graph and therefore a chordal graph).  A similar
  statements holds for $G[A'\cup B']$.
\end{lemma}

\begin{proof}
  By construction, every vertex of $A \cup B$ corresponds to a
  hyperedge of $\mathcal{H_A}$.  We have to check that the ends of
  every edge of $G[A\cup B]$ correspond to hyperedges of $\mathcal H_A$
  that are included one in the other, and that the ends of every
  non-edge correspond to a pair of disjoint hyperedges. This will
  prove that $G[A\cup B]$ is isomorphic to the line graph of
  $\mathcal H_A$ and that $\mathcal H_A$ is laminar.  Let us check all
  the cases.

  For $x, y\in B$, since $\mathcal H_B$ is laminar and $G[B]$ is
  isomorphic to its line graph, we have nothing to prove.

  Let $u, v \in A$. By Theorem~\ref{th:threshold}, we may assume up to
  symmetry that $u\geq_{G[A]} v$.  If $uv\in E(G)$, then clearly
  $H_v\subseteq H_u$. Suppose $uv\notin E(G)$, and let $t$ be a vertex
  of $A$ such that $t\leq_{G[A]} v$.  So, $t\leq_{G[A]} u$.  If $tv\in E(G)$,
  then $uv\in E(G)$, a contradiction. So, $tv\notin E(G)$ and
  $H_v= \{v\}$.  Since $v\notin N[u]$, we have that
  $H_u\cap H_v = \emptyset$.

  Consider finally vertices $u \in A$ and $x\in B$.  Suppose first
  that $ux\in E(G)$.  By property~\eqref{t:makeIt} of templates, we have
  that $u \in N_A[H_x]$.  If $u\in H_x$, then $H_u\subseteq H_x$ by
  Lemma~\ref{l:nxcc} (specifically, we use $N_A(H_x)>_{G[A]} H_x$ to
  conclude that $N_A(H_x) \cap H_u = \emptyset$, and then since
  $H_u \subseteq N_A[H_x]$ it follows that $H_u \subseteq H_x$).  If
  $u\in N_A(H_x)$, then by Lemma~\ref{l:nxcc} (again also using that
  $N_A(H_x) >_{G[A]} H_x$), $H_x\subseteq H_u$.  So, an edge indeed yields
  an inclusion of the corresponding hyperedges.

  Suppose now that $ux\notin E(G)$. So, $u\notin N[H_x]$. Since $u$ is not in
  $H_x$ and has no neighbor in $H_x$, it follows that $H_u$ is disjoint from
  $H_x$.

  So $\mathcal H_A$ is a laminar hypergraph and $G[A\cup B]$ is
  isomorphic to its line graph. It follows from
  Theorem~\ref{th:laminar} that $G[A\cup B]$ is a quasi-threshold
  graph and therefore a chordal graph.
\end{proof}

\begin{lemma}
  \label{l:HinTemplate}
  Every hole of $G$ is formed by two principal paths of $G$ and a
  single vertex of $A\cup B\cup A' \cup B'$ that does not belong to
  these principal paths (it therefore has length $2\ell+1$).
\end{lemma}

\begin{proof}
  By Lemma~\ref{l:abc}, a hole $C$ of $G$ cannot contain only vertices of
  $A\cup B$, and similarly, it cannot contain only vertices of $A'\cup B'$. So
  it must contain vertices of some principal path, and also of a second
  principal path. In fact, $C$ must go through exactly two principal paths,
  since $G[A]$ is isomorphic to the complement of $G[A']$, if three
  paths are involved, there would be a vertex of $C$ with three neighbors in
  $C$, a contradiction.

  Since $G[A]$ is isomorphic to the complement of $G[A']$, up to a
  symmetry, for some nonadjacent vertices $u,v \in A$, the hole $C$ is
  made of a path $P = u\dots v$ with interior in $I\cup A'$ (whose
  length is $2\ell -1$) and a path $Q = u\dots v$ of $G[A\cup B]$.  By
  Lemma~\ref{l:abc}, $Q$ has length at most~2 (because a
  quasi-threshold graph is $P_4$-free), and since $uv\notin E(G)$, it
  has length~2.  So, $C$ has length $2\ell+1$ as claimed.
\end{proof}

\subsection{Connecting vertices of a template}
\label{subs:connectF0}

\begin{lemma}\label{lt:cheminsentrepatatesBBP}
  If $x\in B$ and $y\in B'$, then there exists in $G$ two paths $P$
  and $Q$ of length $\ell +1$ from $x$ to $y$ such that $P$ (resp.\
  $Q$) contains a principal path $P_0$ (resp.\ $Q_0$), and
  $P_0\neq Q_0$.
\end{lemma}

\begin{proof}
  We set $X=\{i\in \{1, \dots, k\} : v_i\in H_x\}$ and
  $Y=\{i\in \{1, \dots, k\} : v'_i\in H'_{y}\}$.  So, $X$ and $Y$ are
  hyperedges of $\mathcal H$ and since $\mathcal H$ is laminar, either
  $X\subseteq Y$, $Y\subseteq X$ or $X\cap Y=\emptyset$.
  
  If $X\subseteq Y$, then let $i, j$ be distinct members of $X$ (and
  therefore of $Y$). The paths $xv_iP_iv'_iy$ and $xv_jP_jv'_jy$ are
  the paths we are looking for.  The proof is similar when
  $Y\subseteq X$.

  If $X\cap Y = \emptyset$, then let $i, j, q, r$ be distinct integers
  such that $i, j\in X$ and $q, r\in Y$.  Since $G[A]$ is isomorphic
  to the complement of $G[A']$, we may assume up to symmetry that
  $v_iv_q\in E(G)$.  So, $v'_iv'_q\notin E(G)$.  Since $H'_y$ is a
  module of $G[A']$, $v'_iv'_r\notin E(G)$.  It follows that
  $v_iv_r\in E(G)$.  So, $v_r, v_q\in N_A(H_x)$.  Hence, by
  property~\ref{t:makeIt} of templates, $xv_r, xv_q\in E(G)$.  It
  follows that $xv_qP_qv'_qy$ and $xv_rP_rv'_ry$ are the two paths we
  are looking for.
\end{proof}

\begin{lemma}\label{lt:cheminsentrepatates}
  If $x\in A\cup B$ and $y\in A'\cup B'$, then there exists in $G$ a
  path $P$ of length  $\ell-1$, $\ell$ or $\ell +1$ from $x$ to
  $y$ that contains a principal path.

  More specifically:
  \begin{itemize}
  \item
    If $x\in A$ and $y\in A'$, then $P$ has length $\ell-1$ or
    $\ell$.
  \item If $x\in A$ and $y\in B'$, or if $x\in B$ and $y\in A'$,
    then $P$ has length $\ell$ or $\ell+1$.
  \item If $x\in B$ and $y\in B'$, then $P$ has length
    $\ell+1$.
  \end{itemize}
\end{lemma}

\begin{proof}
  Suppose first that $x\in A$, say $x=v_i$.  If $y\in A'$, then set
  $y=v'_j$. If $i=j$, then $P_i$ has length $\ell-1$. If $i\neq j$,
  then one of $v_iv_jP_jv'_j$ or $v_iP_iv'_iv_j'$ is a path of length
  $\ell$. If $y\in B'$, then one of $v_iP_iv'_iy$ or $v_iP_iv'_iw'y$
  is the path we are looking for.  The proof is similar when
  $y\in A'$.

  We may therefore assume that $x\in B$ and $y\in B'$. So one of the
  two paths obtained in Lemma~\ref{lt:cheminsentrepatatesBBP} can be
  chosen.
\end{proof}

\subsection{Odd pretemplates}

Checking that a graph is an odd $\ell$-template is tedious. We now
introduce a simpler notion that is in some sense equivalent.  For
every integer $\ell \geq 3$, an \emph{odd $\ell$-pretemplate} is a
graph $G$ whose vertex-set can be partitioned into five sets $A$, $B$,
$A'$, $B'$ and $I$ with the following properties.

\begin{enumerate}
\item\label{a:antiC} $N(B)\subseteq A$ and $N(A\cup B) \subseteq I$. 
\item\label{a:antiCP} $N(B')\subseteq A'$ and $N(A'\cup B') \subseteq I$. 
\item\label{a:cardinality} $|A|=|A'|=k \geq 3$, $A = \{v_1, \dots, v_k\}$ and
  $A' = \{v'_1, \dots, v'_k\}$.
\item\label{a:path} For every $i\in \{1, \dots, k\}$, there exists a
  unique path $P_i$ from $v_i$ to $v'_i$ whose interior is in $I$.
\item \label{a:interI} Every vertex in $I$ has degree 2 and lies on a path from $v_i$
  to $v'_i$ for some $i\in \{1, \dots, k\}$.
\item\label{a:sp} All paths $P_1$, \dots, $P_k$ have length $\ell-1$.
\item\label{a:cn} $G[A\cup B]$ and $G[A'\cup B']$ are both connected
  graphs. 
\item\label{a:tn} Every vertex of $B$ is in the interior of a path of
  $G[A\cup B]$ with both ends in $A$.
\item\label{a:tnP} Every vertex of $B'$ is in the interior of a path of
  $G[A'\cup B']$ with both ends in $A'$.
\end{enumerate}

We then say that $(A, B, A', B', I)$ is an \emph{$\ell$-pretemplate
  partition} of~$G$.  Note that templates are defined for all integers
$\ell \geq 2$, while pretemplates are defined only when $\ell\geq 3$.
In fact we do not need odd 2-templates, we defined them for possible
later use.

It is easy to check that when $\ell\geq 3$, the five first elements of
every $\ell$-partition of $G$ is an $\ell$-pretemplate partition.  The
condition on the connectivity of $G[A\cup B]$ and $G[A'\cup B']$
follows Lemma~\ref{l:ABuniv}.  The condition~\eqref{a:tn} follows from
the fact for every $x\in B$, $H_x$ contains two non-adjacent vertices,
so a vertex $x\in B$ lies on a path of length~2 with ends in $A$ and
condition~\eqref{a:tnP} holds similarly. Conversely, we prove the
following lemma (it is important to note that $\ell\geq 3$).
 
\begin{lemma}
  \label{th:ptist}
  Let $\ell\geq 3$ be an integer. If $G\in \mathcal C_{2\ell +1}$ is
  an odd $\ell$-pretemplate, then $G$ is an odd
  $\ell$-template. Moreover, for every odd $\ell$-pretemplate
  partition $(A, B, A', B', I) $ of $G$, there exist $w$ and $w'$ in
  $V(G)$ such that $(A, B, A', B', I, w, w')$ is an $\ell$-partition
  of $G$.
\end{lemma}

\begin{proof}
  Let $(A, B, A', B', I)$ be an $\ell$-pretemplate partition of
  $G$.  We first study the structure of $G[A]$ and $G[A']$.

\begin{claim}
  \label{l:comp}
  For all distinct $i, j \in \{1, \dots, k\}$, $v_iv_j\in E(G)$ if and
  only if $v'_iv'_j\notin E(G)$. In particular, $G[A]$ is isomorphic
  to the complement of $G[A']$.
\end{claim}

\begin{proofclaim}
  If $v_iv_j, v'_iv'_j \in E(G)$, then $P_i$ and $P_j$ form a hole of
  even length, a contradiction. If $v_iv_j, v'_iv'_j \notin E(G)$,
  then $P_i$, $P_j$, a path from $v_i$ to $v_j$ in $ G[A \cup B] $ and a
  path from $v'_i$ to $v'_j$ in $G[A' \cup B']$ form a hole of length at
  least $2\ell +2$, a contradiction.
\end{proofclaim}

\begin{claim}
  \label{c:length2}
  Every path of $G[A\cup B]$ with both ends in $A$ is of length at
  most~2. 
\end{claim}

\begin{proofclaim}
  Let $P=v_i\dots v_j$ be a path of $G[A\cup B]$ with both ends in
  $A$. If $P$ has length at least~3, then by~(\ref{l:comp}), paths $P$,
  $P_i$ and $P_j$ form a hole of length at least~$2\ell +2$, a
  contradiction.
\end{proofclaim}

\begin{claim}
  \label{l:A}
  $G[A]$ is a threshold graph.
\end{claim}

\begin{proofclaim}
  $G[A]$ is obviously $C_4$-free.  Since the complement of $C_4$ is
  $2K_2$ and since $G[A']$ is also $C_4$-free, it follows
  by~(\ref{l:comp}) that $G[A]$ is $2K_2$-free.  By~(\ref{c:length2}),
  $G[A]$ is $P_4$-free.  So $G[A]$ is ($P_4$, $C_4$, $2K_2$)-free and
  is therefore a threshold graph.
\end{proofclaim}

We now study the structure of $G[B]$ and its relation with $G[A]$.

\begin{claim}
  \label{l:twoInA}
  For every vertex $x \in B$, $G[N_A(x)]$ has a unique anticonnected
  component of size at least~2.
\end{claim}

\begin{proofclaim}
  By the definition of odd pretemplates, $x$ is in the interior of a path
  $P= v_i\dots v_j$ of $G[A\cup B]$ with both ends in $A$.
  By~(\ref{c:length2}), $P$ has length~2, so $x$ is adjacent to $v_i$
  and $v_j$.  Hence $G[N_A(x)]$ has an anticonnected component of size
  at least~2.  It is unique, for otherwise $G[A]$ contains a $C_4$.
\end{proofclaim}

For all $x\in B$, we define $H_x$ to be the anticonnected component
of $G[N_A(x)]$ of size at least~2 whose existence follows
from~(\ref{l:twoInA}).

\begin{claim}
  \label{l:PR}
  For every $x$ in $B$, $H_x$ is a module of $G[A]$.  
\end{claim}

\begin{proofclaim}
  Otherwise, since $H_x$ is anticonnected and is not a module, there
  exists $v_h \in A \sm H_x$ and non-adjacent $v_i, v_j\in H_x$ such
  that $v_iv_h\in E(G)$ and $v_jv_h\notin E(G)$.  Note that
  $xv_h\notin E(G)$ because otherwise, $v_h$ would be in $H_x$.
  Hence, $v_i$, $x$, $P_j$ and $P_h$ form a hole of length
  $2\ell + 2$, a contradiction.
\end{proofclaim}

\begin{claim}
  \label{l:edgexy}
  If $xy$ is an edge of $G[B]$, then $H_x \subseteq H_y$ or
  $H_y\subseteq H_x$.  
\end{claim}

\begin{proofclaim}
  Up to symmetry, we may assume that $N_A(x) \subseteq N_A(y)$, for
  otherwise vertices $v_i\in N_A(x)\sm N_A(y)$ and  $v_j\in
  N_A(y)\sm N_A(x)$ either form a $C_4$ with $x$ and $y$ or a hole of
  length $2\ell +2$ with $P_i$ and $P_j$.

  By~(\ref{l:twoInA}), $G[N_A(y)]$ has only one anticonnected component of size
  at least~2, namely $H_y$.  Since $H_x$ is anticonnected, has size
  at least~2 and is included in $N_A(y)$, it must be included in
  $H_y$.
 \end{proofclaim}

 \begin{claim}
   \label{l:nonedgexy}
   If $x$ and $y$ are non-adjacent vertices of $B$,\,then $H_x$ and
   $H_y$ are disjoint.
\end{claim}

\begin{proofclaim}
  On the contrary, suppose that $x$ and $y$ are nonadjacent vertices of $B$ but
  there exists a vertex $v \in H_x \cap H_y$.  Since $H_x$ is anticonnected and
  of size at least~2, there exists $v_i\in H_x$ non-adjacent to $v$. Note that
  $v_iy\notin E(G)$, for otherwise $x$, $y$, $v_i$ and $v$ form a $C_4$.
  Similarly, there exists a vertex $v_j\in H_y$ that is non-adjacent to $v$ and
  to $x$.  If $v_iv_j\in E(G)$, then $\{x, y, v, v_i, v_j\}$ induces a $C_5$, a
  contradiction.  Otherwise, $P_i$, $P_j$, $x$, $y$ and $v$ form a hole of
  length $2\ell+3$, a contradiction.
\end{proofclaim}

We are now ready to define the hypergraph $\mathcal H$.  For every
$x\in B$, we defined a set $H_x\subseteq A$.  We may define similarly
a set $H'_x\subseteq A'$ for every $x\in B'$.  From~\eqref{l:edgexy}
and~\eqref{l:nonedgexy}, the sets $H_x$ for $x\in B$ form a laminar
hypergraph $\mathcal H_B$ (with vertex set $A$).  Symetrically, the
sets $H'_x$ for $x\in B'$ form a laminar hypergraph $\mathcal H_{B'}$
(with vertex set $A'$).  Let $\mathcal H$ be the hypergraph whose
vertex set is $\{1, \dots, k\}$ and such that
$H\subseteq \{1, \dots, k\}$ is a hyperedge of $\mathcal H$ if and
only if $H=\{i : v_i\in H_x\}$ for some $x\in B$ or
$H = \{i : v'_i\in H'_x\}$ for some $x\in B'$.

\begin{claim}
  \label{c:global}
  The hypergraph $\mathcal H$ is laminar.
\end{claim}

\begin{proofclaim}
  If $\mathcal H$ is not laminar, then there exist
  $X, Y\in E(\mathcal H)$ such that $X\sm Y$, $Y \sm X$ and
  $X\cap Y$ are all non-empty.  Since $\mathcal H_B$ and
  $\mathcal H_{B'}$ are both laminar, there exists $x\in B$ such that
  $H_x = \{v_i : i\in X\}$ and $y\in B'$ such that
  $H'_y = \{v'_i : i\in Y\}$.

  We set $H_y = \{v_i: i\in Y\}$.  Note that $H_x\sm H_y$,
  $H_y\sm H_x$ and $H_x\cap H_y$ are all non-empty. Also, because of
  the properties of $H'_y$ and by \eqref{l:comp}, $G[H_y]$ is connected (because $G[H'_y]$
  is anticonnected) and $H_y$ is a module of $G[A]$.

  Since $G[H_x]$ is anticonnected, there exist non-adjacent vertices
  $u\in H_x\sm H_y$ and $v\in H_x\cap H_y$.  Since $G[H_y]$ is
  connected, there exists a path from $v$ to $t \in H_y\sm H_x$ and we may assume that $vt$ is an edge.  
  Since $H_y$ is a module of $G[A]$, $ut\notin E(G)$.  So, $t$ is adjacent
  to $v$ and non-adjacent to $u$. This contradicts $H_x$ being a
  module of $G[A]$.  
\end{proofclaim}

We may now finish the proof of Lemma~\ref{th:ptist}. We show how $G$
can be built by the process described in the definition of odd
templates. We start by setting $V(J)=\{1, \dots, k\}$, and by making
$i$ adjacent to $j$ in $J$ if and only if $v_iv_j\in
E(G)$. By~(\ref{l:A}), $J$ is a
threshold graph as required. Clearly condition~\eqref{tempd} of odd $\ell$-templates holds, 
the paths linking $A$ to $A'$ are as in condition~\eqref{i:linkP} of odd $\ell$-templates and condition~\eqref{tempe7} of odd $\ell$-templates holds. 
By~(\ref{l:comp}), conditions~\eqref{i:BJ} and~\eqref{i:BJP}
of odd $\ell$-templates hold. 
We then consider the hypergraph $\mathcal H$
defined above.  It is laminar by~\eqref{c:global}.  By~\eqref{l:PR},
condition~\eqref{a:module} of templates is satisfied.

By definition of $H_x$, for every $x$ in $B$,
$N_A(x) \subseteq N_A[H_x]$.  Suppose that there exists
$u\in N_A[H_x]\sm N_A(x)$.  Since by~(\ref{l:PR}) $H_x$ is a module, 
it follows from Lemma~\ref{l:nxcc} that $u$ is complete to $H_x$, so
$x$ and $u$ together with two non-adjacent vertices from $H_x$ induce
a $C_4$, a contradiction. Hence, $N_A(x) = N_A[H_x]$ and
condition~\eqref{a:makeit} of odd templates is satisfied.

By~(\ref{l:edgexy}) and~\eqref{l:nonedgexy}, condition~\eqref{a:last} of templates
is satisfied. 

By symmetry and by~\eqref{l:comp}, conditions~\eqref{tempe4} and~\eqref{tempe6} of templates are satisfied. Therefore condition~\eqref{tempe} of templates is satisfied.

To conclude the proof, let us check condition~\eqref{a:universal} of
templates.  By~\eqref{l:comp}, \eqref{l:A} and Theorem~\ref{th:threshold}, 
up to symmetry, we may assume that $G[A]$ contains an isolated
vertex $v_i$.  Since $G[A\cup B]$ is connected and $|A|\geq 3$ by the
definition of odd pretemplates, there exists a path $P$ in
$G[A\cup B]$ from $v_i$ to a vertex $u\in A\sm\{v_i\}$.
By~(\ref{c:length2}) and since $v_i$ has no neighbor in $A$, we have
that $P=uyv_i$ where $y\in B$.  So, $H_y$ contains $v_i$. We may
therefore consider the hyperedge $W$ of $\mathcal H$ that contains $i$
and that is inclusion wise maximal w.r.t.\ this property.  If there
exists $j\in \{1, \dots, k\} \sm W$, since
$v_jv_i\notin E(G)$, we deduce as above that $\mathcal H$ has a
hyperedge $Z$ that contains $i$ and $j$.  Because of $j$,
$Z\subseteq W$ is impossible; because of $i$, $W\cap Z=\emptyset$ is
impossible; and because of the maximality of $W$, $W\subsetneq Z$ is
impossible.  Hence, $W$ and $Z$ contradict $\mathcal H$ being
laminar.  This proves that $W = \{1, \dots, k\}$, as claimed in
condition~\eqref{a:universal} of templates.

Hence, $G[A\cup B]$ has universal vertex $w$. Also, $G[A'\cup B']$ has
a universal vertex $w'$ (we may apply Lemma~\ref{l:ABuniv} since we
now know that $G$ is an odd $\ell$-template).  So,
$(A, B, A', B', I, w, w')$ is an $\ell$-partition of $G$.
\end{proof}

\subsection{Twins and proper partitions}

Two distinct vertices $x$ and $y$ in a graph are \emph{twins} if
$N[x] = N[y]$ (in particular, $x$ and $y$ are adjacent).  A graph is
\emph{twinless} if it contains no twins.

\begin{lemma}
  \label{l:wTwins}
  Let $(A, B, A', B', I, w, w')$ be an $\ell$-partition of an odd
  $\ell$-template $G$.  Two vertices $x$ and $y$ of $G$ are twins if
  and only if $x, y\in B$ and $H_x=H_y$, or $x, y\in B'$ and
  $H'_x = H'_y$.
\end{lemma}

\begin{proof}
  If $x, y\in B$ and $H_x=H_y$, or $x, y\in B'$ and $H'_x=H'_y$, then
  $x$ and $y$ are obviously twins. 
  
  We claim that for all $x\in A\cup I\cup A'$, there exist two
  vertices $a,b\in N_G(x)$ such that $N[a]\cap N[b]=\{x\}$. If
  $x\in I$ choose $a$ and $b$ to be the only two neighbors of $x$. If
  $x\in A$, then set $a=w$ when $x\neq w$, and choose for $a$ any
  vertex of $A\sm \{x\}$ when $x=w$. Choose for
  $b$ the neighbor of $x$ in $I$. In both cases, by condition (e7) of
  templates, $N_G[a]\cap N_G[b]=\{x\}$.  The proof is similar when
  $x\in A'$.  So, $x$ has no twin in $G$.
\end{proof}

An $\ell$-partition $(A, B, A', B', I, w, w')$ of an odd
$\ell$-template $G$ is \emph{proper} if one of $G[A]$ or $G[A']$
contains at least two isolated vertices.

\begin{lemma}
  \label{l:TProper}
  For all integers $\ell \geq 3$, every twinless odd $\ell$-template
  $G$ admits a proper $\ell$-partition.
\end{lemma}

\begin{proof}
  Let $(A, B, A', B', I, w, w')$ be an $\ell$-partition of $G$ such
  that the number $M$ of isolated vertices of $G[A]$ is maximum.  We
  suppose that $v_1, \dots, v_k$ is a domination ordering of $G[A]$.
  
  By Theorem~\ref{th:threshold} and since we may swap $A, B, w$ and
  $A', B', w'$, by the maximality of $M$, $v_1$ is an isolated vertex
  of $G[A]$.  It follows that $w\in B$.  By definition of templates, $v'_1$ is a
  universal vertex of $G[A'\cup B']$.  Suppose for a contradiction that $v_{2}$
  is not isolated in $G[A]$.  So, $M=1$.

  Let $H_x$ be any hyperedge of $\mathcal H_B$ containing $v_1$.  By
  Lemma~\ref{l:ModuleIso}, since $H_x$ contains a non-isolated vertex
  of $G[A]$, it contains all of them. So, $H_x=A$.  Hence, $N[x]= N[w]$,
  so $x=w$ since $G$ is twinless.  This proves that
  $N(v_1) = \{w, v_1^+\}$ where $v_1^+$ is the neighbor of $v_1$ in
  $I$. Let $v'^+_1$ be the neighbor of $v'_1$ in $I$.  We now describe
  a new partition of the vertices of $G$.  We set:

  \begin{itemize}
  \item $A_1 = \{w, v_2, \dots, v_{k}\}$,
  \item $B_1 = B\sm \{w\}$, 
  \item $A'_1 = \{v'^+_1, v'_2, \dots, v'_{k}\}$,
  \item $B'_1 = B' \cup \{v'_1\}$ and
  \item $I_1 = \{v_1\} \cup I \sm \{v'^+_1\}$.
  \end{itemize}

  All conditions of the definition of a pretemplate are easily checked
  to be satisfied by $(A_1, B_1, A'_1, B'_1, I_1)$.  By
  Lemma~\ref{l:HinTemplate}, every hole in $G$ has length $2\ell +1$.
  We may therefore apply Lemma~\ref{th:ptist} to prove that
  $(A_1, B_1, A'_1, B'_1, I_1, w, v'_1)$ is an $\ell$-partition of
  $G$.  So, $(A'_1, B'_1, A_1, B_1, I_1, v'_1, w)$ contradicts that
  maximality of $M$ since $G[A'_1]$ has two isolated vertices, namely
  $v'^+_1$ and $v'_k$ (note that since $M=1$, it follows by Theorem~\ref{th:threshold} that $G[A \sm \{v_1\}]$ has a universal vertex ; in particular $v_k$ is a universal vertex of $G[A \sm \{v_1\}]$ and hence $v'_k$ is an isolated vertex of $G[A'_1]$).
\end{proof}

\begin{lemma}
  \label{l:wProper}
  Every proper $\ell$-partition $(A, B, A', B', I, w, w')$ of a
  twinless odd $\ell$-template satisfies one of the following:
\begin{itemize}
\item $w\in B$, $w$ is the unique universal vertex of $G[A\cup B]$,
  $G[A]$ contains at least two isolated vertices, $H_{w} = A$,
  $w'\in A'$ and $A'\sm \{w'\}$ contains at least one universal vertex
  of $G[A']$.
\item $w\in A$, $A\sm \{w\}$ contains at least one universal vertex of
  $G[A]$, $w'\in B'$, $w'$ is the unique universal vertex of
  $G[A'\cup B']$, $G[A']$ contains at least two isolated vertices and
  $H'_{w'} = A'$.
\end{itemize}
\end{lemma}

\begin{proof}
  Since $(A, B, A', B', I, w, w')$ is a proper $\ell$-partition, up to symmetry, we may assume that $G[A]$
  contains two isolated vertices.  So, $w\in B$.  By definition of
  $\ell$-partitions, it follows that $w'\in A'$.  Since $G$ is
  twinless and $G[A]$ contains isolated vertices, $w$ is the unique universal vertex of $G[A\cup B]$. By
  Lemma~\ref{l:recoverHx}, $H_{w} = A$.  Also $A'\sm \{w'\}$ contains
  at least one universal vertex of $G[A']$ since $G[A]$ contains two
  isolated vertices.
\end{proof}

We do not use the following lemma formally, but it illustrates a key
property of proper partitions.  In non-proper partitions, there may
exist vertices in $A$ that have degree~2 and have one neighbor in $A$
and one in $I$. These are hard to think of, because they yield edges
with both ends in $A$ that can be ``blown up'' into a general half
graph as we will see in the next section. The next lemma states that
this situation does not occur with proper partitions.

\begin{lemma}
  \label{l:properDeg3}
  Suppose $\ell \geq 3$ and $G$ is an odd $\ell$-template with a
  proper $\ell$-partition $(A, B, A', B', I, w, w')$.  If a vertex $v$ in
  $A \cup B \cup A' \cup B'$ has degree~2 (in $G$), then
  $v\in A\cup A'$ and $v$ is adjacent to a vertex of $B\cup B'$ and
  has its other neighbor in $I$.
\end{lemma}

\begin{proof}
  Up to symmetry, suppose that $G[A]$ contains at least two isolated
  vertices. So, $w\in B$.  Consider a vertex $v \in A\cup B \cup A' \cup B'$ that is of degree $2$ in $G$.  
  By Lemma~\ref{l:Bdeg3}, $v\notin B\cup B'$.  Since $G[A']$ has two universal vertices, every
  vertex in $A'$ has degree at least~3, so $v\in A$, and $v$ is
  adjacent to $w\in B$ and to some vertex in $I$ as claimed.
\end{proof}

\section{Blowup}

Our goal in this section is to see how a bigger graph can be obtained
from a template $G$ by turning every vertex into a non empty clique. This will
be called \emph{blowing up $G$}.  In the blowup operation,
non-adjacent vertices yield cliques that are anticomplete to each
other. Adjacent vertices $u$ and $v$ yield cliques that are complete
to each other in some situations (when $uv$ is a so-called \emph{solid
  edge} of the template), but in some other situations, they may yield
pairs of cliques that induce a more general half graph, like when a
ring is obtained from ``blowing up'' a chordless cycle. This happens
when $uv$ is a so-called \emph{flat} or \emph{optional edge} of the
template.  We now define all this formally.

Throughout all this section, $\ell\geq 3$ is an integer, $G$ is a an
odd $\ell$-template with a fixed an $\ell$-partition
$(A, B, A', B', I, w, w')$.

\subsection{Flat, optional and solid edges} \label{def_edges_temp}

An edge of $G$ is \emph{flat} if at least one of its end is in $I$.
An edge of $G$ is \emph{optional} if one end is a vertex $x\in B$
(resp.\ $x\in B'$) and the other end is a vertex $u\in H_x$ that is an
isolated vertex of $G[H_x]$ (resp.\ a vertex $u\in H'_x$ that is an
isolated vertex of $G[H'_x]$).  An edge that is neither flat nor
optional is \emph{solid}.  See Figure~\ref{f:template} where solid
edges are represented by solid lines, flat edges by doted lines and
optional edges by dashed lines.

Observe that the status of an edge depends on the $\ell$-partition of
the odd $\ell$-template. See Figure~\ref{f:pyr}, where the same
template is represented with two different $\ell$-partitions.  Recall
that throughout this section, the $\ell$-partition is fixed, and so is
the status of the edges. 

\begin{lemma}
  \label{c:FXsubY}
  If $ux$ is an optional edge of $G$ with $u\in A$ and $x\in B$, then
  $N_{A}(H_x) = N_{A}(u)$.  Moreover, if $y\in B\sm \{x\}$ and
  $yu \in E(G)$, then $H_x\subseteq H_y$ or $H_y\subseteq H_x$ (in
  particular, $xy\in E(G)$).
\end{lemma}

\begin{proof}
  Since $H_x$ is a module of $G[A]$ and $u$ is isolated in $H_x$, we have
  $N_{A}(H_x) = N_{A}(u)$.  If the second conclusion fails, then since
  $\mathcal H_{B}$ is laminar, $H_y\cap H_x = \emptyset$.  So
  $xy \notin E(G)$.  Since $yu\in E(G)$ and $u\notin H_y$, we have
  $u\in N_A(H_y)$, so $u$ is complete to $H_y$ since $H_y$ is a
  module.  So $H_y\subseteq N_A(u) = N_A(H_x)$, which is a clique by
  Lemma~\ref{l:nxcc}.  This contradicts $H_y$ being anticonnected.
\end{proof}

A clique of $G$ is \emph{solid} if all its edges are solid.

\begin{lemma}
  \label{l:NuSolid}
  If $ux$ is an optional edge of $G$ such that $u\in A$ and $x\in B$,
  then $N_{A\cup B}(u)$ is a solid clique of $G$.
\end{lemma}

\begin{proof}
  By Lemma~\ref{c:FXsubY}, $N_A(H_x)=N_A(u)$.  By Lemma~\ref{l:nxcc},
  $N_A(H_x)=N_A(u)$ is a clique. It is solid because edges with both
  ends in $A$ are solid. Hence $N_A(u)$ is a solid clique.
 
    By Lemma~\ref{c:FXsubY} all vertices from $N_B(u)$ are adjacent
    since they correspond to hyperedges of $\mathcal H_B$ that are included in
    each other.  Therefore, $N_B(u)$ is a clique and it is solid
    because edges with both ends in $B$ are solid. Hence $N_B(u)$ is a solid clique.

    It remains to prove that $N_A(u)$ is complete to $N_B(u)$ and that
    all edges between these two sets are solid. So let $y\in B$ and
    $v\in A$ be two neighbors of $u$.  Note that $v\notin H_x$ and
    possibly $y=x$. If $u\in H_y$, then $vy$ is an edge because
    $v\in N_A(u)$ (and so $v \in N_A[H_y]$), and it is a solid edge
    because $v$ is not an isolated vertex of $H_y$.  If $u\notin H_y$,
    then by Lemma~\ref{c:FXsubY}, $H_y \subseteq H_x$.  So,
    $u\in N_A(H_y)$ since $uy\in E(G)$, and this contradicts $u$ being
    isolated in $H_x$.
\end{proof}

\begin{lemma}
  \label{l:removeFr}
  Let $C$ be a cycle of $G$ of length at least 4 with no solid chord. 
  If $C$ is not a hole then there exist three consecutive vertices
  $x$, $y$, $u$ in $C$ such that: 
  
  - $u\in A$, $x, y\in B$,
  $\{u\} \subseteq H_y \subseteq H_x$ and $u$ is an isolated vertex of
  $H_x$,  or 
  
  - $u\in A'$, $x, y\in B'$,
$\{u\} \subseteq H'_y \subseteq H'_x$ and $u$ is an isolated vertex of
  $H'_x.$ 
  In particular $ux$ is an optional edge of $G$ and a chord of
  $C$.
\end{lemma}

\begin{proof}
 
  We may assume that $C$ has a chord $e$ for otherwise it is a hole.
  This chord cannot be a flat edge of $G$ because a flat edge contains
  a vertex of $I$, so a vertex of degree~2, and it therefore cannot be
  a chord of any cycle.  Hence, $e$ is an optional edge of $G$.  So,
  up to symmetry, we may assume that $e=ux$ with $u\in A$ and
  $x\in B$.  By definition of optional edges, $u$ is an isolated
  vertex of $G[H_x]$.

  Let $u'$ and $y$ be the two neighbors of $u$ along $C$.  If
  $u', y \in A\cup B$, then by Lemma~\ref{l:NuSolid}, $u'y$ is a
  solid chord of $C$, a contradiction.  So, up to symmetry,
  $y\in A\cup B$ and $u'\in I$. 

  Suppose first that $y\in A$.  Since $uy\in E(G)$ and $u$ is isolated
  in $H_x$, we have that $y\in N_A(H_x)$.  If follows that
  $xy\in E(G)$, and moreover, $xy$ is a solid edge since
  $y\notin H_x$.  Since $x$ and $y$ are both in $C$ and $C$ has no
  solid chord, $C$ visits consecutively $u'$, $u$, $y$ and $x$.  Let
  $x'$ be the neighbor of $x$ in $C\sm y$.  If $x'\in B$ then
  $H_x\cap H_{x'}\neq \emptyset$, and since $y$ is complete to $H_x$,
  $y$ has a neighbor in $H_{x'}$. It follows that $yx'$ is an edge of
  $G$, the edge $yx'$ is solid, and is therefore a solid chord of $C$,
  a contradiction.  Hence $x'\in A,$ and so since $xx'$ is an edge, $x' \in N_A[H_x]$.  
  If $x'\in H_x$, then $yx'$ is a
  solid chord of $C$, and if $x'\in N(H_x)$, then (since $H_x$ is a module of $G[A]$)
  $x'u$ is a solid
  chord of $C$, in each case a contradiction.

  Suppose now that $y\in B$. If $H_y \subseteq H_x$, then $xy$ is an
  edge that is solid and hence is an edge of $C$, so the conclusion of
  the lemma holds. So we may assume by Lemma~\ref{c:FXsubY} that
  $H_x \subseteq H_y$.  In particular, $xy$ is an edge, and since it
  is solid, $u$, $y$ and $x$ are consecutive along $C$.  Let $v$ be
  the neighbor of $x$ in $C\sm y$.  If $v\in B$, then
  $H_v\cap H_x\neq \emptyset$, so $H_v\cap H_y\neq \emptyset$, showing
  that $yv$ is a solid chord of $G$, a contradiction.  Hence,
  $v\in A$.  We have $uv\notin E(G)$ for otherwise $uv$ would be a
  solid chord of $G$.  Hence, $v\in H_x$ since $vx\in E(G)$ and $H_x$ is a module of $G[A]$. So,
  $v\in H_y$ (and hence $vy \in E(G)$) and $v$ is an isolated vertex
  of $H_y$, for otherwise $vy$ would be a solid chord of $G$.

  Now, we have three consecutive vertices $y$, $x$, $v$ in $C$ such
  that: $v\in A$, $y, x\in B$, $\{v\} \subseteq H_x \subseteq H_y$ and
  $v$ is an isolated vertex of $H_y$.  So, the conclusion of the lemma
  is satisfied again with these three vertices.
\end{proof}

\subsection{Blowups and holes} \label{def_blowup}

Let $G$ be a twinless odd $\ell$-template with an $\ell$-partition
$(A, B, A', B', I, w, w')$ A \emph{blowup} of $G$ is any graph $G^*$
that satisfies the following:

\begin{enumerate}
\item\label{i:buK} For every vertex $u$ of $G$ there is a clique
  $K_u$ in $G^*$ on $k_u\geq 1$ vertices $u_1, \dots, u_{k_u}$ such that
  $u_{k_u} = u$ ; for distinct vertices $u,v$ of $G$, $K_u \cap K_v = \emptyset$ and $V(G^*) = \bigcup_{u\in V(G)} K_u,$ so  $V(G) \subseteq V(G^*)$.

\item\label{i:KKhalf} For all vertices $u \in V(G)$ and all integers
  $1\leq i \leq j \leq k_u$,  in $G^*$ $N[u_i] \subseteq N[u_j]$ (in
  particular, for all $u, v \in V(G)$, $G^*[K_u\cup K_v]$ is a half
  graph).

\item\label{i:bKKanti} If $u$ and $v$ are non-adjacent vertices of
  $G$, then $K_u$ is anticomplete to $K_v$ (in particular
  $uv \notin E(G^*)$).

\item\label{i:Kcomp} If $uv$ is a solid edge of $G$, then $K_u$ is
  complete to $K_v$ (in particular $uv \in E(G^*)$).
  
\item\label{i:flat} If $uv$ is a flat edge of $G$, then $u$ is complete to
  $K_v$ and $v$ is complete to $K_u$ (in particular $uv\in E(G^*)$).

\item\label{i:optAB} If $ux$ is an optional edge of $G$ with $u\in A$
  and $x\in B$ (resp.\ $u\in A'$ and $x\in B'$), then $u$ is complete
  to $K_x$ (in particular $uv\in E(G^*)$).

\item\label{i:condOpt} If $ux$ and $uy$ are optional edges of $G$ with
  $u\in A$, $x, y\in B$ and $H_y\subsetneq H_x$ (resp.\ $u\in A'$,
  $x, y\in B'$ and $H'_y\subsetneq H'_x$), then every vertex of $K_u$
  with a neighbor in $K_y$ is complete to $K_x$.
  
\item\label{i:condIso} $w$ (resp.\ $w'$) is a universal vertex of
  $G^*[\bigcup_{u\in A\cup B} K_u]$ (resp.\
  $G^*[\bigcup_{u\in A'\cup B'} K_u]$).
\end{enumerate}

\vspace{2ex}

Observe that $G=G^*[V(G)]$ follows clearly from the definition, so
$G$ is an induced subgraph of $G^*$.  For every vertex $u$ of $G$, the
clique $K_u$ is called a \emph{blown up clique}, more specifically the
\emph{clique blown up from $u$}.

Note that to define the blowup of a graph, it is first needed to fix
an $\ell$-partition of it.  Also, it should be stressed that the
blowup is defined only for twinless graphs.  Hence, in
condition~\eqref{i:condOpt} of the definition, since $G$ is twinless,
when $x\neq y$, $H_y\subsetneq H_x$ is equivalent to
$H_y\subseteq H_x$ because $H_x=H_y$ would imply that $x$ and $y$ are
twins.

\begin{lemma}
  \label{l:atMost1}
  A hole $C$ in a blowup of a twinless odd $\ell$-template contains at
  most one vertex in each blown up clique.
\end{lemma}

\begin{proof}
  Since a hole is triangle-free, $C$ intersects any clique in at most
  two vertices. So suppose for a contradiction that some blown up
  clique $K_v$ contains two vertices $x$ and $y$ of $C$.  Let $x'$ be
  the neighbor of $x$ in $C\sm y$ and $y'$ be the neighbor of $y$ in
  $C\sm x$.  Since by condition \eqref{i:KKhalf} of the definition of the blowup we have that in $G^*,$  $N[x] \subseteq N[y]$ or
  $N[y] \subseteq N[x]$, one of $xyx'$ or $xyy'$ is a
  triangle of $C$, a contradiction.
\end{proof}

\begin{lemma}  \label{blowoddtemp}
  In a blowup $G^*$ of a twinless odd $\ell$-template $G$, every hole
  has length $2\ell + 1$.
\end{lemma}

\begin{proof}
  Let $C^*$ be a hole in $G^*$. By Lemma~\ref{l:atMost1}, it contains
  at most one vertex in each blown up clique. Let $C$ be the subgraph
  of $G$ that is induced by all vertices $v$ such that some vertex of
  $C^*$ is in $K_v$. By Lemma~\ref{l:atMost1}, $|V(C^*)|=|V(C)|$.  By
  the definition of blowup (specifically conditions~\eqref{i:bKKanti}
  and~\eqref{i:Kcomp}), $C^*$ is isomorphic to some graph obtained
  from $C$ by removing optional or flat edges of $G$.  Hence, $C$ is a
  cycle of $G$ with no solid chord. If $C$ is a hole of $G$, then
  since it has the same length as $C^*$, by Lemma~\ref{l:HinTemplate},
  $C^*$ has length $2\ell+1$.  Hence, we may assume that $C$ has
  chords, so by Lemma~\ref{l:removeFr}, without loss of generality, $C$ contains three consecutive
  vertices $x$, $y$, $u$ such that: $u\in A$,
  $x, y\in B$, $\{u\} \subsetneq H_y \subsetneq H_x$ and $u$ is
  an isolated vertex of $H_x$.  Note that it follows that both
  $ux$ and $uy$ are optional edges of $G$.  Because of $C^*$, the vertex $u_i$ of
  $K_u\cap V(C^*)$ has a neighbor in $K_y$.  So, by
  condition~\eqref{i:condOpt} of blowups, $u_i$ is complete to $K_x$.
  Hence, $C^*$ has a chord, a contradiction.
\end{proof}

\subsection{Preblowup} \label{def_preblowup}

Checking that a graph is the blowup of a template is tedious. Here we
provide a simpler notion and prove it is in some sense equivalent.

A \emph{preblowup} of an odd $\ell$-template $G$ with an
$\ell$-partition $(A, B, A', B', I, w, w')$ is any graph $G^*$
obtained from $G$ as follows.  Every vertex $u$ of $A\cup A' \cup I$
is replaced by a clique $K_u$ on $k_u \geq 1$ vertices such that
$u \in K_u$.  We denote by $A^*$ the set $\bigcup_{u\in A} K_u$ and use a
similar notation $A'^*$ and $I^*$.  The set $B$ (resp.\ $B'$) is
replaced by a set $B^*$ (resp.\ $B'^*$) of vertices such that
$B\subseteq B^*$ (resp.\ $B'\subseteq B'^*$).  So,
$V(G^*) = A^* \cup B^* \cup A'^* \cup B'^* \cup I^*$.  The sets $A^*$,
$B^*$, $A'^*$, $B'^*$, $I^*$ are disjoint.  Vertices of $G$ are
adjacent in $G^*$ if and only if they are adjacent in $G$, so $G$ is
an induced subgraph of $G^*$.  Finally, we require that the following
conditions hold (throughout $N$ refers to the neighborhood in $G^*$):

\begin{enumerate}
\item\label{pb:A} For all $u\in A$,
  $N(K_u)\subseteq A^*\cup B^* \cup K_{u^+}$ where $u^+$ is the
  neighbor of $u$ in $I$ and:
  
  \begin{enumerate}
  \item\label{pb:Acomp}
   For every $u^*\in K_u$, $N_A(u^*)=N_A[u]$.
     
  \item\label{pb:AI}
    Every vertex of $K_u$ has a neighbor in $K_{u^+}$.
  \end{enumerate}

\item\label{pb:B}$N(B^*) \subseteq A^*$ and: 
  \begin{enumerate}
  \item\label{pb:Bw} If $w\in B$, then there exists $w^*\in B^*$ that is complete to $A^*$.
  \item\label{pb:BAN} If $u^*\in B^*$, then there exist non-adjacent
    $a, b\in A$ such that $u^*$ has neighbors in both $K_a$ and $K_b$.
  \end{enumerate}

  \setcounter{enumi}{8}
\item\label{pb:I}For all $u\in I$, $N(K_u)\subseteq K_a \cup K_{b}$
  where $a$ and $b$ are the neighbors of $u$ in $G$, and:
  \begin{enumerate}
  \item\label{pb:II} Every vertex $u^*\in K_u$ has at least one
    neighbor in each of $K_a$ and $K_b$.
  \end{enumerate}
  
  Conditions (a$'$) and (b$'$) analogous to~\eqref{pb:A} and~\eqref{pb:B} hold for
  $A'$ and  $B'$.
\end{enumerate}
 
Recall that to blowup (resp. preblowup) a template, one needs to
first fix an $\ell$-partition.  If this partition is proper, the
blowup (resp. preblowup) is \emph{proper}.  Recall that by
Lemma~\ref{l:TProper}, a proper $\ell$-partition
$(A, B, A', B', I, w, w')$ exists for every twinless odd
$\ell$-template $G$ (but this remark will be used only in the next
section, so far we just assume the $\ell$-partition we work with is
proper).

When $G^*$ is a preblowup of a template
$G$, the \emph{domination score} of $G$ w.r.t.\ $G^*$ is (where $N$ refers to the neighborhood in $G^*$):
$$
s(G, G^*) = \sum_{x\in A\cup A' \cup I}\left| \left\{ x^*\in K_x :
    N[x^*] \subseteq N[x] \right\} \right|
$$

Observe that the
blowup is defined only for twinless templates while the preblowup is
defined for any template. It is straightforward to check that a blowup
is a particular preblowup. The following is a converse of this
statement.

\begin{lemma}
  \label{l:preblowup}
  Let $\ell \ge 3$ and let $G^*$ be a proper preblowup of an odd $\ell$-template
  with $k\geq 3$ principal paths. If $G^*\in \mathcal C_{2\ell+1}$,
  then $G^*$ is a proper blowup of a twinless odd $\ell$-template $G$
  with $k$ principal paths (in particular, $G$ is an induced subgraph
  of $G^*$).
\end{lemma}

\begin{proof}
  Among all the induced subgraphs of $G^*$ that are odd
  $\ell$-templates and for which $G^*$ is a proper preblowup,
  we suppose that $G$ is one that maximizes $s(G, G^*)$. We denote
  by $(A, B, A', B', I, w, w')$ the proper $\ell$-partition of $G$ that
  is used for its preblowup and by $(A^*, B^*, A'^*, B'^*, I^*)$ the
  corresponding partition of the vertices of $G^*$.
   
  \begin{claim}\label{cb:chapeau}
    There exist vertices $w^*$ and $w'^*$ that are complete to
    respectively $A^*\sm\{w^*\}$ and $A'^* \sm\{w'^*\},$ and such that
    either $w^* \in B^*$ and $w'^*\in A'^*$, or $w^*\in A^*$ and
    $w'^*\in B'^*$.
  \end{claim}

  \begin{proofclaim}
    If $w\in A$, then from the definition of $w$ (see Lemma
    \ref{l:ABuniv}), the definition of $A^*$ and
    condition~\eqref{pb:Acomp}, it follows that $w^*=w$ is complete to
    $A^* \sm\{w^*\}$.  If $w\in B,$ by condition~\eqref{pb:Bw} there
    exists $w^*\in B^*$ that is complete to $A^*$.
    
    The statement about $w'^*$ holds by symmetry.  The last statement
    comes from the fact that by Lemma \ref{l:ABuniv} exactly one of
    $w, w'$ is in $A\cup A'$, and the other one is in $B\cup B'$.
  \end{proofclaim}

  \begin{claim}
    \label{c:principB}
    For every principal path $P_u=u\dots u'$ of $G$ and $u^*\in K_u$,
    there exists in $G^*$ a path $P_{u^*}$ of length $\ell-1$
    from $u^*$ to some ${u}'^*\in K_{u'}$ whose interior is in
    $\bigcup_{x\in I\cap V(P_u)} K_x$. Moreover, the interior of $P_{u^*}$ is
    anticomplete to $V(G^*) \sm \bigcup_{v\in V(P_u)} K_v$.
  \end{claim}
  
  \begin{proofclaim}
    The existence of a path from $u^*$ to some $u'^*\in K_{u'}$ whose
    interior is in $\bigcup_{x\in I\cap V(P)} K_x$ follows from conditions~\eqref{pb:A},~\eqref{pb:I},~\eqref{pb:II},  
    and~\eqref{pb:AI} of preblowup.  Its length is $\ell-1$ by
    condition~\eqref{i:linkP} of templates.  The statement about its interior follows from
    conditions~\eqref{pb:A}, \eqref{pb:B} and~\eqref{pb:I} of
    preblowup.
  \end{proofclaim}

  \begin{claim}\label{cb:Anticomplet_in_A}
    For all $u,v\in A$ such that $uv\notin E(G)$, $K_u$ is anticomplete
    to $K_v$. A similar statement holds for $A'$.
  \end{claim}
  
  \begin{proofclaim}
    Suppose that there exists $u^*\in K_u$ and $v^*\in K_v$ such that
    $u^*v^*\in E(G^*)$.  By condition~\eqref{pb:Acomp} of preblowup,
    $u\neq u^*$ and $v\neq v^*$.  Let $P_u = u\dots u'$ and
    $P_v=v\dots v'$ be principal paths. Denote by $u^+$ the neighbor
    of $u$ in $P_u$ and by $v^+$ the neighbor of $v$ in $P_v$. By
    %condition \eqref{i:linkP} of templates, 
   property~\eqref{t:threshold} of a template,
    $u'v'\in E(G)$. Hence
    $uP_uu'v'P_vvv^*u^*u$ is a cycle $C$. By
    conditions \eqref{pb:A} and \eqref{pb:Acomp} of preblowup, the only
    possible chords in $C$ are $u^+u^*$ and $v^+v^*$. Without loss of
    generality,  we may assume that $u^*u^+\in E(G^*)$ for otherwise $C$ is a hole of
    length $2\ell +2$, a contradiction.
    
    Let $P_{v^*}$ be a path of length $\ell -1$ from $v^*$ to $v'^{*}$
    as defined in~\eqref{c:principB}. Since $v'^{*}\in K_{v'}$ and
    by~\eqref{pb:Acomp} applied to $A'$, $v'^{*}u'\in E(G^*)$ and
    $v^*P_{v^*}v'^{*}u'P_uu^+u^*v^*$ is a hole of length $2\ell$, a
    contradiction.
    
    The result for $A'$ holds symmetrically.
  \end{proofclaim}
  
  \begin{claim}\label{cb:Complet_in_A}
    For all $u,v\in A$ such that $uv\in E(G)$, $K_u$ is complete to
    $K_v$.  A similar statement holds for $A'$.
  \end{claim}
  
  \begin{proofclaim}
    Suppose that there exists $u^*\in K_u$ and $v^*\in K_v$ such that
    $u^*v^*\notin E(G^*)$. Let $P_{u^*} = u^*\dots u'^*$ and
    $P_{v^*} = v^*\dots v'^*$ be defined as
    in~\eqref{c:principB}. Observe that $u'^{*}\in K_{u'}$ and
    $v'^{*}\in K_{v'}$. Furthermore $u'v'\notin E(G)$ by
      property~\eqref{t:threshold} of templates.  Hence,
    by~\eqref{cb:Anticomplet_in_A}, $u'^{*}v'^{*}\notin E(G^*)$.

    We claim that there exists a vertex $a\in (A\cup B) \sm \{u,v\}$ that is
    adjacent to both $u^*$ and $v^*$. If $w^* \neq u,v$, then
    by~\eqref{cb:chapeau} and condition~\eqref{pb:AI}, we may choose
    $a=w^*$.  Otherwise, up to symmetry, $w^*=u$.  Since the
    $\ell$-partition of $G$ is proper, by Lemma~\ref{l:wProper}, $A$
    contains a universal vertex $x$ distinct from $w^*=u$.  If
    $x \neq v$, we set $a=x$. If $x=v$, then both $u$ and $v$ are
    universal vertices of $G[A]$ and we may choose for $a$ any vertex
    of $A\sm \{u, v\}$. This proves our claim.
    
    Now, $au^*P_{u^*}u'^{*}w'^*v'^{*}P_{v^*}v^*a$ is a hole of length
    $2\ell +2$, a contradiction.  The result for $A'$ holds
    symmetrically.
  \end{proofclaim}
  
  \begin{claim}\label{cb:Inested}
    For all $u\in I$ and $u_1, u_2\in K_u$, either
    $N[u_1]\subseteq N[u_2]$ or $N[u_2]\subseteq N[u_1]$.
  \end{claim}

  \begin{proofclaim}
    Otherwise, there exists $x^*_1\in N[u_1]\setminus N[u_2]$ and
    $x^*_2\in N[u_2] \setminus N[u_1]$.  Note that
    $x^*_1x^*_2\notin E(G^*)$ for otherwise,
    $\{x^*_1, x^*_2, u_1, u_2\}$ induces a $C_4$.  It follows that $x^*_1$ and $x^*_2$ 
    belong respectively to distinct cliques $K_{x_1}$ and $K_{x_2}$, 
    where $x_1$ and $x_2$ are the two neighbors of $u$
    along some principal path $P = v\dots v'$  of $G$.  Because of
    $x^*_1$, $x^*_2$ and condition~\eqref{pb:II} of
    preblowup, there exists a path $P^*$ of length $\ell$ from some
    $v^*\in K_v$ to some $v'^*\in K_{v'}$ whose interior is  in
    $\bigcup_{x\in I\cap V(P)} K_x$.

    Let $q \neq v$ be a vertex of $A$ and $Q = q\dots q'$ be a principal path of $G$,
    and suppose up to symmetry that $qv\notin E(G)$.  Now, by
    conditions~\eqref{pb:I} and~\eqref{pb:Acomp} of preblowup
    and~\eqref{cb:chapeau}, $P^*$, $Q$ and $w^*$ form a hole of length
    $2\ell+2$.
  \end{proofclaim}

  \begin{claim}\label{cb:Inestedu}
    For all $u\in I$ and $u^*\in K_u$, 
    $N[u^*]\subseteq N[u]$.
  \end{claim}

  \begin{proofclaim}
    Otherwise, by~\eqref{cb:Inested}, there exists a vertex
    $u^*\in K_u$ such that $N[u] \subsetneq N[u^*]$.  Hence
    $(V(G) \sm \{u\}) \cup \{u^*\}$ induces a subgraph $G_0$ of $G^*$
    and it is easy to verify that $G^*$ is a preblowup of $G_0$. This
    contradicts to the maximality of $s(G, G^*)$.
  \end{proofclaim}
  
  By~(\ref{cb:Inested}), for every $u\in I$, the clique $K_u$ can be
  linearly ordered by the inclusion of the neighborhoods as
  $u_1, \dots, u_{k_u}$ with $u=u_{k_u}$ by~\eqref{cb:Inestedu} (so,
  for $1\leq i \leq j \leq k_u$, $N[u_i] \subseteq N[u_j]$). From 
  condition~\eqref{pb:I} of the preblowup it also follows that, in $G^*$, $u$ is complete to the cliques
  associated to its two neighbors in~$G$.

  \begin{claim}\label{cb:Anested}
    For every $u\in A$ and $u_1,u_2\in K_u$, either
    $N[u_1]\subseteq N[u_2]$ or $N[u_2]\subseteq N[u_1]$. A similar statement holds for $A'$.
  \end{claim}
  
  \begin{proofclaim} 
    Otherwise, there exist $x_1\in N[u_1]\setminus N[u_2]$ and
    $x_2\in N[u_2]\setminus N[u_1]$.  Note that $x_1x_2\notin E(G^*)$
    for otherwise, $\{x_1, x_2, u_1, u_2\}$ induces a $C_4$.
    
    Observe first that by~\eqref{cb:Anticomplet_in_A} and~\eqref{cb:Complet_in_A}, $N_{A^*}[u_1]=N_{A^*}[u_2]$. Hence
    by~\eqref{pb:A} of preblowup, $x_1,x_2\in B^*\cup K_{u^+}$ where $u^+$ is the neighbor of $u$ in the principal path that contains $u$. Without loss of
    generality and since $K_{u^+}$ is a clique, $x_1\in B^*$.

    By condition~\eqref{pb:BAN}, there exist non-adjacent $a, b\in A$
    such that $x_1$ has neighbors $a^* \in K_a$ and $b^*\in K_b$, and  by~\eqref{cb:Anticomplet_in_A} $ a^*b^*\notin E(G^*).$
    Note that $a^*, b^*\neq u_2$ because $u_2x_1\notin E(G^*)$.  If
    $u_2$ is complete to $\{a^*, b^*\}$, then $\{u_2,a^*,x_1,b^*\}$
    induces a $C_4$, a contradiction.  So, up to symmetry
    $u_2a^*\notin E(G)$.  So, $a^*\notin K_u$ and
    by~\eqref{cb:Complet_in_A} and~\eqref{cb:Anticomplet_in_A}, $a^*u_1\notin E(G^*)$.
   Observe
    that $x_2a^*\notin E(G^*)$ for otherwise $\{a^*,x_1,u_1,u_2,x_2\}$
    induces a $C_5$.

    Suppose that $x_2\in B^*$. As above, we can show that $x_2$ has a neighbor
    $c^* \in A^*$ that is anticomplete to $\{u_1, u_2, x_1\}$. 
    Note that $a^*c^*\notin E(G^*)$ for otherwise
    $\{x_1,a^*,c^*,x_2,u_2,u_1\}$ induces a $C_6$. Let $P_{a^*} = a^*\dots a'^*$ and $P_{c^*} = c^*\dots c'^*$ 
    be defined as in \eqref{c:principB}.
    
    By~\eqref{cb:Anticomplet_in_A} and~\eqref{cb:Complet_in_A} and
    since $a^*c^*\notin E(G^*)$, $a'^*c'^*\in E(G^*)$. So, by conditions
    \eqref{pb:A}, \eqref{pb:B} and \eqref{pb:I},
    $u_1x_1a^*P_{a^*}a'^*c'^*P_{c^*}c^*x_2u_2u_1$ is a hole of length
    $2\ell+4$, a contradiction.

    So $x_2\in K_{u^+}$.  Hence by condition~\eqref{pb:II} of
    preblowup, there exists a path $Q$ of length $\ell-2$ from $x_2$
    to some $u'^*\in K_{u'}$. Now $x_2Qu'^*a'^*P_{a^*}a^*x_1u_1u_2x_2$
    is a hole of length $2\ell +2$, a contradiction.
    
    The result for $A'$ holds symmetrically.
  \end{proofclaim}

  \begin{claim}\label{cb:Anestedu}
    For all $u\in A$ and $u^*\in K_u$, $N[u^*]\subseteq N[u]$. A
    similar statement holds for $A'$.
  \end{claim}

  \begin{proofclaim}
    Otherwise, there exists a vertex $u^*\in K_u$ such that
    $N[u] \subsetneq N[u^*]$.  Hence, $(V(G)\sm \{u\}) \cup\{u^*\} $
    induces a subgraph $G_0$ of $G^*$ and it is easy to verify that
    $G^*$ is a preblowup of $G_0$ (that is a template by
    Lemma~\ref{th:ptist} and whose partition is proper
    by~\eqref{cb:Anticomplet_in_A} and~\eqref{cb:Complet_in_A}). This
    contradicts the maximality of $s(G, G^*)$.  The result for $A'$
    holds symmetrically.
  \end{proofclaim}

  By~(\ref{cb:Anested}), for every $u\in A \cup A'$, the clique $K_u$ can be
  linearly ordered by the inclusion of the neighborhoods as
  $u_1, \dots, u_{k_u}$, and by~(\ref{cb:Anestedu}) $u_{k_u} =u$ (so,
  for $1\leq i \leq j \leq k_u$, $N[u_i] \subseteq N[u_j]$).

  \begin{claim}
    \label{cb:Bnested}
    If $xy$ is an edge of $G[B^*]$, then either
    $N_{A^*}(x)\subseteq N_{A^*}(y)$ or
    $N_{A^*}(y)\subseteq N_{A^*}(x)$. 
  \end{claim}
  
  \begin{proofclaim}
    Otherwise, there exists $u^*\in N_{A^*}(x)\setminus N_{A^*}(y)$
    and $v^*\in N_{A^*}(y)\setminus N_{A^*}(x)$. Note that
    $u^*v^*\notin E$ for otherwise $\{u^*,x,y,v^*\}$ induces a
    $C_4$. 
    So, for some $u, v\in A$, we have $u^*\in K_u$ and
    $v^*\in K_v$. Hence, by \eqref{cb:Complet_in_A},
    $uv\notin E(G)$.  Let $P_{u^*} = u^*\dots u'^*$ and  
    $P_{v^*} = v^*\dots v'^*$ be defined 
    as in \eqref{c:principB}.  So, $xu^*P_{u^*}u'^*v'^*P_{v^*}v^*yx$ form a hole
    of length $2\ell+2$, a contradiction. 
  \end{proofclaim}

  \begin{claim}
    \label{c:nonAdjBstar}
    For every $x\in B^*$, there exist non-adjacent $u,v\in A$ such
    that $xu, xv \in E(G^*)$. 
  \end{claim}
    
  \begin{proofclaim}
    This follows from condition~\eqref{pb:BAN} of
    preblowup and from (\ref{cb:Anestedu}).
  \end{proofclaim}

  Two vertices $x, y$ in $B^*$ are \emph{equivalent} if
  $N_{A}(x) = N_{A}(y)$.

  \begin{claim}\label{cb:Bclique}
    If $x$ and $y$ are equivalent vertices of $B^*$, then
    $xy\in E(G^*)$.
  \end{claim}
  
  \begin{proofclaim}
    If $xy\notin E(G^*)$, then $x$, $y$ and two of their neighbors
    provided by~\eqref{c:nonAdjBstar} induce a $C_4$. 
  \end{proofclaim}

  Vertices of $B^*$ are partitioned into equivalence
  classes. By~\eqref{cb:Bclique}, each equivalence class is a clique
  $X$, and by~\eqref{cb:Bnested}, vertices of $X$ can be linearly
  ordered according to the inclusion of neighborhoods in $A^*$.  In each such a
  clique $X$ we choose a vertex $x$ maximal for the order and call
  $B_1$ the set of these maximal vertices.  For every $x\in B_1$, we
  denote by $K_x$ the clique of $B^*$ of all vertices equivalent to
  $x$. Observe that if $w^*\in B$, then $w^*$ is a maximal vertex of its
  clique. Hence, we can set $w^*\in B_1$.
 
So, for every $u\in B_1$, the clique $K_u$ can be linearly
  ordered by the inclusion of the neighborhod in $A^*$ as
  $u_1, \dots, u_{k_u}$ with $u=u_{k_u}$ (so, for
  $1\leq i \leq j \leq k_u$, $N_{A^*}(u_i) \subseteq N_{A^*}(u_j)$).
  
 Statements similar to \eqref{cb:Bnested}, \eqref{c:nonAdjBstar}, \eqref{cb:Bclique}
hold for $B'^*$ and we define $B'_1$ as well. 
  
We set $G_1 = G^*[A \cup B_1 \cup A' \cup B'_1 \cup I]$ and claim that
$(A, B_1, A', B'_1, I)$ is an $\ell$-pretemplate partition of $G_1$.
Since $G_1[A\cup I \cup A']$ is exactly $G[A\cup I \cup A']$,
conditions~\eqref{a:cardinality}, \eqref{a:path}, \eqref{a:interI} and \eqref{a:sp}
hold. Adding the fact that $N_{G_1}(B_1)\subseteq A^*\cap V(G_1)=A$
by condition \eqref{pb:B} of preblowup, condition \eqref{a:antiC} for
a pretemplate holds and symmetrically also condition
\eqref{a:antiCP}. Now condition \eqref{a:cn} holds because $w^*$ and
$w'^*$ are complete to respectively $A\cup B_1$ and $A'\cup B'_1$. By
\eqref{c:nonAdjBstar}, the last two conditions for a pretemplate are
fulfilled by $(A, B_1, A', B'_1, I)$. Hence, by Lemma \ref{th:ptist},
$G_1$ is a an odd $\ell$-template.  It is twinless by
Lemma~\ref{l:wTwins}. We also notice that by construction $w^*$ and
$w'^*$ belong to $G_1$. Furthermore, by \eqref{cb:chapeau}, $w^*$
(respectively $w'^*$) is complete to $A\sm\{w^*\}$ (respectively
$A'\sm\{w'^*\}$). From the definition of a template it is easy to
conclude that $w^*$ (respectively $w'^*$) is universal in
$G_1[A \cup B_1]$ (respectively $G_1[A' \cup B'_1]$). Hence
$(A, B_1, A', B'_1, I,w^*,w'^*)$ is a proper $\ell$-partition of
$G_1$.

   We now prove that
  $G^*$ is a proper blowup of $G_1$.

 By the definition of a preblowup and by~\eqref{cb:Bclique}, for all $u\in V(G_1)$, $K_u$ is a clique and $V(G^*)= \bigcup_{u\in V(G_1)} K_u$ 

  \begin{claim}\label{cb:anticomplete}
    If $u, v\in V(G_1)$ and $uv\notin E(G_1)$, then $K_u$ is anticomplete to $K_v$.
  \end{claim}

  \begin{proofclaim}
    Suppose $u, v\in V(G_1)$ and $uv\notin E(G_1)$.  If $u\in I$ or
    $v\in I$, the conclusion follows directly from
    condition~\eqref{pb:I} of preblowup.  So we may assume up to
    symmetry that $u\in A\cup B_1$.  By conditions~\eqref{pb:A}
    and~\eqref{pb:B} of preblowup, we may assume $v\in A \cup B_1$.
    If $u, v\in A$, then the result follows from
  \eqref{cb:Anticomplet_in_A}, so we may assume that $v\in B_1$. 
  
  Now suppose for a contradiction that there exist $u^*\in K_u$ and
  $v^*\in K_v$ such that $u^*v^*\in E(G_1)$. By the choice of vertices
  in $B_1$, for all $v^*\in K_v$, $N[v^*]\subseteq N[v]$. So
  $u^*v\in E(G_1)$. For the same reason or by~\eqref{cb:Anestedu}, for
  all $u^*\in K_u$, $N[u^*]\subseteq N[u]$. Hence $uv \in E(G_1)$, a
  contradiction.
  \end{proofclaim}

 \begin{claim}\label{cb:cliquescompletes}
   If $uv$ is a solid edge of $G_1$ then $K_u$ is complete to $K_v$.
  \end{claim}

  \begin{proofclaim}
    Otherwise, let $u^*\in K_u$ and $v^*\in K_v$ such that $u^*v^*\notin
    E(G)$. Since $uv$ is a solid edge, up to symmetry, $u,v\in A$ or
    $u,v\in B_1$ or $u\in A$, $v\in B_1$ and in this last case $u$ is not an isolated vertex
    of $G[H_v]$.
    
    By~\eqref{cb:Complet_in_A} the case where $u$ and $v$ are in $A$
    cannot happen.  Assume then that
    $v \in B_1$.  By Lemma \ref{l:recoverHx}, there exist
    $a,b \in H_v$ (and hence in $A$) that are not adjacent. Assume
    that $u$ is also in $B_1$. Since $u$ and $v$ are adjacent, by \eqref{cb:Bnested} we may
    assume without loss of generality that $H_v \subseteq H_u$ and so
    $a$ and $b$ belong to $H_u$ too.  Then, by the definition of $K_u$
    and $K_v$, we get a $C_4$ induced by $\{u^*,v^*, a,b\}$, a
    contradiction.  So $u$ should be in $A,$ and to avoid a $C_4$
    induced by $\{u^*,v^*, a,b\}$, $u^*$ should be non-adjacent to at
    least one of $a$ and $b$, say $a$. In particular, $a\neq u$.
    Then, by~\eqref{cb:Complet_in_A}, $ua \notin E(G_1)$. So $u$ does
    not belong to $N(H_v)$ and since $uv$ is an edge of $G_1$, we get
    that $u \in H_v$. Since $uv$ is solid, $u$ has at least one neighbor in
    $H_v$, and it is not adjacent to $a \in H_v$. Hence, as $H_v$ is anticonnected, there exist  non-adjacent
    vertices $c, d \in H_v$ such that $uc \notin E(G_1)$ and
    $ud \in E(G_1)$. Now $u^*P_{u^*}u'^*c'P_ccv^*du^* $ is a hole of
    length $2\ell+2$, a contradiction again.

  \end{proofclaim}

  \begin{claim}\label{cb:nested}
    For all $u\in V(G_1)$ and $1\leq i \leq j \leq k_u$, 
    $N[u_i]\subseteq N[u_j]$. 
  \end{claim}

  \begin{proofclaim}
    The result follows from how vertices are ordered after the proof of
    \eqref{cb:Inested} (vertices in $I$), \eqref{cb:Anested} (vertices in $A$ or $A'$) and
    \eqref{cb:Bclique} (vertices in $B_1$ or $B'_1$).
  \end{proofclaim}

  \begin{claim}
    \label{cb:flat} If $uv$ is a flat edge of $G_1$, then $u$ is
    complete to $K_v$ and $v$ is complete to $K_u$.
  \end{claim}
  
  \begin{proofclaim}
    By definition of a flat edge, either $u$ and $v$ are in $I$ or one
    is in $I$ and the other is in $A$ or in $A'$. The result follows
    from~\eqref{cb:Inestedu}, \eqref{cb:Anestedu}, and
    conditions~\eqref{pb:AI} (applied to $A$ or $A'$)
    and~\eqref{pb:II} of the preblowup.
  \end{proofclaim}

  \begin{claim}
    \label{c:optAB}
    If $ux$ is an optional edge of $G_1$ with $u\in A$
    and $x\in B_1$ (resp.\ $u\in A'$ and $x\in B_1'$), then $u$ is
    complete to $K_x$.
  \end{claim}

  \begin{proofclaim}
    The result follows from the definition of $K_x$ when $x \in B_1$. 
  \end{proofclaim}

  \begin{claim}\label{cb:cascade}
    If $ux$ and $uy$ are optional edges with $u\in A$, $x, y\in B_1$
    and $H_y\subsetneq H_x$ (resp.\ $u\in A'$, $x, y\in B_1'$ and
    $H'_y\subsetneq H'_x$), then every vertex of $K_u$ with a neighbor
    in $K_y$ is complete to $K_x$.
  \end{claim}

  \begin{proofclaim}
    Otherwise, let $u^*$ be a vertex in $K_u$ that has a neighbor
    $y^*$ in $K_y$ and a non-neighbor $x^*$ in $K_x$. Since
    $H_x$ and $H_y$ are not disjoint, $xy$ is a solid edge of $G_1$ and
    by~(\ref{cb:cliquescompletes}), $x^*y^*\in E(G_1)$.
    
%  Let us notice that $N_A(y)\subseteq N_A(x)$. Indeed, let $a \in A \sm N_A(x)$. By definition of a template, $a \in A \sm (H_x \cup N_A[H_x])$. Then since $H_y\subsetneq H_x$ and $H_x$ is a module of $A$ we get that $a$ is anticomplete to $H_x$ and hence to $H_y$. So $a \notin N_A(y)$.
  
  Since $x$ and $y$ are not equivalent, there exists a vertex $a$  such that $a\in N_A(y)\sm N_A(x)$ or $a\in N_A(x)\sm N_A(y)$. In the first case, by definition of a template, $a \in A \sm N_A[H_x]$. Then since $H_y\subsetneq H_x$ and $H_x$ is a module of $A$ we get that $a$ is anticomplete to $H_x$ and hence to $H_y$. So $a \notin N_A(y)$, a contradiction; we may then conclude that $a\in N_A(x)\sm N_A(y)$
  
  By definition of the cliques in $B$, $x^*a\in E(G^*)$ and $y^*a\notin E(G^*)$. Therefore, to avoid a $C_4$ induced by $\{x^*, y^*, u^*,a\}$, it should be that $u^*a\notin E(G^*)$. 

%    $x^*a\in E(G^*)$ and $y^*a\notin E(G^*)$. Therefore, to avoid a $C_4$
%    it should be that $u^*a\notin E(G^*)$.
%    $H_y\subsetneq H_x$, $N_A(y)\subsetneq N_A(x)$ and there exists
%    $a\in N_A(x)\sm N_A(y)$.
%    Since $x$ and $y$ are not equivalent, $N_A(x)\neq N_A(y)$. Since
%    $H_y\subsetneq H_x$, $N_A(y)\subsetneq N_A(x)$ and there exists
%    $a\in N_A(x)\sm N_A(y)$.  By definition of the cliques in $B$,
%    $x^*a\in E(G^*)$ and $y^*a\notin E(G^*)$. Therefore, to avoid a $C_4$
%    it should be that $u^*a\notin E(G^*)$.
    
    Now $aP_aa'u'^*P_{u^*}u^*y^*x^*a$ is a hole of length $2\ell +2$ a
    contradiction.
 \end{proofclaim}

  \begin{claim}
    \label{cb:condIso} $w^*$ (resp.\ $w'^*$) is a universal vertex of
    $G^*[\bigcup_{u\in A\cup B_1} K_u]$ (resp.\
    $G^*[\bigcup_{u\in A'\cup B_1'} K_u]$).
  \end{claim}

  \begin{proofclaim}
  
    By \eqref{cb:chapeau}, $w^*$ is complete to $A^* \sm \{w^*\}$ and
    so to $\bigcup_{u\in A} K_u \sm \{w^*\}$. Furthermore, from the definition of
    $G_1$ we know that $w^*$ is complete to $B_1 \sm \{w^*\}$. If
    $w^*\in B_1$, since all edges between vertices in $B_1$ are solid,
    by~(\ref{cb:cliquescompletes}), $w^*$ is complete to
    $B^*\sm \{w^*\}$. Similarly, if $w^*\in A$,
    by~\eqref{cb:cliquescompletes} and~\eqref{c:optAB}, we get that
    $w^*$ is complete to $B^*$.  In both cases $w^*$ is a universal
    vertex of $G^*[\bigcup_{u\in A\cup B} K_u]$. The proof for $w'^*$ is
    symmetric.
  \end{proofclaim}

  From all the claims above, $G^*$ satisfies all conditions to be a
  proper blowup of $G_1$.
\end{proof}

\section{Graphs in $\mathcal C_{2\ell +1}$ that contain a pyramid}

The goal of this section is to prove the the following.

\begin{lemma}
  \label{l:GPy}
  Let $\ell\geq 3$ be an integer.  If $G$ is a graph in
  $\mathcal C_{2\ell +1}$ and $G$ contains a pyramid, then one of the
  following holds:

  \begin{enumerate}
  \item\label{c:Pyblowup} $G$ is a proper blowup of a twinless odd
    $\ell$-template;
  \item\label{c:Pyuniv} $G$ has a universal vertex;
  \item\label{c:Pycliquecut} $G$ has a clique cutset.
  \end{enumerate}
\end{lemma}

The rest of this section is devoted to the proof of
Lemma~\ref{l:GPy}. So from here on $\ell\geq 3$ is an integer and $G$
is a graph in $\mathcal C_{2\ell + 1}$ that contains a pyramid~$\Pi$.
By Lemma~\ref{l:holeTruemperS}, the three paths of $\Pi$ have
length~$\ell$.  By Lemma~\ref{l:pyrTemplate}, $\Pi$ is an odd
$\ell$-template.  Hence, we may define an integer $k$ and a sequence
$F_0, F_1, F_2$ of induced subgraphs of $G$ as follows.

\begin{itemize}
\item $k$ is the maximum integer such that $G$ contains
  an odd $\ell$-template with $k$ principal paths. Observe that by
  Lemma~\ref{l:wTwins}, $G$ in fact contains a twinless template with
  $k$ principal paths, because twins can be eliminated from templates
  by deleting hyperedges with equal vertex-set while there are some.

\item In $G$, pick a proper blowup $F_1$ of a twinless odd
  $\ell$-template $F_0$ with $k$ principal paths.  Note that $F_0$
  exists and the proper $\ell$-partition needed for the proper blowup
  exists by Lemma~\ref{l:TProper}.

\item Suppose that $F_0$ and $F_1$ are chosen subject to the
  maximality of the vertex-set of $F_1$ (in the sense of
  inclusion). Note that possibly $F_0$ is not a maximal template in
  the sense of inclusion, it can be that a smaller template leads to a
  bigger blowup (but $F_0$ has $k$ principal paths).

\item $F_2$ is obtained from $F_1$ by adding all vertices of
  $G\sm F_1$ that are complete to $F_1$. 
\end{itemize}

\begin{lemma}
  \label{l:F2-F1Clique}
  $V(F_2)\sm V(F_1)$ is a (possibly empty) clique that is
  complete to $F_1$.
\end{lemma}

\begin{proof}
  Otherwise, $G$ contains a $C_4$.  
\end{proof}

We now introduce some notation.  We denote by $(A, B, A', B', I, w, w')$ the
proper $\ell$-partition that is used to blow up $F_0$.  When $u$ is a vertex
of $F_0$, we denote by $K_u$ the clique of $F_1$ that is blown up from
$u$.  We set $A^* = \bigcup_{u\in A} K_u$.  We use a similar notation
$B^*$, $A'^*$, $B'^*$ and $I ^*$.

\subsection{Technical lemmas}

We now prove lemmas that sum up several structural properties of $G$.

%63
\begin{lemma}\label{lb:universalevenforblow-up}
  If $u\in A\cup A' \cup I \cup \{w, w'\}$ and $v\in N_{V(F_0)}(u)$,
  then $u$ is complete to $K_v$. 
\end{lemma}

\begin{proof}
  We prove this lemma using the conditions from the definition of
  blowups. If $u\in \{w,w'\}$, then the result follows from
  condition~\eqref{i:condIso}. If $u\in A\cup A'$, then the conclusion
  follows from conditions~\eqref{i:Kcomp}, \eqref{i:flat} and~\eqref
  {i:optAB}.  If $u\in I$, then the conclusion follows from
  condition~\eqref {i:flat}.
\end{proof}

Very often, Lemma~\ref{lb:universalevenforblow-up} will be used in the
following way. Suppose there exists a principal path $P=u\dots u'$ of
$F_0$.  Suppose there exists a vertex $x$ of $P$ and $x^*\in K_x$.
Then by Lemma~\ref{lb:universalevenforblow-up} and
condition~\eqref{i:bKKanti} of blowups,
$\{x^*\} \cup( V(P) \sm \{x\})$ induces a path of $F_1$. If $y\neq x$
is a vertex of $P$ and $y^*\in K_y$, then
$\{x^*, y^*\} \cup (V(P) \sm \{x, y\})$ might fail to induce a path of
$F_1$, because it is possible that $xy\in E(G)$ while
$x^*y^*\notin E(G)$.  But under the assumption that $x^*y^*\in E(G)$
or $xy\notin E(G)$, we do have that
$\{x^*, y^*\} \cup (V(P) \sm \{x, y\})$ induces a path of
$F_1$. Several variant of this situation will appear soon and we will
simply justify them by refering to
Lemma~\ref{lb:universalevenforblow-up}.

When $u$ is a vertex in $A$, we denote by $P_u$ the unique principal
path of $F_0$ that contains $u$. Its end in $A'$ is then denoted by
$u'$.  We denote by $u^+$ the neighbor of $u$ in $P_u$. We denote by
$u^{++}$ the neighbor of $u^+$ in $P_u\sm u$.  Note that $u^+\in I$
and $u^{++}\in I\cup A'$ ($u^{++} \in A'$ if and only if $\ell=3$).

For any distinct $u, v\in A$, from the definition of templates,
exactly one of $V(P_u) \cup V(P_v) \cup \{w\}$ or
$V(P_u) \cup V(P_v) \cup \{w'\}$ induces a hole that is denoted by
$C_{u, v}$.  Such a hole is called a \emph{principal hole}.

Note that there are two kinds of principal holes: those that contain
$w$, and those that contain $w'$.  Recall that by
Lemma~\ref{l:HinTemplate}, every hole of a template contains two
principal paths plus an extra vertex, but it may fail to be a
principal hole (because it may fail to contain $w$ or $w'$).  Though
we do not use this information formally, it is worth noting that by
Lemma~\ref{lb:universalevenforblow-up}, when $C$ is a principal hole,
$\bigcup_{v\in V(C)} K_v$ induces a ring. But when $C$ is a non-principal
hole, it may happen that $\bigcup_{v\in V(C)} K_v$ does not induce a
ring (because there might be in $C$ an optional edge $uv$ with
$u\in A$ and $v\in B$, and after the blowup process, there might be no
vertex in $K_v$ that is complete to $K_u$).

\begin{lemma}
  \label{l:twoNonAdj}
  If $u\in V(F_0)$ and $u^*\in K_u$, then $u^*$ has two 
  neighbors in $V(F_0)\sm K_u$ that are not adjacent.
\end{lemma}

\begin{proof}
  If $u\in I$, then let $P$ be the principal path that contains $u$. By
  Lemma~\ref{lb:universalevenforblow-up}, $u^*$ is adjacent to the two
  neighbors of $u$ in $P$.

  If $u\in A\cup A'$, say $u\in A$ up to symmetry, then we claim that
  $u$ has a neighbor $z$ in $A\cup B$.  This is clear if $u$ is not
  isolated in $A$  and otherwise we set $z=w$.  By
  Lemma~\ref{lb:universalevenforblow-up}, $z$ and $u^+$ are
  non-adjacent neighbors of $u^*$.

  If $u\in B$, then by the definition of a template, $H_u$ contains two non adjacent vertices $a$ and $b$ that are neighbors of $u$.  
  By Lemma~\ref{lb:universalevenforblow-up}, $a$ and $b$ are
  both adjacent to $u^*$. 
\end{proof}

\begin{lemma}
  \label{l:htab}
  If $uv$ is an edge of $F_0[A\cup A' \cup I \cup \{w, w'\}]$, then
  some principal hole of $F_0$ goes through $uv$.
\end{lemma}

\begin{proof}
  If at least one of $u$, $v$ is in $I$ then $uv$ is an edge of a
  principal path and we know that this principal path belongs to a
  principal hole.  Else, since $A\cup \{w\}$ is
  anticomplete to $A'\cup \{w'\}$, up to symmetry both $u$ and $v$ are
  in $A$ or $u= w\in B$ and $v\in A$.

  If $u,v \in A$ then $C_{u,v}$ is a principal hole containing $uv$.

  If $u= w\in B$ and $v\in A$ : since $w$ is in $B$, $G[A]$ has no
  universal vertex and there exists $a \in A$ which is not adjacent to
  $v$. Now $w, P_v, P_a$ form a principal hole containing the edge
  $uv$.
\end{proof}

\begin{lemma}
  \label{l:onlyKaKb}
  If $K$ is a clique of $F_0$, $K^* = \bigcup_{v\in K} K_v$ and $D$ is a
  connected induced subgraph of $G\sm F_2$ such that
  $N_{V(F_1)}(D) \subseteq K^*$, then $N_{V(F_1)}(D)$ is a clique.
\end{lemma}

\begin{proof}
  For suppose not.  This means that there exists
  $u^*, v^*\in K^*$ and $x_u, x_v\in D$ such that
  $u^*v^*\notin E(G)$ and $x_uu^*, x_vv^*\in E(G)$ (possibly
  $x_u=x_v$).  Since $D$ is connected, there exists a path $P$ in $D$
  from $x_u$ to $x_v$.  Suppose that $u^*$, $x_u$, $v^*$, $x_v$ and
  $P$ are chosen subject to the minimality of $P$.  It follows that
  $u^*x_uPx_vv^*$ is a path, and recall that by assumption
  its interior is anticomplete to $F_1\sm K^*$.
  
  Since $u^*v^*\notin E(G)$, $u^*$ and $v^*$ are in different blown-up
  cliques. Denote by $K_u$ and $K_v$ the blown-up cliques such that
  $u^*\in K_u$ and $v^*\in K_v$. By hypothesis, $uv\in K$ and so
  $uv\in E(G)$. Since $u^*v^*\notin E(G)$, by
  condition~\eqref{i:Kcomp} of blowups, $uv$ is not a solid edge of
  $G$.

  If $uv$ is a flat edge of $F_0$, then by Lemma~\ref{l:htab} a
  principal hole $C$ goes through $uv$. Note that apart from $u$ and
  $v$, no vertex of $C$ is in $K$ since $K$ is a clique. By
  Lemma~\ref{lb:universalevenforblow-up}, in $G$,
  $(\{u^*, v^*\}) \cup V(C))\sm \{u, v\}$ induces a path $Q$ of length
  $2\ell$.  So $P$ and $Q$ form a hole of length at least $2\ell +2$,
  a contradiction.

  If $uv$ is an optional edge of $F_0$, say with $u\in A$ and
  $v\in B$, then $u\in H_v$, and there exists $a$ in $H_v$ such that
  $au\notin E(F_0)$.  Therefore, $P_u$, $P_a$ and $v$ form a hole
  $C^*$. By condition~\eqref{i:optAB} of blowups (if $va$ is optional),
  or by condition~\eqref{i:Kcomp} (if $va$ is solid), $a$ is complete
  to $K_v$.  By Lemma~\ref{lb:universalevenforblow-up} it follows that
  $(\{u^*, v^*\}) \cup V(C^*))\sm \{u, v\}$ induces a path $Q$ of length
  $2\ell$.  So $P$ and $Q$ form a hole of length at least $2\ell +2$,
  a contradiction again.
\end{proof}

When $C$ is a hole of $G$, a vertex $v$ of $V(G)\sm V(C)$ is
\emph{minor} w.r.t.\ $C$ if $N_{V(C)}(v)$ is included in a 3-vertex
path of $C$.  A vertex of $V(G)\sm V(C)$ that is not minor w.r.t.\ $C$ is \emph{major}
w.r.t.\ $C$.

\begin{lemma}
  \label{l:xMinor}
  If $x\in V(G) \sm V(F_2)$ and $C$ is a principal hole of $F_0$, then $x$
  is minor w.r.t.\ $C$.
\end{lemma}

\begin{proof}
  Suppose up to symmetry that $w\in V(C)$ and suppose $C=C_{u, v}$.
  If $x$ is major w.r.t.\ $C$, then $C$ and $x$  form a theta or a
  wheel that is not a twin-wheel. So by Lemma~\ref{l:holeTruemperS},
  $x$ and $C$ form a universal wheel.  Let $P_t = t\dots t'$ be a
  principal path where $t\neq u, v$. If $t$ is complete to $\{u, v\}$,
  then $xt\in E(G)$ for otherwise $\{t, u, v, x\}$ induces a $C_4$.
  Hence $x$ has at least~4 neighbors in $C_{u, t}$, so by
  Lemma~\ref{l:holeTruemperS}, $x$ is complete to $P_t$.  If $t$ is not
  complete to $\{u, v\}$, say $tu\notin E(G)$, then $x$ again has at
  least~4 neighbors in $C_{u, t}$ because $w\in V(C_{u, t})$, so again
  $x$ is complete to $P_t$.

  We proved that $x$ is complete to all principal paths, so to
  $I\cup A \cup A'$.  Let $y\in B\cup B'$.  By definition of a
  template $y$ has two neighbors $a$ and $b$, both in $A$ or both in
  $A'$, that are non-adjacent. Therefore $a$, $b$, $y$ and $x$ form a
  $C_4$, unless $x$ is adjacent to $y$.  This proves that $x$ is
  complete to $B\cup B'$, and so to $V(F_0)$.

  Let $z$ be a vertex of $F_0$ and $z^*\in K_z$.  By
  Lemma~\ref{l:twoNonAdj}, there exists $a,b \in V(F_0)$ such that $z^*a, z^*b \in E(G)$
  and $ab \notin E(G)$, so since there is no $C_4$ in $G$ it should be that $xz^*\in E(G)$.  
  This proves that $x$ is complete to $F_1$.  Hence, $x\in V(F_2)$, a
  contradiction.
\end{proof}

\begin{lemma}
  \label{l:xKaKb}
  Let $a$ and $b$ be two non-adjacent vertices of some principal hole
  $C$ of $F_0$.  If some vertex $x$ of $V(G)\sm V(F_2)$ has neighbors
  in both $K_a$ and $K_b$, then $a$ and $b$ have a common neighbor $c$
  in $C$, $x$ is adjacent to $c$, and $x$ is anticomplete to every
  $K_d$ such that $d\in V(C) \sm \{a, b, c\}$.
\end{lemma}

\begin{proof}
  Let $a^* \in K_a$ and $b^* \in K_b$ be two neighbors of $x$.
  Since $ab\notin E(G)$, by Lemma~\ref{lb:universalevenforblow-up},
  $\{a^*, b^*\} \cup V(C)\sm \{a, b\}$ induces a hole~$C^*$.  Since
  $x$ is adjacent to $a^*$ and $b ^*$, by Lemma~\ref{l:holeTruemperS}, $x$
  has another neighbor $c$ in $C^*$ (and in fact in $C$ since $c\neq
  a^*, b^*$).  If $c$ is not adjacent to $a^*$ and $b^*$, then $x$ is major
  w.r.t.\ $C^*$, so by Lemma~\ref{l:holeTruemperS}, $C^*$ and
 $x$ form a universal wheel.  It follows that $x$ is major w.r.t.\
 $C$, a contradiction to Lemma~\ref{l:xMinor}. 

 We proved that $a$ and $b$ have a common neighbor $c$ in $C$ and that
 $x$ is adjacent to $c$.  Suppose for a contradiction that $x$ has a
 neighbor $d^*\in K_d$ where $d\in V(C) \sm \{a, b, c\}$.  By the same
 argument as above, since $x$ has neighbors in $K_d$ and $K_c$, $c$
 and $d$ must have a common neighbor in $C$, and this common neighbor
 must be $a$ or $b$, say $a$ up to symmetry.  So, $x$ has neighbors in
 $K_d$ and $K_b$ while $b$ and $d$ have no common neighbors in $C$, so
 we may reach a contradiction as above. 
\end{proof}

\subsection{Connecting vertices of $F_1$}

We here explain how lemmas of Subsection~\ref{subs:connectF0} are
extended from $F_0$ to $F_1$.

\begin{lemma}
  \label{l:patatesBlowUp}
  If $u^* \in A^*\cup B^*$ and $v^*\in A'^* \cup B'^*$, then there exists
  in $F_1$ a path $P^*$ of length $\ell-1$, $\ell$ or $\ell +1$ from $u^*$
  to $v^*$ that contains the interior of a principal path.

  More specifically:
  \begin{itemize}
  \item
    If $u^*\in A^*$ and $v^*\in A'^*$, then $P^*$ has length $\ell-1$ or
    $\ell$.
  \item If $u^*\in A^*$ and $v^*\in B'^*$, or if $u^*\in B^*$ and $v^*\in A'^*$,
    then $P^*$ has length $\ell$ or $\ell+1$.
  \item If $u^*\in B^*$ and $v^*\in B'^*$, then $P^*$ has length
    $\ell+1$.
  \end{itemize}
\end{lemma}

\begin{proof}
  Let $u$ and $v$ be such that $u^*\in K_u$ and $v^*\in K_v$.  Let $P$
  be a path in $F_0$ like in Lemma~\ref{lt:cheminsentrepatates}
  from $u$ to $v$ (so $P$ contains the interior of some principal path $Q$). By Lemma~\ref{lb:universalevenforblow-up},
  $\{u^*, v^*\} \cup V(P) \sm \{u, v\}$ induces a path of the same length
  as $P$ that contains the interior of $Q$.
\end{proof}

\begin{lemma}\label{lt:patateBBPblowup}
  If $u^*\in B^*$ and $v^*\in B'^*$, then there exist in $G$ two paths
  $P^*$ and $Q^*$  from $u^*$ to $v^*$ both of length at most $\ell +1$ such
  that $P^*$ (resp.\ $Q^*$) contains the interior of a principal path
  $P$ (resp.\ $Q$), and $P\neq Q$.
\end{lemma}

\begin{proof}
  Let $u$ and $v$ be such that $u^*\in K_u$ and $v^*\in K_v$.  Let
  $P = u\dots v$ and $Q=u\dots v$ be as in the
  conclusion of Lemma~\ref{lt:cheminsentrepatatesBBP}. By
  Lemma~\ref{lb:universalevenforblow-up},
  $\{u^*, v^*\} \cup V(P) \sm \{u, v\}$ and
  $\{u^*, v^*\} \cup V(Q) \sm \{u, v\}$ are the desired paths.
\end{proof}

% suspect path
\begin{lemma}
  \label{l:suspect}
  If some vertex $x$ of $G$ is adjacent to the ends of a path $P$ of length
  at most $\ell+1$ of $G\sm x$, then $x$ is complete to $V(P)$.  
\end{lemma}

\begin{proof}
  Otherwise, a shortest cycle in $G[V(P) \cup \{x\}]$ has length at
  least~4 and at most $\ell + 3$. Since
  $\ell\geq 3$ implies $\ell+3 < 2 \ell +1$, this is a contradiction.
\end{proof}

\subsection{Attaching a vertex to $F_1$}

In this subsection, we show that for all vertices $x$ of $G\sm F_2$,
$N_{V(F_1)}(x)$ is a clique (see Lemma~\ref{l:attachVertex10}).  In
Figure~\ref{f:attach}, several situations where $N_{V(F_1)}(x)$ is not
a clique are represented and we explain informally how they lead to a
contradiction. The first figure is an odd 3-template $F_0$ with its
vertices $w$ and $w'$, and here $F_1=F_0$.  Then, vertex $x_1$ can be
included in $K_{y_1}$, a contradiction to the maximality of $F_1$ (see
Lemma~\ref{l:blowupInB}).  The vertex $x_2$ cannot be included in an
existing blown up clique, but it can be added to $F_0$ to yield a
bigger template (see Lemma~\ref{l:blowupInB}). The vertex $x_3$ can be
added to $K_{u_6}$ (see Lemma~\ref{l:blowupInA}).  The vertex $x_4$
can be added to $K_{u_1}$, but at the expense of modifying the
template (see Lemma~\ref{l:blowupInA}).  The vertex $x_5$ can be added
to $K_{i_3}$ (see Lemma~\ref{l:blowupInI}).

The vertex $x_6$ is kind of pathological because it cannot be
added to any blown-up clique, and does not increase the template.  The idea
for this one is to observe that
$\{x_6\} \cup V(F_0) \sm \{y_1, u_6^+\}$ induces a template and that $y_1$ can be incorporated in
the set $K_{u_6}$ and  $K_{x_6}=K_{u_6^+}\cup \{x_6\}$ (see
Lemma~\ref{l:attachVertex10}).  Note that in this case, we increase
the size of the blowup while decreasing the size of the template. 

In each case, we prove that adding $x$ yields a preblowup of $F_0$, so
that the maximality of $F_1$ is contradicted.  

\begin{figure}
\begin{center}
\includegraphics[width=5cm]{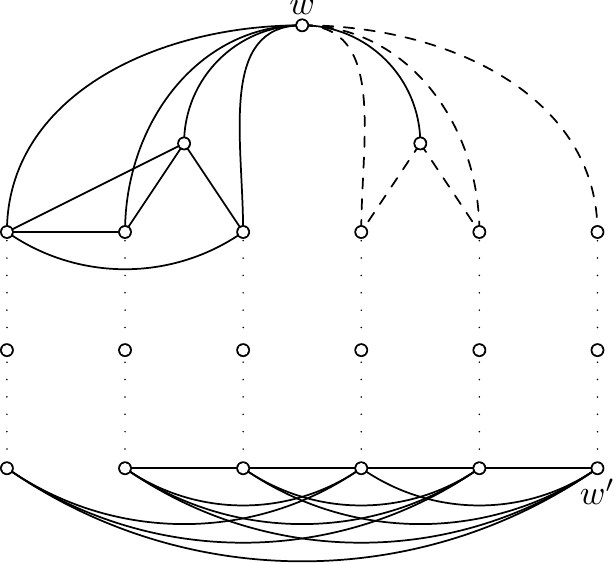}  \hspace{2em}
\includegraphics[width=5cm]{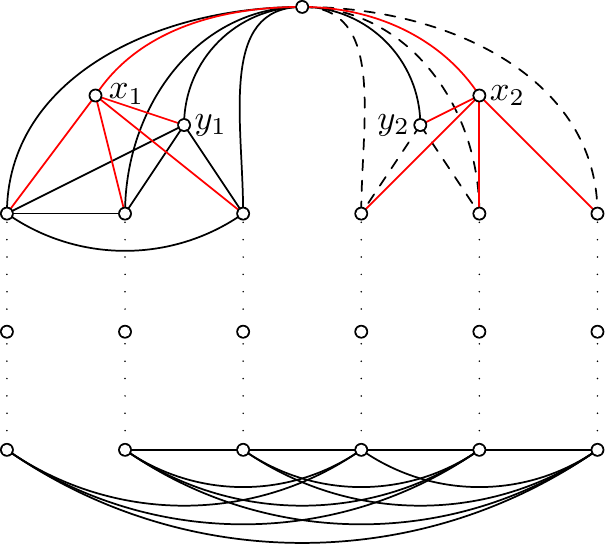}

\rule{0em}{3ex}

\includegraphics[width=5.8cm]{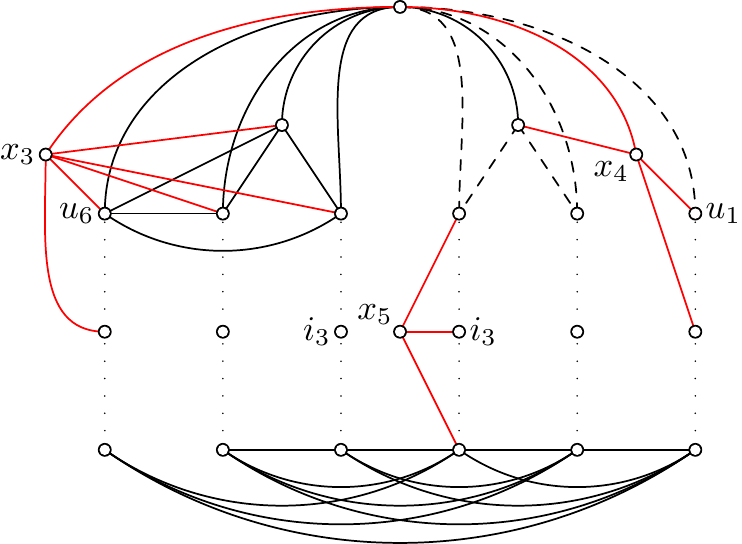}  \hspace{2em}
\includegraphics[width=5.6cm]{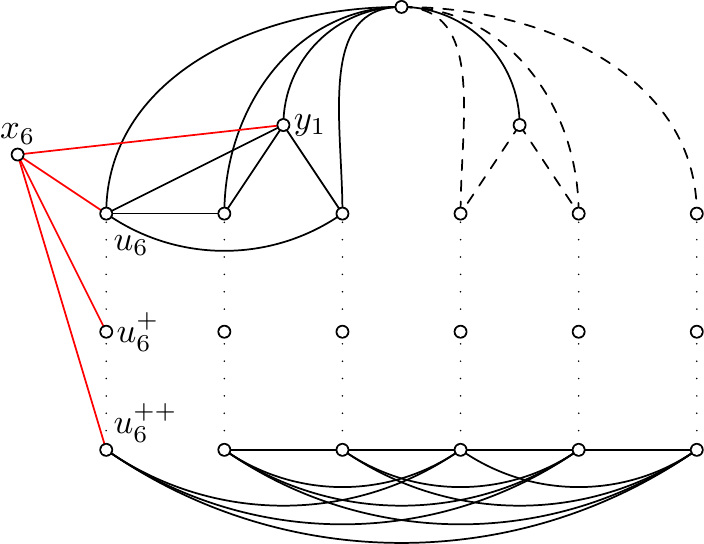}

\end{center}

\caption{Vertices attaching to an odd 3-template\label{f:attach}}
\end{figure}

\begin{lemma}
  \label{l:blowupInB}
  If $x\in G\sm F_2$ has no neighbor in $I^*$, then $N_{V(F_1)}(x)$ is a 
  clique.
\end{lemma}

\begin{proof}
  Suppose for a contradiction that $N_{V(F_1)}(x)$ is not a 
  clique. 

  \begin{claim}
    \label{c:bBnoNeigh}
    We may assume that $N_{V(F_1)}(x) \subseteq A^*\cup B^*$.  
  \end{claim}

  \begin{proofclaim}
    If $x$ has neighbors in both $A^* \cup B^*$ and $A'^* \cup B'^*$,
    then consider a path $P$ as in Lemma~\ref{l:patatesBlowUp} from a
    neighbor of $x$ in $A^* \cup B^*$ to a neighbor of $x$ in
    $A'^* \cup B'^*$.  By Lemma~\ref{l:suspect}, $x$ is complete to
    $V(P)$.  This is a contradiction since $x$ has no neighbor in
    $I^*$.  Hence $x$ does not have neighbors in both $A^* \cup B^*$ and
    $A'^* \cup B'^*$, and our claim follows up to symmetry.
  \end{proofclaim}

  \begin{claim}
    \label{c:bCondPT}
    There exist non-adjacent $a, b\in A$ such that $x$ has neighbors
    in both $K_a$ and $K_b$.  
  \end{claim}
  
  \begin{proofclaim}
        
    By Lemma~\ref{l:onlyKaKb}, since $N_{V(F_1)}(x)$ is not a clique,
    there should exist two non-adjacent vertices $a,b \in V(F_0)$ such
    that $x$ has a neighbor $a^* \in K_a$ and a neighbor
    $b^* \in K_b.$ By~\eqref{c:bBnoNeigh}, $a,b \in A\cup B.$

    If $a, b\in A$, then our conclusion holds, so we may assume that
    $b\in B$.

    If $a\in A$, then since $ab\notin E(G)$, $H_b$ is anticomplete to
    $a$. Let $P^*_a$ be the path induced by
    $\{a^*\} \cup (V(P_a) \sm \{a\})$. Let $v\in H_b$.  We may assume
    that $xv\notin E(G)$ for otherwise our claim holds (with $a$ and
    $v$). Note that since $ab, av \notin E(G),$ by~\eqref{i:bKKanti} of blowup, $a^*b^*, a^*v \notin E(G)$.
    Now, the paths $P_a^*$, $P_v$, $a^*xb^*v$ form a hole of
    length $2\ell +2$, a contradiction.  Hence, we may assume
    $a\in B$. 

    Since $ab\notin E(G)$, by Lemma~\ref{lt:deuxsommetsdeB},
    $\{a\}\cup H_a$ is
    anticomplete to $\{b\} \cup H_b$.  We may assume that $x$ is
    anticomplete to $H_a\cup H_b$ for otherwise we may apply the proofs
    above. Hence, for $u\in H_a$ and $v\in H_b$, the
    two paths $P_u$ and $P_v$ together with the path $ua^*xb^*v$ form
    a hole of length $2\ell+3$. 
  \end{proofclaim}

  Now the sets $K_u$ for all $u\in A\cup A' \cup I$, $B^*\cup \{x\}$
  and $B'^*$ form a preblowup of $F_0$.  All conditions are easily
  checked. In particular $x$ satisfies condition~\eqref{pb:B}
  by~\eqref{c:bBnoNeigh} and~\eqref{pb:BAN} by~\eqref{c:bCondPT}).
  So, by Lemma~\ref{l:preblowup}, $G[V(F_1)\cup \{x\}]$ is a proper
  blowup of some $\ell$-template with $k$ principal paths. This contradicts
  the maximality of $F_1$.
\end{proof}

\begin{lemma}
  \label{l:blowupInA}
  If there exist $x\in V(G)\sm V(F_2)$ and $u\in A$ such that $x$
  has neighbors in both $K_u$ and $K_{u^{+}}$ and is anticomplete to
  $K_{u^{++}}$, then $N_{V(F_1)}(x)$ is a clique.
\end{lemma}

\begin{proof}
  Suppose for a contradiction that $N_{V(F_1)}(x)$ is not a clique.

  \begin{claim}
    \label{c:bAxanti}
    $x$ is anticomplete to $A'^* \cup B'^* \cup (I^*\sm K_{u^+})$.
  \end{claim}
  
  \begin{proofclaim}
    If $x$ has a neighbor $t^*$ in some $K_t$ such that
    $t\in (A' \cup I) \sm \{u^+\}$, then note that $t\neq u^{++}$ by
    assumption.  Let $C$ be a principal hole that contains $t$
    and~$u$.  There is a contradiction to Lemma~\ref{l:xKaKb} because by~\eqref{i:bKKanti} of blowup 
    $u$, $u^+$ and $t$ cannot be consecutive along $C$.
    
    It remains to prove that $x$ is anticomplete to $B'^*$. Otherwise,
    $x$ has a neighbor $t\in B'^*$. Consider a path $P$ from $t$ to
    the neighbor of $x$ in $K_u$ as in Lemma~\ref{l:patatesBlowUp} and let $Q$ be the principal
    path whose interior is contained in $P$.  By
    Lemma~\ref{l:suspect}, $x$ is complete to $V(P)$.  This is a
    contradiction because if $Q=P_u$ then $x$ is anticomplete to
    $K_{u^{++}}$, and if $Q\neq P_u$ then we already proved that $x$
    is anticomplete to $(A'^* \cup I^*) \sm K_{u^+}$.
  \end{proofclaim}

  From here on, $u^*$ and $u^{+*}$ are neighbors of $x$ in
  respectively $K_{u}$ and $K_{u^+}$.  Note that $x$ has a neighbor
  $y^*\in K_y$ for some $y \in A \cup B\sm \{u\}$, for otherwise,
  by~\eqref{c:bAxanti}, $N_{V(F_1)}(x) \subseteq K_u\cup K_{u^+}$ and
  by Lemma~\ref{l:onlyKaKb}, $N_{V(F_1)}(x)$ is a clique, a
  contradiction.

  \begin{claim}\label{c:wstaricomplete}
    If $w \in B$, then $x$ has a neighbor $w^* \in B^*$ that is
    complete to $A^*$.
  \end{claim}
  
  \begin{proofclaim}
    We may assume that $x$ is non-adjacent to $w$, for otherwise by
    condition~\eqref{a:tn} of blowups, we may choose $w^*=w$.  In
    particular $y^*\neq w$.

    We claim that we may assume that $y^*$ has a non-neighbor $v^*$
    such that $v^*\in K_v$, $v\in A$ and $v\neq u$.

    If $y^*\in B^*$, this is because we may assume that $y^*$
    has a non-neighbor $v^*\in A^*$ (so $v^*\in K_v$ for some $v\in A$)
    for otherwise we choose $w^* = y^*$ from the start.  It remains to
    check that $u\neq v$.  This is because if $u=v$, then there exists
    a path $Q$ of length 1, 2 or 3 from $x$ to $v^*$ with interior in
    $K_{u^+}$ (through $xv^*$, $u^+$, $u^{+*}$ or $u^+u^{+*}$).
    Hence, $xQv^*wy^*x$ is a hole of length~4, 5 or~6, a
    contradiction.
    
    If $y^*\in A^*$, then $u^*y^*\in E(G)$ for otherwise,
    $\{x,y^*,w,u^*\}$ induces a $C_4$. By condition~\eqref{i:bKKanti}
    of blowups, $uy \in E(G)$.  It follows that none of $u$ and $y$ is
    isolated in $G[A]$, so the existence of $v^*$ follows from
    Lemma~\ref{l:wProper} that guarantees the existence of isolated
    vertices in $G[A]$ since $w\in B$ by assumption.

    So, our claim is proved.  Note that $xv^*\notin E(G)$ for
    otherwise $\{x,y^*,w,v^*\}$ induces a $C_4$. Now either
    $xy^*wv^*v^+P_vv'w'u'P_uu^{++}u^{+*}x$ is a hole of length
    $2\ell+3$ (in case $u'v'\notin E(G)$) or
    $xy^*wv^*v^+P_vv'u'P_uu^{++}u^{+*}x$ is a hole of length $2\ell+2$
    (in case $u'v'\in E(G)$). In both cases we get a contradiction.
  \end{proofclaim}

  \begin{claim}  
    \label{c:bAneigheq}
    $N_{A}(x) \sm \{u\} = N_{A}(u)$.
  \end{claim}
  
  \begin{proofclaim}
    If there exists $v\in N_{A}(x)\sm N_{A}[u]$, then
    $vP_vv'u'P_uu^{++}u^{+*}xv$ is a hole of length $2\ell$, a
    contradiction.

    Conversely, suppose there exists $v\in N_{A}(u) \sm N_{A}(x)$.  We
    claim that there exists a path $Q$ of length~2 from $x$ to
    some $z\in N_A(u)$  with interior in $(A^* \cup B^*) \sm (K_u \cup K_z)$.  
    
    If $w\in B$, then we may choose $z=v$ and $Q=xw^*z$
    by~\eqref{c:wstaricomplete}.

    Otherwise, $w\in A$.  So, by Lemma~\ref{l:wProper}, $G[A]$
    contains at least two universal vertices.  So, let
    $t\in A\sm \{u, v\}$ be adjacent to $u$ and $v$ (if $u$ and $v$
    are the universal vertices of $G[A]$, $t$ can be any vertex of
    $A\sm \{u, v\}$ and otherwise choose $t$ to be a universal
    vertex).

    If $x$ has a neighbor $t^*$ in $K_t$, then we choose $Q=xt^*v$.
    So, suppose $x$ is anticomplete to $K_t$ (in particular,
    $y\neq t$).  If $x$ has a neighbor $v^*$ in $K_v$, then we choose
    $Q=xv^*t$.  So, suppose $x$ is anticomplete to $K_v$ (in
    particular, $y\neq v$).  Now, by the way we chose $v$ and $t$, one
    of $v$ or $t$ is a universal vertex of $G[A]$ and therefore a
    universal vertex of $G[A^*\cup B^*]$.  So, we may choose $Q=xy^*v$
    or $Q=xy^*t$.

    So, our claim is proved. Hence $z'P_zzQxu^{+*} u^{++}P_uu'w'z'$ is
    a hole of length $2\ell+2$, a contradiction.
  \end{proofclaim}

  \begin{claim}\label{c:EtLaClique}
    $x$ is complete to $K_u$.
  \end{claim}
  
  \begin{proofclaim}
    Suppose there exists $r\in K_u$ such that $rx\notin E(G)$.  We
    claim that $x$ and $r$ have a common neighbor $z$ in $(A^*\cup B^*)
    \sm K_u$.

    If $w\in B$, then $rw^*\in E(G)$ by \eqref{c:wstaricomplete} so we
    may choose $z=w^*$.  If $w\in A$, then by Lemma~\ref{l:wProper},
    some vertex $z\in A\sm \{u\}$ is a universal vertex of $G[A]$, and
    by~\eqref{c:bAneigheq}, $z$ is adjacent to $x$.  So, $z$ exists as
    claimed.

    If $xu^+\in E(G)$ then $\{r,z,u^+,x\}$ induces a $C_4$, a
    contradiction. Hence $xu^+\notin E(G)$.  Now by condition
    \eqref{i:flat} of blowups, either $\{x,z,r,u^{+*}\}$ induces a
    $C_4$ or $\{x,z,r,u^+,u^{+*}\}$ induces a $C_5$.
  \end{proofclaim}

  Now, the sets $K_v$ for all $v\in (A\sm{u})\cup I\cup A'$,
  $K_u\cup \{x\}$, $B^*$ and $B'^*$ form a preblowup of $F_0$.  All
  conditions are easy to check. In particular, $K_u\cup \{x\}$ is a
  clique by \eqref{c:EtLaClique}, conditions \eqref{pb:A},
  \eqref{pb:B} and \eqref{pb:I} follows from \eqref{c:bAxanti},
  condition \eqref{pb:Acomp} from \eqref{c:bAneigheq}, condition
  \eqref{pb:Bw} from \eqref{c:wstaricomplete} and condition
  \eqref{pb:AI} from our assumptions.

  Hence, by Lemma~\ref{l:preblowup} $G[V(F_1)\cup \{x\}]$ is a
  proper blowup of some twinless odd $\ell$-template with $k$
  principal paths that is an induced subgraph of $G$ a contradiction
  to the maximality of $F_1$.
\end{proof}

\begin{lemma}
  \label{l:blowupInI}
  If $x\in V(G)\sm V(F_2)$ has no neighbor in $B^* \cup B'^*$, then
  $N_{V(F_1)}(x)$ is a clique.
\end{lemma}

\begin{proof}
  Suppose for a contradiction that $N_{V(F_1)}(x)$ is not a clique.  By
  Lemma~\ref{l:blowupInB}, $x$ has neighbors in $I^*$.  So $x$ has a
  neighbor in a clique blown up from an internal vertex of some
  principal path $P_v = v\dots v'$.  Let $a$ (resp.\ $b$) be  the
  vertex of $P_v$ closest to $v$ (resp.\ to $v'$) along $P_v$ and such
  that $x$ has a neighbor in $K_a$ (resp.\ $K_{b}$).

  Suppose first that $a=b$ (so $a \in I$).  Then $x$ has a neighbor in
  some $K_y$ with $y \in V(F_0) \sm \{a\}$, and since by assumption
  $x$ has no neighbor in $B^* \cup B'^*$, $y\in A\cup A' \cup I$. So,
  $y$ and $a$ are non-adjacent members of some principal hole.  By
  Lemma~\ref{l:xKaKb}, $x$ has a neighbor in some clique $K_{d}$
  where $d$ is adjacent to $a=b$, a contradiction to $a=b$.  
  
  Suppose now that $ab\in E(G)$. If both $a$ and $b$ are internal
  vertices of $P_v$, then as in the previous paragraph, we may deduce from
  Lemma~\ref{l:xKaKb} that $N_{V(F_1)}(x) \subseteq K_a \cup K_{b}$.  So,
  by Lemma~\ref{l:onlyKaKb}, $N_{V(F_1)}(x)$ is a clique, a
  contradiction.  It follows that at least one of $a$ or $b$ is an
  end of $P_v$.  Up to symmetry, we may assume that $a=v$ and
  $b=v^+$.  Note that $x$ is then anticomplete to $K_{v^{++}}$.  Hence, by
  Lemma~\ref{l:blowupInA}, $N_{V(F_1)}(x)$ is a clique, a contradiction.

  Hence, $a\neq b$ and $ab\notin E(G)$.  So, by Lemma~\ref{l:xKaKb},
  $a$ and $b$ have a common neighbor $u$ in $P_v$. So, $a$, $u$ and
  $b$ are consecutive along $P_v$ (in particular, $u\in I$).

  \begin{claim}
    \label{c:xuInE}
    $x$ is complete to $K_u$. 
  \end{claim}

  \begin{proofclaim}
    Otherwise, let $u^*\in K_u$ be a non-adjacent to $x$. There exists
    a path $Q_a$ of length 2 or 3 from $u^*$ to $x$ with interior in
    $K_a$ (either $xa^*u^*$, or $xa^*au^*$ for some $a^*$ in $K_a$).
    There exists a similar path $Q_b$.  So, $Q_a$ and $Q_b$ form a
    hole of length 4, 5 or 6, a contradiction.
  \end{proofclaim}

  \begin{claim}
    \label{c:BlowIanti}
    $x$ is anticomplete to $V(F_1) \sm (K_a \cup K_u \cup K_b)$. 
  \end{claim}

  \begin{proofclaim}
    This follows from Lemma~\ref{l:xKaKb} and from the fact that $x$ is anticomplete to $B^* \cup B'^*$.
  \end{proofclaim}

  \begin{claim}
    \label{c:xInotIso}
    $x$ has neighbors in each of $K_a$, $K_b$.  
  \end{claim}

  \begin{proofclaim}
    This follows from the definition of $a$ and $b$. 
  \end{proofclaim}

  Now the sets $K_v$ for all $v\in (A \cup A' \cup I) \sm \{u\}$, $K_u\cup \{x\}$,
  $B^*$ and $B'^*$ form a preblowup of $F_0$. All conditions are
  easily checked, in particular $K_u\cup \{x\}$ is a clique
  by~\eqref{c:xuInE}, it satisfies condition~\eqref{pb:I}
  by~\eqref{c:BlowIanti} and condition~\eqref{pb:II}
  by~\eqref{c:xInotIso}.
  
  Hence by Lemma \ref{l:preblowup}, $G[V(F_1)\cup \{x\}]$ is a
  proper blowup of some twinless odd $\ell$-template with $k$
  principal paths that is an induced subgraph of $G$. This contradicts
   the maximality of $F_1$.
\end{proof}

\begin{lemma}
  \label{l:attachVertex10}
  For all vertices $x$ of $G\sm F_2$, $N_{V(F_1)}(x)$ is a clique. 
\end{lemma}

\begin{proof}
  Suppose for a contradiction that $N_{V(F_1)}(x)$ is not a clique.

  \begin{claim}
    \label{l:atMost1internal}
    There exists a principal path $P_u= u\dots u'$ of $F_0$ such that
    $x$ is anticomplete to $I^*\sm \bigcup_{v\in V(P_u)} K_v$.
  \end{claim}

  \begin{proofclaim}
    Otherwise, there exist two principal paths $P$ and $Q$ of $F_0$,
    $a$ in the interior of $P$ and $b$ in the interior of $Q$ such
    that $x$ has neighbors in both $K_a$ and~$K_b$.  Note that $P$ and
    $Q$ are in some principal hole $C$ of $F_0$.  By
    Lemma~\ref{l:xKaKb}, $a$ and $b$ have a common neighbor $c$ in
    $C$.  This contradicts $a$ and $b$ being in the interior of
    distinct principal paths.
  \end{proofclaim}

  \begin{claim}
    \label{l:notBBP}
    We may assume that $x$ has no neighbor in $B'^*$ and has a
    neighbor $y^*\in K_y$ where  $y\in B$.
  \end{claim}

  \begin{proofclaim}
    Suppose that $x$ has a neighbor $u^*\in B^*$ and a neighbor
    $v^*\in B'^*$.  Let $P$ and $Q$ be like in
    Lemma~\ref{lt:patateBBPblowup}.  By Lemma~\ref{l:suspect}, $x$ is
    complete to both $V(P)$ and $V(Q)$.  In particular, $x$ has
    neighbors in the interior of two distinct principal paths, a
    contradiction to~\eqref{l:atMost1internal}. So, up to symmetry, we
    may assume that $x$ has no neighbor in $B'^*$.  Hence, by
    Lemma~\ref{l:blowupInI}, $x$ has neighbors in $B^*$.
  \end{proofclaim}

  \begin{claim}
    \label{c:pathI}
    $x$ is adjacent to $u$, $u^+$ and has a neighbor in
    $K_{u^{++}}$. Moreover, $x$ is anticomplete to
    $(A^*\cup I^* \cup A'^* \cup B'^*) \sm (K_u\cup K_{u^+} \cup
    K_{u^{++}})$.
  \end{claim}

  \begin{proofclaim}
    By Lemma~\ref{l:blowupInB}, $x$ has at least one neighbor in $I^*$
    and by~\eqref{l:atMost1internal}, such a neighbor is in a clique
    blown up from an internal vertex of $P_u$.  So, let $v$ be the
    vertex of $P_u$ closest to $u'$ along $P_u$ such that $x$ has a
    neighbor $v^*\in K_v$. So $v\neq u$ and $v\in A'\cup I$. We set
    $Q = y^*uP_uv$ if $y^*u\in E(G)$ and $Q=y^*wuP_uv$ otherwise.  Let
    $Q^*$ be the path induced by $\{v^*\} \cup (V(Q) \sm \{v\})$ and
    observe that $Q^*$ has length at most $\ell +1$.  By
    Lemma~\ref{l:suspect}, $x$ is complete to $Q^*$.  If
    $v\notin \{u^+, u^{++}\}$, then $x$ has neighbors in at least~4
    cliques blown up from vertices of $P_u$ and this contradicts
    Lemma~\ref{l:xKaKb}.  If $v=u^+$, $x$ is adjacent to $u$ (since
    $x$ is complete to $Q^*$) and anticomplete to $K_{u^{++}}$, so by Lemma~\ref{l:blowupInA},
    $N_{V(F_1)}(x)$ is a clique, a contradiction.  So, $v=u^{++}$,
    meaning that $x$ is adjacent to $u$ and $u^+$, and is anticomplete
    to $I^* \sm (K_{u^+} \cup K_{u^{++}})$
    by~\eqref{l:atMost1internal}.
    
    If $x$ has neighbors in some $K_a$ for $a\in A\sm \{u\}$ then $x$
    and $C_{u,a}$ contradict Lemma~\ref{l:xKaKb}. Hence $x$ is
    anticomplete to $A^*\sm \{K_u\}$.

    By~\eqref{l:notBBP}, $x$ is anticomplete to $B'^*$.  It remains to
    check that $x$ is anticomplete to $A'^*\sm K_{u^{++}}$.  So,
    suppose $x$ has a neighbor $z^*$ in some $K_z$ where
    $z\in A'\sm \{u^{++}\}$. Then a principal hole that contains $z$
    and $u$ contradicts Lemma~\ref{l:xKaKb}.
  \end{proofclaim}

  Let $u^{++*}$ be a neighbor of $x$ in $K_{u^{++}}$ and $P_u^*$ be
  the path induced by $(V(P_u)\sm \{u^{++}\}) \cup \{u^{++*}\}$.

  \begin{claim}
    \label{c:yTwinu}
    For every $z\in B$ such that $x$ is adjacent to some $z^*$ in
    $K_z$ we have $N_A(z) = N_A[u]$ (in particular $N_A(y) = N_A[u]$).
  \end{claim}

  \begin{proofclaim}
    Suppose there exists $v\in N_A(z) \sm N_A[u]$.  By
    condition~\eqref{i:Kcomp} or~\eqref{i:optAB} of blowups,
    $vz^*\in E(G)$.  So, by \eqref{c:pathI},
    $xz^*vP_vv'u'P^*_uu^{++*}x$ is a hole of length $2\ell$, a
    contradiction.  This proves that
    $N_A(z) \subseteq N_A[u]$.  In particular, $u$ has at
    least one neighbor in $H_z$, so by condition~\eqref{a:makeit} of
    templates, $uz\in E(G)$.
    
    Suppose there exists $v\in N_A(u) \sm N_A(z)$ (so $z$ and $v$ are
    not universal vertices of $G[A\cup B]$). By
    condition~\eqref{i:bKKanti} of blowups, $vz^*\notin
    E(G)$. By~\eqref{c:pathI}, $xv\notin E(G)$.  Hence $xd\in E(G)$
    for every universal vertex $d$ of $G[A\cup B]$, for otherwise
    $xz^*dvP_vv'w'u'P_u^*u^{++*}x$ is a hole of length $2\ell +2$.

    Now, by~\eqref{c:pathI} and Lemma~\ref{l:wProper}, $w\in B$. So,
    there exists an isolated vertex $c\in A$. Again
    by~\eqref{c:pathI}, $xc\notin E(G)$ and $xwcP_cc'u'P_u^*u^{++*}x$
    is a hole of length $2\ell$, a contradiction.
  \end{proofclaim}

  \begin{claim}
    \label{c:xanticomAB}
    $N_{F_1}(x)\subseteq K_{u^{++}}\cup K_{u^+}\cup K_u\cup K_y$
  \end{claim}
  
  \begin{proofclaim}
    By \eqref{c:pathI}
    $N_{F_1}(x)\subseteq K_{u^{++}}\cup K_{u^+}\cup K_u\cup
    B^*$. Suppose there exists $z^*\in K_z$ such that $xz^*\in E(G)$
    and $z\in B\sm \{y\}$.  By~\eqref{c:yTwinu}, $N_A(z)= N_A[u]$ and
    $N_A(y)= N_A[u]$.  So, by Lemma~\ref{l:wTwins}, $y$ and $z$ are
    twins of $F_0$, a contradiction.
  \end{proofclaim}

  \begin{claim}
    \label{c:yneqw}
    $y\neq w$. 
  \end{claim}

  \begin{proofclaim}
    If $y=w$, then $w\in B$.  So by Lemma~\ref{l:wProper}, there exist
    isolated vertices in $G[A]$. But by~\eqref{c:yTwinu},
    $N_A(w) = N_A[u]$ so $u$ is a universal vertex of $G[A]$, so
    $G[A]$ has a universal vertex and an isolated vertex, a
    contradiction.
  \end{proofclaim}

  \begin{claim}\label{c:lemevoisindans1}
    $N_{K_y}(x)$ is complete to $N_A[u]$.
  \end{claim}
  
  \begin{proofclaim}
    By \eqref{c:yTwinu}, $N_A(y)=N_A[u]$. The result follows from
    conditions~\eqref{i:Kcomp} and~\eqref{i:optAB} of blowups.
  \end{proofclaim}

  \begin{claim}\label{c:Uneseuleclique1}
    $x$ is complete to $K_{u^+}$.
  \end{claim}
  
  \begin{proofclaim}
    By~\eqref{c:pathI}, $ux\in E(G)$.  Suppose for a contradiction
    that there exists $u^{+*}\in K_{u^+}$ non-adjacent to $x$. By
    condition~\eqref{i:flat} of blowups,
    $u^{+*}u, u^{+*}u^{++}\in E(G)$. Hence $xu^{++}\notin E(G)$ for
    otherwise $\{x,u^{++},u^{+*},u\}$ induces a $C_4$. But now, either
    $\{x,u^{++*},u^{+*},u\}$ induces a $C_4$ (if
    $u^{+*}u^{++*}\in E(G)$) or $\{x,u^{++*},u^{++},u^{+*},u\}$
    induces a $C_5$ (if $u^{+*}u^{++*}\notin E(G)$), a contradiction.
  \end{proofclaim}

  \begin{claim}\label{c:Uneseuleclique2}
    $K_u\cup K_y$ is a clique.
  \end{claim}
  
  \begin{proofclaim}
    Since by~\eqref{c:yTwinu} $N_A(y) = N_A[u]$, $u$ cannot be an
    isolated vertex of $H_y$.  Hence, $uy$ is a solid edge. So, by
    condition~\eqref{i:Kcomp} of blowups, $K_u$ is complete $K_y$. 
  \end{proofclaim}
  
  We define $B_0=B^*\sm N_{K_y}(x)$. 
  
  Now the sets $K_v$ for all $v\in (A\cup I\cup A')\sm \{u,u^+\}$,
  $K_u\cup N_{K_y}(x)$, $K_{u^+}\cup \{x\}$, $B_0$ and $B'^*$ form a
  preblowup of $F_0$. All conditions are easy to check. In particular,
  $K_u\cup N_{K_y}(x)$ is a clique by \eqref{c:Uneseuleclique2},
  $K_{u^+}\cup \{x\}$ is a clique by \eqref{c:Uneseuleclique1},
  conditions \eqref{pb:A}, \eqref{pb:B} and \eqref{pb:I} follows from
  \eqref{c:xanticomAB}, condition \eqref{pb:Acomp} from
  \eqref{c:lemevoisindans1}, condition \eqref{pb:AI} holds because $x$
  is complete to $N_{K_y}(x)$, condition \eqref{pb:II} follows from
  \eqref{c:pathI} and condition \eqref{pb:Bw} holds because
  \eqref{c:yneqw} implies that if $w\in B$ then $w\in B_0$.

  Hence, by Lemma \ref{l:preblowup}, $G[V(F_1)\cup \{x\}]$ is a proper
  blowup of some twinless odd $\ell$-template with $k$ principal paths
  that is an induced subgraph of $G$, a contradiction to the maximality
  of $F_1$.
\end{proof}

\subsection{Attaching a component}

\begin{lemma}
  \label{l:attachComp}
  If $D$ is a connected component of $G\sm F_2$, then $N(D)$ is a
  clique. 
\end{lemma}

\begin{proof}
  Suppose that $N(D)$ is not a clique.  By Lemma~\ref{l:F2-F1Clique},
  $N_{V(F_1)}(D)$ is not a clique.  So, there exist $a$ and $b$ in $D$
  such that $N_{V(F_1)}(a) \cup N_{V(F_1)}(b)$ is not a clique, and a
  path $P$ from $a$ to $b$ in $D$.  We choose $a$ and $b$ subject to
  the minimality of the length of $P$.  By
  Lemma~\ref{l:attachVertex10}, $a \neq b$ (so $P$ has length at
  least~1).

  We set $S^*_a = N_{V(F_1)}(a)$ and $S^*_b= N_{V(F_1)}(b)$.  By
  Lemma~\ref{l:attachVertex10}, $S^*_a$ and $S^*_b$ are both cliques.
  Note that possibly $S^*_a\cap S^*_b\neq \emptyset$. We denote by
  $\text{int}(P)$ the set of the internal vertices of $P$.  We set
  $S^*_{\circ} = N_{V(F_1)}(\text{int}(P))$.

  We set $S_a = \{t\in V(F_0) : S_a^*\cap K_t \neq \emptyset\}$. We
  define $S_b$ and $S_{\circ}$ similarly.  Note that $S_a$ is possibly
  not included in $S_a^*$, and the same remark holds for $S_b$ and
  $S_{\circ}$.

  \begin{claim}
    \label{d:PCL}
    There exist non-adjacent $x^*_a\in S^*_a$ and $x^*_b\in S^*_b$. Moreover,
    for all such $x^*_a$ and $x^*_b$, $x^*_aaPbx^*_b$ is a  path.
  \end{claim}

  \begin{proofclaim}
    The existence of $x^*_a$ and $x^*_b$ follows from the definition of $a$
    and $b$, and $x^*_a a P b x^*_b$ is a path because of the
    minimality of $P$.
  \end{proofclaim}

  \begin{claim}\label{d:attachpath:pas_de_voisin}
    $S^*_a\cup S^*_\circ$ and $S^*_b\cup S^*_{\circ}$ are cliques (in
    particular, $S^*_\circ$ is a (possibly empty) clique of $F_1$ that
    is complete to both $S^*_a \sm S^*_{\circ}$ and
    $S^*_b \sm S^*_{\circ}$).
  \end{claim}
  
  \begin{proofclaim}
    If $S^*_a \cup S^*_{\circ}$ is not a clique, then let $x^*y^*$ be
    a non-edge in $S^*_a \cup S^*_{\circ}$.  Since $S^*_a$ is a clique
    by Lemma~\ref{l:attachVertex10}, we may assume
    $y^*\in S^*_{\circ}$.  By definition of $S^*_{\circ}$, $y^*$ has a
    neighbor in $\text{int}(P)$, and then $x^*,y^*$ and some subpath of $P$
    contradict the minimality of $P$.

    The proof is similar for $S^*_b\cup S^*_{\circ}$.
  \end{proofclaim}

  Note that while $S^*_a\cup S^*_b$ is not a
  clique by assumption, it might be that $S_a\cup S_b$ is a clique
  (for instance when $S_a = \{u\}$, $S_b=\{v\}$ and $uv$ is an
  optional edge of $F_0$).

  \begin{claim}\label{d:pathF0}
    $S_a\cup S_\circ$ and $S_b\cup S_{\circ}$ are cliques of $F_0$ (in
    particular, $S_a$ and $S_b$ are (non-empty) cliques of $F_0$ and
    $S_\circ$ is a (possibly empty) clique of $F_0$ that is complete
    to both $S_a \sm S_{\circ}$ and $S_b \sm S_{\circ}$).
  \end{claim}
  
  \begin{proofclaim}
    If $S_a \cup S_{\circ}$ is not a clique, then let $xy$ be a
    non-edge of $S_a \cup S_{\circ}$.  Since
    $x\in S_a \cup S_{\circ}$, there exists
    $x^*\in K_x\cap (S^*_a \cup S^*_{\circ})$ and
    $y^*\in K_y\cap (S^*_a \cup S^*_{\circ})$.  By
    condition~\eqref{i:bKKanti} of blowups, since $xy\notin E(G)$,
    $K_x$ is anticomplete to $K_y$.  So, $x^*y^*\notin E(G)$, a
    contradiction to~\eqref{d:attachpath:pas_de_voisin}.

    The proof is similar for $S_b\cup S_{\circ}$. 
  \end{proofclaim}

\begin{claim}\label{d:holeTS}
  If a hole $C$ of $F_1$ contains two non adjacent vertices
  $x\in S^*_a$ and $y\in S^*_b$, then $P$ and $C$ form a pyramid
  $\Pi_{C,x,y}$. More specifically, $C$ contains a vertex $z$ such
  that either:

    \begin{itemize}
    \item $S^*_a\cap V(C) = \{x, z\}$, $S^*_b\cap V(C) = \{y\}$ ; the
      apex of $\Pi_{C,x,y}$ is $y$, its triangle is $axz$, and its
      three paths, all of length $\ell$, are the path from $x$ to $y$ in
      $C\sm z$, the path from $y$ to $z$ in $C\sm x$, and the path
      from $a$ to $y$ obtained by adding the edge $by$ to $P$ ; or

    \item $S^*_b\cap V(C) = \{y,z\}$, $S^*_a\cap V(C) = \{x\}$ ; the
      apex of $\Pi_{C,x,y}$ is $x$, its triangle is $byz$, and its
      three paths, all of length $\ell$, are the path between $y$ and
      $x$ in $C\sm z$, the path from $z$ to $x$ in $C\sm y$, and
      the path from $b$ to $x$ obtained by adding the edge $ax$ to
      $P$.
    \end{itemize}
  \end{claim}
  
  \begin{proofclaim} 
  Note that since $S^*_a$ is a clique,
    $S^*_a\cap V(C)$ contains $x$ and at most one other vertex which
    should be adjacent to $x$. The same holds for $S^*_b$ and $y$.

    Let us assume that $S^*_\circ \cap V(C) \neq \emptyset$.  Then
    by~(\ref{d:attachpath:pas_de_voisin}), there exists a unique
    vertex $t \in S^*_\circ \cap V(C)$,
    $S^*_a\cap V(C) \subseteq \{x, t\}$ and
    $S^*_b\cap V(C) \subseteq \{y, t\}$. Hence $C$ and $P$ form a
    proper wheel centered at $t$, a contradiction to
    Lemma~\ref{l:holeTruemperS}.  So,
    $S^*_\circ \cap V(C) = \emptyset$.

    If $a$ and $b$ have a common neighbor $t$ in $C$, then $x$ and $y$
    are the two neighbors of $t$ in $C$ and so, $C$ and $P$ form a
    proper wheel centered at $t$, again a contradiction to
    Lemma~\ref{l:holeTruemperS}. So the neighborhoods of $a$ and $b$
    in $C$ are disjoint.

    From this, we obtain that $C$ and $P$ form a theta, a prism or a
    pyramid. So, by Lemma~\ref{l:holeTruemperS}, $C$ and $P$ form a
    pyramid whose three paths have length $\ell$. This can happen only
    if we are in one of the two cases described in (\ref{d:holeTS}).
  \end{proofclaim}

  \begin{claim}\label{d:pasI}
    $S_a\cap I = S_b \cap I = \emptyset$. 
  \end{claim}

\begin{proofclaim}
  Otherwise, up to symmetry, $S_a\cap I \neq \emptyset$. So, there
  exists a principal path $P_u=u\dots u'$ of $F_0$ whose interior
  intersects $S_a$. By~\eqref{d:pathF0}, $S_a$ is a clique, so
  $1\leq |S_a|\leq 2$ and $S_a\subseteq V(P_u)$.  
  We now break into three cases.

  \medskip
  
  \noindent{\bf Case 1:} $S_b\subseteq V(P_u)$. 
  
  By~\eqref{d:PCL} there exist vertices $x_a$ and
 $x_b$ of $P_u$ such that there exist non adjacent vertices $x^*_a\in S^*_a \cap K_{x_a}$ and $x^*_b\in S^*_b \cap K_{x_b}$.
 
 We first show that there exist such $x_a$ and $x_b$ that are not adjacent. Otherwise, and since $S_a, S_b\subseteq V(P_u)$, we have that  $S_a \cup S_b= \{x_a,x_b\}$. By replacing $x_a$ and $x_b$ by $x^*_a$ and $x^*_b$ in any principal hole $C$ containing $P_u$ we obtain a path $P_C$ of length $2\ell$ and $V(P_C) \cup V(P)$ induces a hole
  of length at least $2\ell+3$, a contradiction. So we may assume that $x_a$ and $x_b$ are not adjacent.

  %  Then $x^*_a$ and $x^*_b$ (as defined in~\eqref{c:PCL}) belong
%  respectively to $K_{x_a}$ and $K_{x_b}$ for some vertices $x_a$ and
%  $x_b$ of $P_u$. 
  Let $C$ be any principal hole of $F_0$ that contains
  $P_u$. By Lemma~\ref{lb:universalevenforblow-up},
  $\{x^*_a, x^*_b\} \cup (V(C) \sm \{x_a, x_b\})$ induces a hole
  $C^*$. Let us apply~\eqref{d:holeTS} to $C^*$, $x^*_a$ and
  $x^*_b$. We obtain that the shortest path in $C^*$ between $x^*_a$
  and $x^*_b$ has length $\ell$. However $x^*_a$ and $x^*_b$ both
  belong to the path of length $\ell-1$, contained in $C^*$, which is
  obtained from $P_u$ by replacing $x_a$ by $x^*_a$ and $x_b$ by
  $x^*_b$, a contradiction.

  \medskip
  
  \noindent{\bf Case 2:} $S_b$ contains a vertex of some principal
  path $P_v$ distinct from $P_u$.  Up to symmetry, since $S_b$ is a
  clique (by~\eqref{d:pathF0}), we assume that $b$ is anticomplete to $K_{v'}$.

  Let $y$ be the vertex of $P_u$ closest to $u'$ such that $a$ has a
  neighbor $y^*\in K_y$.  Let $z$ be the vertex of $P_v$ closest to
  $v$ such that $b$ has a neighbor $z^*\in K_z$.  Possibly $y=u'$
  and $z=v$, but $y\neq u$ since $a$ has a neighbor in
  $I^*$ by assumption, and $z\neq v'$ since $b$ is anticomplete to $K_{v'}$.  In particular, $yz\notin E(G)$  and by
  condition~\eqref{i:bKKanti} of blowups, $y^*z^*\notin E(G)$.

  Let $C$ be the principal hole of $F_0$ that contains $P_u$ and
  $P_v$.  By Lemma~\ref{lb:universalevenforblow-up},
  $\{y^*, z^*\} \cup (V(C) \sm \{y, z\})$ induces a hole $C^*$.
  Applying ~\eqref{d:holeTS} to $C^*$, $y^*$ and $z^*$, we obtain that
   $P$ has length $\ell-1$ and that $y^*$ and $z^*$ are at distance $\ell$ on $C^*$.  Hence $y^*$ and $z^*$ have no common neighbor in $F_0$ and $S_\circ = \emptyset$ by \eqref{d:pathF0}. 
   We denote by $P^*_u$ the path
  obtained from $P_u$ by replacing $y$ by $y^*$ and by $P^*_v$ the
  path obtained from $P_v$ by replacing $z$ by $z^*$. Let $P^*$ be
  the path $vP^*_vz^*bPay^*P^*_uu'$ (in case $z=v$ one should replace $vP^*_vz^*$ by $z^*$, and in case $y= u'$ one should replace $y^*P^*_uu'$ by $y^*$).  The length of $P^*$ is at least
  $\ell+1$.

  Consider now any principal path $P_r$ for $r \in A \sm \{u,v\}$.
  Depending on the adjacencies of $r$ with $u$ and $v$, one of
  $rvP^*u'r'P_rr$ or $rwvP^*u'r'P_rr$ or $rvP^*u'w'r'P_rr$ or
  $rwvP^*u'w'r'P_rr$ (with possibly $u'$ replaced by $y^*$ when
  $u'=y$) is a cycle of length at least $2 \ell +2$ with at most one
  chord that must be $br$ (observe that $ar'$ cannot be an edge since $S_a\subseteq V(P_u)$).  The only possibility which avoids a hole
  of forbidden length is if $z=v$, $y=u'$ and $br, u'r', vr$ are edges
  of $G$.  This proves that $v$ is complete to $A\sm \{u, v\}$ and
  $u'$ is complete to $A' \sm \{u', v'\}$.

  Hence, $G[A]$ has at most one isolated vertex (namely $u$), and
  $G[A']$ has at most one isolated vertex (namely $v'$).  This
  contradicts $(A, B, A', B', I, w, w')$ being a proper
  $\ell$-partition of $F_0$.

  \medskip

  \noindent{\bf Case 3:} we are neither in Case~1 nor in Case~2.

  Since we are not in Case~1, $S_b$ contains a vertex of $F_0\sm P_u$,
  and since we are not in Case~2, this vertex must be in $B\cup B'$.
  Up to symmetry, we assume that $S_b\cap B\neq \emptyset$.  Since
  $S_b$ is a clique  (by~\eqref{d:pathF0}), $S_b\cap (B' \cup A' \cup I) = \emptyset$.  Since
  we are not in Case~2, $S_b\cap (A \sm \{u\}) = \emptyset$.
  Hence, $S_b \subseteq B\cup \{u\}$ and there exists $x\in B\cap S_b$. Let $x^* \in K_x \cap S^*_b$.

Let $u_a$ be the vertex of $S_a$ which is the closest to $u$ in
  $P_u$ and let $u'_a$ be the vertex of $S_a$ which is the closest to
  $u'$ in $P_u$. Notice that, since $S_a$ is a clique (by~\eqref{d:pathF0}), either $u_a=u'_a$ or $u_au'_a$ is an edge. So it may be that $u_a= u$ or $u'_a= u'$ but since $S_a \cap I \neq \emptyset$ we know that $u_a \neq u'$ and $u'_a \neq u$. Let now
  $u^*_a \in K_{u_a} \cap S^*_a$ and $u'^*_a \in K_{u'_a} \cap S^*_a$. We denote by $P^*_u$ the path obtained from $P_u$ by replacing $u_a$ by  $u^*_a$
  and, in case $u_a \neq u'_a$,  by replacing $u'_a$ by $u'^*_a$.
 Notice that if $u^*_a \neq u'^*_a$ then
  $u^*_au'^*_a\in E(G)$ since $S_a^*$ is a clique.
 
 Suppose that $u'_a = u^+$, where $u^+$ is the neighbor of $u$ in $P_u$. Since $H_x$ contains at least two vertices there exists $v \in H_x \setminus \{u\}$. By~\eqref{d:pathF0} and the fact that $P$ contains at least one edge, depending on the adjacency of $u$ and~$v$, one of $aPbx^*vP_vv'u'P^*_uu'^*_aa$ or $aPbx^*vP_vv'w'u'P^*_uu'^*_aa$ is a hole of length at least $2\ell+2$, a contradiction. Hence  from now on, we may assume that $u_a \neq u$ (hence $a$ is not adjacent to $u$) and that if $u_a= u^+$ then $u'_a \neq u_a$. Now  by~\eqref{d:pathF0} we get that $S_\circ = \emptyset$.

 Suppose that $x$ is adjacent to $u$ in $F_0$. Depending on whether $b$ is
adjacent to $u$ not, one of $u^*_aaPbuP^*_uu^*_a$ or
$u^*_aaPbx^*uP^*_uu^*_a$  is a hole, implying that $P$ has length at
 least $\ell$. 
 Let us choose any vertex  $v\in H_x$ distinct from $u$ (since $H_x$ has cardinality at least 2, such a vertex do exist). 
 Then $aPbx^*vP_vv'(w')u'P_uu'^*_aa$ is a hole of length at least $2\ell +2$, a contradiction. Hence, from here on, we may assume that no vertex in
  $B \cap S_b$ is adjacent to~$u$.
 
 %Let $z$ be the
%  vertex in $I \cap S_a$ which is the closest to $u$ in $P_u$ and let
%  $z^* \in K_z \cap S^*_a$.  
%By
%Lemma~\ref{lb:universalevenforblow-up},
%$\{x^*, u^*_a\} \cup V(C) \sm \{x, u_a\}$ induces a hole $C^*$ of $F_1$
%  ($ux^*, vx^*\in E(G)$ by condition~\eqref{i:optAB} of blowups).  The
%  distance between $x^*$ and $u^*_a$ in $C^*$ is at most $\ell-1$ since
% $P_u$ has length $\ell-1$, a contradiction to~\eqref{d:holeTS}
%  applied to $C^*$, $x^*$ and $u^*_a$ (because one path of the pyramid
% obtained by~\eqref{d:holeTS} is either $x^*uP_uu^*_a*$ or $uP_uu^*_a$ when
%  $bu\in E(G)$). Hence, from here on, we may assume that no vertex
%  $x\in B \cap S_b$ is such that $u\in H_x$.
So $x$ is not adjacent to $u$ in $F_0$. Hence $x^*\neq w$, $u \neq w$ and $w \notin S_b$.  Then to avoid a $C_4$ $bx^*wub$, $b$ is not adjacent to $u$
  and $u^*_a a P b x^* w u P^*_u u^*_a$ is a hole implying that $P$ has length at
  least $\ell-1$.  So, for any $v\in H_x$, the hole
  $x^*bPau'^*_a P^*_uu' v' P_v v x^*$ (in case $u'_a=u'$ one should
  replace $u'^*_a P^* u'$ by $u'^*_a$) has length at least $2\ell +2$,
  a contradiction.

\end{proofclaim}

\begin{claim}\label{d:attachpath:pas_trop_prche}
  We may assume that $S_a \subseteq A \cup B$ and
  $S_b \subseteq A' \cup B'$.  
\end{claim}

\begin{proofclaim}
  Otherwise, by~\eqref{d:pasI} and since $S_a$ and $S_b$ are cliques (by~\eqref{d:pathF0}),
  we may assume that $S_a, S_b \subseteq A\cup B$.

  We claim that there exist non-adjacent vertices $x^*\in S^*_a$ and
  $y^*\in S^*_b$, and a path $Q^*$ from $x^*$ to $y^*$ of length at
  least $2\ell - 1$ that forms a hole together with $P$. This is a
  contradiction because it implies that $P$ has length at most~$0$.
  So, to conclude the proof, it remains to prove the existence of
  $Q^*$.

  By~\eqref{d:PCL}, there exist non-adjacent $x^*_a\in S^*_a$ and
  $x^*_b\in S^*_b$. Let $x_a$ and $x_b$ be the vertices of $F_0$ such
  that $x^*_a \in K_{x_a}$ and $x^*_b \in K_{x_b}$.  Note that
  possibly $x_ax_b$ is an edge, but this happens only if $x_ax_b$ is
  an optional edge of $F_0$ (since $x^*_ax^*_b$ is not an edge). We break into three cases.

\medskip

  \noindent{\bf Case 1:} $x_a, x_b\in A.$
  
  Then $x_ax_b\notin E(G)$ (otherwise it would be a solid edge of
  $F_0$), so from the definition of templates, there exists a path $Q$
  of length $2\ell -1$ from $x_a$ to $x_b$ whose interior is in
  $I\cup A'$. By Lemma~\ref{lb:universalevenforblow-up},
  $\{x^*_a, x^*_b\} \cup (V(Q) \sm \{x_a, x_b\})$ induces the path
  $Q^*$ that we are looking for.  Note that $Q^*$ and $P$ form a hole
  by~\eqref{d:attachpath:pas_de_voisin}, our assumption that
  $S_a, S_b \subseteq A\cup B$,  and~\eqref{d:pathF0}.

\medskip

\noindent{\bf Case 2:} $x_a \in A$ and $x_b \in B.$

Whether $x_ax_b$ is an optional edge or a non-edge, an immediate
consequence of the definition of a template is that there exists a
vertex $z\in H_{x_b}$ that is non-adjacent to $x_a$.  We may
furthermore assume that $z\notin S_b$ since else we are in the same
situation as in Case 1. By definition of a template, there exists a
path $Q_0$ of length $2\ell -1$ between $x$ and $z$ whose interior is
in $I\cup A'$. Then $x_bzQ_0x_a$ is a path of length
$2\ell$ and by replacing in this path $x_a$ and $x_b$ by respectively
$x^*_a$ and $x^*_b,$ we obtain by
Lemma~\ref{lb:universalevenforblow-up} a path $Q^*$ of the same
length. Note that $Q^*$ and $P$ form a hole
by~\eqref{d:attachpath:pas_de_voisin}, ~\eqref{d:pathF0}, our assumption that
$S_a, S_b \subseteq A\cup B$, $z\notin S_b$ and $z\notin S_a$ (since
$S_a$ is a clique).
  
\medskip

 \noindent{\bf Case 3:} $x_a, x_b\in B.$

 Then $x_ax_b\notin E(G)$ (otherwise it would be a solid edge of
 $F_0$).  Hence, by Lemma~\ref{lt:deuxsommetsdeB},
 $H_{x_a} \cup \{x_a\}$ is anticomplete to $H_{x_b}\cup \{x_b\}$.  So,
 let $u_a\in H_{x_a}$ and $u_b\in H_{x_b}$, there exists then a path
 $Q_0=u_a \dots u_b$ of length $2\ell-1$ with interior in
 $I\cup A' $.  By Lemma~\ref{lb:universalevenforblow-up},
 $Q^*=x^*_au_aQ_0u_bx^*_b$ is also a path, it is of length $2\ell +1$.
 We may assume that $u_a\notin S_a$ and $u_b \notin S_b$ since else we
 are in the same situation as in Case 2. Now,
 by~\eqref{d:attachpath:pas_de_voisin} and~\eqref{d:pathF0}, $Q^*$ and $P$ form a hole of
 length at least $2 \ell +4$.

   \end{proofclaim}

\begin{claim}\label{d:Pdisjoint}
  $S_\circ=\emptyset$.
\end{claim}

\begin{proofclaim}
  By~\eqref{d:attachpath:pas_trop_prche} and~\eqref{d:pathF0}, if
  $S_\circ\neq \emptyset$, then $\ell=3$, and there exists a unique principal
  path $P_u=u\dots u'$ of $F_0$ such that $S_a=\{u\},$ $S_b=\{u'\}$
  and $S_\circ = \{c\}$ where $c$ is the unique internal vertex of
  $P_u$.  Let $u^*\in K_u\cap S^*_a$, $c^*\in S^*_\circ$ and
  $u'^*\in K_{u'} \cap S^*_b$.  Observe that by \eqref{d:attachpath:pas_de_voisin}
%  from the rules of the
%  blowup, each of $u^*c^*, u'^*c^*$ may be an edge or a non-edge of
%  $G$. By definition, $c^*$ has a neighbor in $\text{int}(P)$.
%
%  We claim that  
  $c^*u^*, c^*u'^*\in E(G)$. 
%   If $c^*=c$ this
%  follows from condition~\eqref{i:flat} of blowups, so suppose
%  $c^*\neq c$. Then $P$, $c$, $c^*$, $u^*$ and $u'^*$ form a theta or
%  a non-twin wheel $W$ centered at $c^*$. So, by
%  Lemma~\ref{l:holeTruemperS}, $W$ is a universal wheel and again
%  $c^*u^*, c^*u'^*\in E(G)$.

  Let $P_v=v\dots v'$ be a principal path distinct from $P_u$ and
  suppose up to symmetry that $uv\in E(G)$.  Now by~\eqref{d:pathF0}, $P_v$, $P$, $u^*$,
  $u'^*$, $w'$ and $c^*$ form a proper wheel centered at $c^*$, a
  contradiction to Lemma~\ref{l:holeTruemperS}.
\end{proofclaim}

\begin{claim}\label{d:length}
  $P$ has length $\ell -1$, or P has length $\ell-2$ and we may assume
  that $S_a\cap A = \emptyset$.
\end{claim}

\begin{proofclaim}
  By~\eqref{d:PCL} and~\eqref{d:attachpath:pas_trop_prche}, consider
  non-adjacent $x^*\in S^*_a$ and $y^*\in S^*_b$ where $x^*\in K_x \cap S^*_a$ and $y^*\in K_y \cap S^*_b$ for $x\in S_a\cap (A\cup B)$ and $y\in S_b \cap (A' \cup B')$.
%  $x\in S_a\cap (A\cup B)$ and $y\in S_b \cap (A' \cup B')$.  Let
%  $x^*\in K_x \cap S^*_a$ and $y^*\in K_y \cap S^*_b$.

  If $x\in A$ and $y\in A'$, then let $C$ be a principal hole that
  contains $x$ and $y$.  By Lemma~\ref{lb:universalevenforblow-up},
  $\{x^*, y^*\} \cup (V(P) \sm \{x, y\})$ induces a hole $C^*$. We may
  apply~\eqref{d:holeTS} to $C^*$, $x^*$ and $y^*$.  It follows that
  $P$ has length $\ell-1$.  By symmetry we may therefore assume from
  here on that $S_a \cap A=\emptyset$.
   
  Let $y'^*$ be a vertex in $S^*_b$ which is the closest to $x^*$ in $F_1$. By Lemma~\ref{l:patatesBlowUp}, there exists a path $Q$ in $F_1$
  from $x^*$ to $y'^*$ of length $\ell$ or $\ell+1$. From our assumption on $y'^*$ we get that $Q$ and $P$
  form a hole (since $S_\circ=\emptyset$
  by~\eqref{d:Pdisjoint}). Therefore, if $Q$ has length $\ell$, then
  $P$ has length $\ell-1$ and if $Q$ has length $\ell+1$, then $P$ has
  length $\ell-2$.
\end{proofclaim}

We may now conclude the proof. 

If $P$ has length $\ell-1$, then we set $A_0= A \cup \{a\}$,
$A'_0= A' \cup \{b\}$ and $I_0= I \cup \text{int}(P)$.  We claim that
$(A_0,B, A'_0,B', I_0)$ is an $\ell$-pretemplate partition of
$G[A_0 \cup B \cup A'_0 \cup B' \cup I_0]$. All conditions are easily
checked to hold (in particular conditions~\eqref{a:antiC}, \eqref{a:antiCP} 
and~\eqref{a:cn} are satisfied because
by~\eqref{d:attachpath:pas_trop_prche}, $a$ (resp.\ $b$) has a
neighbor in $G[A\cup B]$ (resp.\ $G[A' \cup B']$), condition~\eqref{a:interI} holds by~\eqref{d:Pdisjoint} and
conditions~\eqref{a:tn} and~\eqref{a:tnP} hold because they hold in
$F_0$). Then, by Lemma~\ref{th:ptist}, $G$ contains an odd
$\ell$-template with $k+1$ principal paths, a contradiction to the
maximality of~$k$.

So, by~\eqref{d:length}, $P$ has length $\ell-2$ and we may assume that
$S_a \cap A=\emptyset$ and $S_a\cap B \neq \emptyset$ (recall that
by~\eqref{d:attachpath:pas_trop_prche}, $S_a \subseteq A \cup B$ and
$S_b \subseteq A' \cup B'$).  Let us choose $x \in S_a\cap B$ such
that $H_x$ is maximal (note that $x$ is unique because $S_a\cap B$ is
a clique and $F_0$ is twinless).  Let $x^*\in K_x\cap S^*_a$.  We set
$A_0= A \cup \{x^*\}$, $B_0=B\sm S_a$, $A'_0 = A' \cup \{b\}$ and
$I_0= I \cup \text{int}(P) \cup \{a\}$.  Note that the path $x^*aPb$
has length $\ell - 1$ and has interior in $I_0$.
We break into two cases.

\medskip

  \noindent{\bf Case 1:} $b$ has a neighbor  in $A' \cup B'^*$.  

We claim that in that case
$(A_0,B_0, A'_0, B'^*, I_0)$ is an $\ell$-pretemplate partition of
$G[A_0 \cup B_0 \cup A'_0 \cup B'^* \cup I_0]$.    All conditions are
easily checked to hold (in particular condition~\eqref{a:cn} is
satisfied for $A_0\cup B_0$ because if $x=w$, then $x^*$ is complete
to $(A_0\cup B_0) \sm \{x^*\}$, and otherwise, by the maximality of
$H_x$, $w\in A_0\cup B_0,$ condition~\eqref{a:cn} is satisfied for
$A'_0\cup B'^*$ because $b$ has a neighbor in $A' \cup B'^*$ and by the rules of the blowup, 
conditions~\eqref{a:antiC},~\eqref{a:antiCP},~\eqref{a:tn} and~\eqref{a:tnP} hold because they hold in
$F_0$ and by the rules of the blowup). Then, by Lemma~\ref{th:ptist}, $G$ contains an odd
$\ell$-template in with $k+1$ principal paths, a contradiction to the
maximality of~$k$.
\medskip

\noindent{\bf Case 2:} $b$ has no neighbor in $A' \cup B'^*$.

Then, by~\eqref{d:attachpath:pas_trop_prche} there exists $x'^* \in K_{x'} \cap S^*_b$ for some $x' \in A'$. 

Let $A'_1=(A'_0\cup \{x'^*\})\sm \{x'\}$. If $w' \in B'$ we set $B'_1 = \{w'\}$  and else we set $B'_1=\emptyset$. 
We claim that $(A_0,B_0, A'_1,B'_1, I_0)$ is an $\ell$-pretemplate partition of
$G[A_0 \cup B_0 \cup A'_1 \cup B'_1 \cup I_0]$. 

Most conditions are easily checked to hold as in the previous case. Notice that conditions~\eqref{a:cn} and~\eqref{a:tnP} hold because $x'^*$ is by definition adjacent to $b$ and by the rules of the blowup, $G[A'_1\sm \{b\}]$ is isomorphic to $G[A'_0\sm \{b\}]$ and $x'^*$ is adjacent to $w'$.
Then, by Lemma~\ref{th:ptist}, $G$ contains an odd
$\ell$-template in with $k+1$ principal paths, a contradiction to the
maximality of~$k$.
\end{proof}

\subsection{End of the proof}

We may now conclude the proof of Lemma~\ref{l:GPy}. 
 If $G\sm F_1$ is empty, then conclusion~\eqref{c:Pyblowup}
holds.  If $G\sm F_1$ is non-empty and $G\sm F_2$ is empty, then
conclusion~\eqref{c:Pyuniv} holds.  Otherwise, we consider a connected
component $D$ of $G\sm F_2$ and apply Lemma~\ref{l:attachComp}.  We
then see that $G$ has a clique cutset, so
conclusion~\eqref{c:Pycliquecut} holds. 

\section{Proof of Theorem~\ref{th:struct}}

\begin{theorem}
  \label{th:struct}
  Let $\ell \ge 3$ be an integer.  If $G$ is a graph in $\mathcal C_{2\ell +1}$ then
  one of the following holds:

  \begin{enumerate}
  \item\label{c:ring} $G$ is a ring of length $2\ell+1$;
  \item\label{c:blowup} $G$ is a proper blowup of a twinless odd
    $\ell$-template;
  \item\label{c:univ} $G$ has a universal vertex or
  \item\label{c:cliquecut} $G$ has a clique cutset.
  \end{enumerate}
\end{theorem}

\begin{proof}
  By Lemma~\ref{l:holeTruemperS}, $G$ contains no prism no theta and no
  proper wheel.  Also, clearly $G$ contains no $C_4$ and no $C_5$.
  Hence, by Theorem~\ref{th:ring}, we may assume that $G$ contains a
  pyramid for otherwise one of the conclusions~\eqref{c:ring},
  \eqref{c:univ} or~\eqref{c:cliquecut} holds.  The result then follows
  from Lemma~\ref{l:GPy}. 
\end{proof}

\section{Even templates}

\subsection{Even $\ell$-template partitions}

 For an integer $\ell \geq 4$,  an \emph{even $\ell$-template partition} of a graph $G$ is a
 partition of the vertex-set of $G$ into five sets $A$, $B$,
$A'$, $B'$ and $I$ satisfying the following conditions. 
%For an integer $\ell\geq 3$, an \emph{even $\ell$-template} is any
%graph $G$ that satisfies what follows.

%\begin{enumerate}
%\item Choose a threshold graph $J$ on vertex set $\{1, \ldots,k\}$,
%  $k\geq 3$.
%
%\item Choose a laminar hypergraph $\mathcal H$ on vertex set
%  $\{1, \ldots,k\}$ such that :
%  \begin{enumerate}
%  \item\label{a:module} every hyperedge $X$ of $\mathcal H$ is a
%    module of $J$ of cardinality at least~$2$ and
%  \item\label{a:universal} at least one hyperedge $W$ of $\mathcal H$
%    contains all vertices of $\mathcal H$.
%  \end{enumerate}

%\item\label{i:linkP} For each $i \in \{1, \ldots,k\}$, $G$ contains
%  two vertices $v_i $ and $v'_i$ that are linked by a path of $G$ of
%  length $\ell-1$. The $k$ paths built at this step are vertex
%  disjoint and are called the \emph{principal paths} of the odd
%  template.

%\item[I)]\label{eventempd} 
%The set of vertices of $G$ is $V(G)= A \cup A'  \cup B \cup B' \cup I$ where:
\begin{enumerate}

\item \label{eventa1} $A=A_K \cup A_S$ where
 $A _K= \{v_1, \ldots, v_k\}$, $A_S= \{v_{k+1}, \ldots, v_{k+s}\}$ and $k+s \ge 3$.
\item \label{eventa2} $A'=A'_K \cup A'_S$ where 
 $A' _K= \{v'_1, \ldots, v'_k\}$ and $A'_S= \{v'_{k+1}, \ldots, v'_{k+s}\}$.
 \item \label{eventa3}
 For each $i \in \{1, \ldots,k\}$,  $v_i $ and $v'_i$ are linked by a path of $G$ of
  length $\ell-1$ and for each $i \in \{1, \ldots,s\}$,  $v_{k+i} $ and $v'_{k+i}$ are linked by a path of $G$ of
  length $\ell-2$. These $k+s$ paths  are vertex
disjoint and they are called the \emph{principal paths} of the partition.
  \item \label{eventa4} $I$ is the set of all internal vertices of the principal paths, every vertex in $I$ has degree $2$ in $G$.
  \item \label{eventa5}
  Both $A_K$ and $A'_K$ are cliques of $G$ and both $A_S$ and $A'_S$ are stable sets of $G$. For $i \in \{1, \ldots,k\}$ and $j \in \{1, \ldots,s\}$, exactly one of $v_iv_{k+j}$ and $v'_iv'_{k+j}$ is an edge. Furthermore $G[A]$ and hence $G[A']$ are threshold graphs.
  
 % \item \label{eventa6}
  %\begin{enumerate}
  \item \label{eventa6} There exists a laminar hypergraph $\mathcal H$ on vertex set
  $\{v_1, \ldots,v_{k+s}\}$ such that:
  \begin{itemize}
  \item[-] \label{b:module} every hyperedge $X$ of $\mathcal H$ is an anticonnected
    module of $G[A]$ of cardinality at least~$2$ and 
    \item[-] \label{notconunivA} if $G[A]$ is not connected then at least one hyperedge of $\mathcal H$ contains all vertices of $A$.
%  \item\label{a:universal} at least one hyperedge $W$ of $\mathcal H$
 %   contains all vertices of $\mathcal H$.
  \end{itemize}
   \item \label{laminar2} There exists a laminar hypergraph $\mathcal H'$ on vertex set
  $\{v'_1, \ldots,v'_{k+s}\}$ such that:
  \begin{itemize}
  \item[-]\label{c:module} every hyperedge $X'$ of $\mathcal H'$ is a
    module of $G[A']$ of cardinality at least~$2$ and 
    \item[-]  \label{notconunivB} if $G[A']$ is not connected then least one hyperedge of $\mathcal H'$ contains all vertices of $A'$.
     \end{itemize}
    
%  \item\label{a:universal} at least one hyperedge $W$ of $\mathcal H$
 %   contains all vertices of $\mathcal H$.
 % \end{enumerate}

\item\label{b:chooseB}
 $B = \{v_X : X \text{ hyperedge of } \mathcal H \}$,  $B' = \{v'_X : X \text{ hyperedge of } \mathcal H' \}$.
%\end{enumerate}

%Note that by Lemma~\ref{l:ModuleThr}, for every hyperedge $X$ of
%  $\mathcal H$, either $v_X\in B$ or $v'_X\in B'$ (and not both). 

%\item[(II)]\label{eventempe} 
The set of edges of $G$ incident to vertices in $B \cup B'$ is defined as follows:
%\begin{enumerate}
% \item\label{i:BJ} for every $v_i, v_j \in A$, $v_iv_j  \in E(G)$ if and only if $ij \in E(J)$, 
% \item\label{i:BJP} for every $v'_i, v'_j \in A'$, $v'_iv'_j  \in E(G)$ if and only if $ij \notin E(J)$,
\item\label{b:last} for every $v_X, v_Y \in B$, $v_X v_Y \in E(G)$ if and only if $X \cap Y \neq \emptyset$,
\item\label{eventempe4} for every $v'_X, v'_Y \in B'$, $v'_X v'_Y \in E(G)$ if and only if $X \cap Y \neq \emptyset$,
\item\label{b:makeit}  for every $v_i \in A$, $v_X \in B$, $v_i v_X \in E(G)$ if and only if $v_i \in N_{G[A]}[X]$,
\item\label{eventempe6} for every $v'_i \in A'$, $v'_X \in B'$, $v'_i v'_X \in E(G)$ if and only if $v'_i \in N_{G[A']}[X]$.
%\item\label{tempe7} for every $v\in I$, $v$ is incident to exactly two edges (those
%  in its principal path).
%\end{enumerate}
%\item \label{no other} There are no other edges than those mentioned above.
\end{enumerate}

%We then say that $(A_K, A_S, B, A'_K, A'_S, B', I)$ is a weak \emph{$\ell$-template
%  partition} of~$G$.  

The following notation is convenient.

\medskip

\noindent{\bf Notation:}  For every vertex $x\in B$ such that $x=v_X$
  where $X$ is a hyperedge of $\mathcal H$, we set
  $H_x= X$. Similarly, for every vertex $x\in B'$ such
  that $x=v'_X$ where $X$ is a hyperedge of $\mathcal H'$, we set
  $H'_x=X$.  

We now list some properties of even $\ell$-template partitions that follow directly from the definition.

\begin{enumerate} [label=(\roman*)]
\item \label{event:threshold} $G[A]$ 
  and $G[A']$ are threshold graphs such that, - $A=A_K \cup A_S$ and $A'=A'_K \cup A'_S$, - $G[A_K]$ and $G[A'_K]$ are complete graphs having the same number of vertices, - $G[A_S]$ and $G[A'_S]$ are complement of complete graphs having the same number of vertices, - the subgraph of $G[A]$ induced by the edges between $A_K$ and $A_S$ is isomorphic to the complement of the subgraph of $G[A']$ induced by the edges between $A'_K$ and $A'_S$.

\item\label{modanticonn} For all $x\in B$, $H_x$ is a
  module of $G[A]$ and $G[H_x]$ is anticonnected. Also for all
  $x\in B'$, $H'_x$ is a module of $G[A']$ and $G[H'_x]$ is
  anticonnected.
  
  \item\label{event:qTh} $G[B]$ is isomorphic to the line graph of the hypergraph
  $\mathcal H$ on vertex set $A$ and hyperedge set
  $\{H_x : x\in B\}$.  Also $G[B']$ is isomorphic to the line graph of
  the hypergraph $\mathcal H'$ on vertex set $A'$ and hyperedge
  set $\{H'_x : x\in B'\}$.  Hence $G[B]$ and $G[B']$ are 
  quasi-threshold graphs by Theorem~\ref{th:laminar}.

  \item\label{adjAB} There is an edge between $v_i \in A$ and
  $x \in B$ if and only if $v_i \in N_A[H_x]$, and there is an edge
  between $v'_i \in A'$ and $x \in B'$ if and only if
  $v'_i \in N_{A'}[H'_x]$.

\end{enumerate}

%\item\label{quasithres} We may apply Lemma \ref{l:abc} since the same conditions hold, and $G[A \cup B]$ is isomorphic to the line graph of a laminar hypergraph
%  $\mathcal H_A$ on vertex set $A$  whose set of hyperedges contains two types of hyperedges. One type consists in the set of hyperedges in $\mathcal H$ . Each hyperedge of the other type is equal to $N_A[v] \cap \{u \in A : u \leq_{G[A]} v\}$ for some $v \in A$ and we denote it by $H_v$. Similarly, $G[A' \cup B']$ is  isomorphic to the line graph of
%  a laminar hypergraph $\mathcal H_{A'}$ on vertex set $A'$ and whose set of hyperedges is equal to $\{H_x : x \in B'\} \cup \{H_{v'} : v' \in A'\}$.  Hence $G[A \cup B]$ and $G[A' \cup B']$ are 
%  quasi-threshold graphs by Theorem~\ref{th:laminar}. By conditions  \eqref{notconunivA} and \eqref{eventempe} of  even templates, $G[A \cup B]$ and $G[A' \cup B']$ are connected, so by Theorem~\ref{th:laminar} again there exists a vertex $w$ universal in $G[A \cup B]$.  Similarly there exists a vertex $w'$ universal in $G[A' \cup B']$.
%

By the fact that $G[A]$ and $G[A']$ are threshold graphs in even and odd template partitions, by  Properties \ref{modanticonn} and \ref{adjAB} of even template partitions (which are the same as Properties \ref{t:antiMod} and \ref{t:makeIt} of odd templates) and conditions \eqref{notconunivA} and \eqref{notconunivB} of even template partitions, we have the following lemma whose proof is similar to the one of Lemma \ref{l:ABuniv}.

\begin{lemma}
  \label{evenl:ABuniv}
  There exist vertices $w$ and $w'$ that are universal vertices in
  respectively $G[A \cup B]$ and $G[A' \cup B']$.
\end{lemma}

 \begin{lemma}
  \label{l:ThetaPrismTemplate}
  For an integer $\ell\geq 4$, every  theta $\Theta$ such that $\Theta\in
  \mathcal C_{2\ell}$ has an even $\ell$-template partition, every prism $\Sigma$ such that $\Sigma \in
  \mathcal C_{2\ell}$ has an even $\ell$-template partition.
\end{lemma}

\begin{proof}
  Since $\Theta\in \mathcal C_{2\ell}$, its three paths have
  length~$\ell$. Let $x$ and $y$ be the common two extremities of these paths, let $A$ and $A'$ be respectively the set of neighbors of $x$ and $y$ and let $I$ be the set of vertices of $\Theta$ that are not in $A\cup A' \cup \{x, y\}$.  It is easy to verify that $(A, \{x\}, A', \{y\}, I,x,y)$ is an even $\ell$-template partition. Similarly all three paths of $\Sigma$ have length $\ell-1$. Let $A$ be the set of the vertices $v_1,v_2,v_3$ of one of the triangles and $A'$ be the set of vertices $v'_1,v'_2,v'_3$ of the other triangle. It is easy to verify that $(A,\emptyset,A', \emptyset, V(\Sigma) \sm (A\cup A' ), v_1, v'_1)$ is an even $\ell$-template partition. 
\end{proof}

Due to the similarities in odd and even template partitions as pointed out above we also get the following results similar to Lemmas \ref{l:recoverHx}, \ref{lt:deuxsommetsdeB} and \ref{l:Bdeg3}
 with the same proofs.

\begin{lemma}
  \label{evenl:recoverHx}
  If $x\in B$ (resp. $x\in B'$), then $H_x$ (resp.\ $H'_x$) is the
  unique anticomponent of $G[N_A(x)]$ (resp.\ $G[N_{A'}(x)]$) that
  contains at least two vertices.
\end{lemma}

\begin{lemma}
  \label{evenlt:deuxsommetsdeB}
  If $x, y \in B$ (resp.\ $x, y\in B'$) are such that $xy\notin E(G)$,
  then $H_x \cup \{x\}$ (resp.\ $H'_x \cup \{x\}$)  is anticomplete to
  $ H_y\cup \{y\}$ (resp.\ $H'_y \cup \{y\}$).   
\end{lemma}

\begin{lemma}
  \label{evenl:Bdeg3}
  Every vertex of $G$ has degree at least~2 and every vertex of
  $B\cup B'$ has degree at least~3.
\end{lemma}

We may also extend $\mathcal H$ into a hypergraph $\mathcal H_A$ with vertex-set $A$ by adding to its edge-set the hyperedge $H_v=N_A[v] \cap \{u \in A : u \leq_{G[A]} v\}$ for every vertex $v\in A$. Similarly we extend $\mathcal H'$ into a hypergraph $\mathcal H'_{A'}$. The following lemma has the same proof as Lemma \ref{l:abc}.

\begin{lemma}
  \label{evenl:abc}
  $\mathcal H_A$ is a laminar hypergraph and $G[A\cup B]$ is
  isomorphic to its line graph (in particular, $G[A\cup B]$ is a
  quasi-threshold graph and therefore a chordal graph).  A similar
  statements holds for $\mathcal H'_{A'}$ and $G[A'\cup B']$.
\end{lemma}

  \subsection{Even $\ell$-templates}
  We will now need more notion and notation.

Given an even $\ell$-template partition of $G$, we define a hypergraph $\mathcal H_G$ whose vertex set is $\{k+1, k+2 , \ldots k+s\}$ and whose  hyperedges are sets of indices of the vertices of $A_S \cup A_{S'}$ in hyperedges of $\mathcal H_A \cup \mathcal H_{A'}$.
More formally, $E(\mathcal H_G)= E_A \cup E_{A'}$ where 

- $E_A= \{\{i: v_i  \in H \cap A_S\} : H \text{ hyperedge of } \mathcal H_A  \text{ s.t. } H \cap A_S \ne \emptyset \}$ 

- $E_{A'}= \{\{i: v_i  \in H \cap A'_S\} : H \text{ hyperedge of } \mathcal H_{A'}   \text{ s.t. } H \cap A'_S \ne \emptyset\}$. 

Notice that $\mathcal H_G$ may contain distinct hyperedges containing the same set of vertices. 

\bigskip

A circular sequence $\mathcal C=(j_1,e_1,j_2,..., j_t, e_t,j_1)$, where the $j_i$'s are distinct vertices of $\mathcal H_G$ and the $e_i$'s are distinct hyperedges of $\mathcal H_G$, is said to be a {\it hyper cycle} of length $t$ of $\mathcal H_G$ if 
\begin{itemize}
\item each $j_i$ belongs to $e_{i-1}$ and $e_i$  (where $e_{t+1}=e_1)$ and to no other hyperedge of $\mathcal C$,
%\item any two non-consecutive edges are disjoint,
\item any two distinct hyperedges of $\mathcal C$ that belong both to $E_A$ or both to $E_{A'}$ are disjoint.
%\item each $j_i \ge k+1$, that is the corresponding vertices of $G$ are in $A_S \cup A_{S'}$.
%\item if $e_i \in H_A$ then $v_{j_i}$ and $v_{j_{i+1}}$ are non adjacent vertices of $G[A]$,
%\item if $e_i \in H_{A'}$ then $v'_{j_i}$ and $v'_{j_{i+1}}$ are non adjacent vertices of $G[A']$.
\end{itemize}

We notice that by definition each $e_i$ contains $j_i$ and $j_{i+1}$ and no other vertex of $\mathcal C$, hence any two consecutive hyperedges $e_i$ and $e_{i+1}$ of $\mathcal C$  have a non empty intersection and none is included in the other.  So, since $\mathcal H_A$ and $\mathcal H_{A'}$ are both laminar hypergraphs, the hyperedges of $\mathcal C$ belong alternately to $E_A$ and $E_{A'}$. In particular the length of $\mathcal C$ is even.

%  An \emph{alternating cycle of the projection $p(G)$} of the weak even template $G$ is a cycle $C$ of $b(G)$ having length at least $4$ and satisfying the following constraints : the labels of the edges of $C$ are all distinct,  the labels of  consecutive edges of $C$ are alternatively vertices of $b(B)$ and vertices of $b(B')$,  the set of labels of the edges of $C$ in $b(B)$ (respectively $b(B')$) is a stable set of $b(B)$ (respectively $b(B')$).
\bigskip
    
For an integer $\ell\geq 3$ and a graph $G$, a \emph{strong even $\ell$-template partition} of $G$ is an even $\ell$-template  partition $(A, A', B, B', I)$ of $G$, such that $\mathcal H_G$ contains no hyper cycle of length greater than $2$. 

A graph $G$ which has a strong even $\ell$-template partition is called an \emph{even $\ell$-template}. We will denote by $(A,B,A',B',I,w,w')$ an even $\ell$-template $G$ with strong even $\ell$-template partition $(A,B,A',B',I)$ such that $w$ is a universal vertex of $G[A \cup B]$ and $w'$ is a universal vertex of $G[A' \cup B']$ (by Lemma \ref{evenl:ABuniv}, such $w$ and $w'$ do exist).

%either $|A_S|=|A'_S| \le 3$ or the projection $p(G)$ contains no alternating cycle. 

\medskip
\begin{lemma}
  \label{l:HinTemplate_even}
  Let $G$ be an even $\ell$-template and $(A, A', B, B', I)$ be a strong even $\ell$-template partition of  $G$. 
  Every hole $H$ of $G$ contains two principal paths of $G$ and either
  \begin{itemize}
\item[-] these two principal paths have length $\ell-1$ and they induce $H$,
\item[-]  one principal path has length $\ell-1$, the other has length $\ell -2$ and $H$ contains exactly one more vertex which belongs to $A_K \cup B \cup A'_K \cup B'$,
\item[-] both principal paths have length $\ell-2$, $H$ contains exactly two more vertices, one in $A_K\cup B$ and the other in $A'_K \cup B'$. 
\end{itemize}
In all cases $H$ has length $2\ell$, and so $G \in \mathcal C_{2\ell}$.
 \end{lemma}

\begin{proof}
Let $H$ be a hole of $G$.  By Lemma \ref{evenl:abc}, $G[A \cup B]$ contains no $P_4$ and no $C_4$, so $H$ cannot contain only vertices of
  $A\cup B$, and similarly, it cannot contain only vertices of $A'\cup B'$. So $H$
  must contain vertices of some principal path, and hence it contains an even number of
  principal paths. 
  
In case $H$ contains two principal paths of length $\ell - 1$ then their endpoints are linked by two disjoint edges (by the definition of an even template partition). So these paths form a hole of length $2\ell$ and $H$ contains no other principal path.

Assume now that  $H$ contains exactly one path $P_u=u  \ldots u'$ of length $\ell - 1$ and at least one path $P_v=v \ldots v'$ of length $\ell-2$, for some $u \in A_K$ and $v \in A_S$.  By the definition of an even template partition, there exists  exactly one edge between $P_u$ and any principal path of length $\ell-2$, hence $H$ cannot contain three  such paths and should hence contain exactly one, namely $P_v$.   Up to symmetry, we may assume that  $uv \notin E(G)$ and then the hole $H$ is
  made of the path $uP_uu'v' P_vv$ of
  length $2\ell -2$ and a path $Q = u\dots v$ of $G[A\cup B]$.  By
  Lemma \ref{evenl:abc}, $Q$ has length at most~2  and since $uv\notin E(G)$, we get that $Q=ubv$ for some $b\in A\cup B$.  So $H$ has length $2\ell$ and since $A_S$ is a stable set in $G[A]$ we have $b\in A_K\cup B$ as claimed.
  
 It remains to consider the case where all principal paths contained in $H$ have length $\ell-2$. Assume that $H$ contains $t\ge 3$ principal paths $P_1, P_2, \ldots P_t$ associated to some $v_{j_1}, v_{j_2}, \ldots v_{j_t}\in A_S$ (note that since $t$ is even, then $t \ge 4$). By the definition of an even template partition, there exists no edge connecting these paths. Since by Lemma \ref{evenl:abc} there is no $P_4$ in $G[A \cup B]$ and in $G[ A' \cup B']$, without loss of generality we may assume that $H=v_{j_1}P_1v'_{j_1}b_1v'_{j_2}P_2v_{j_2}b_2v_{j_3}, \ldots v'_{j_t}P_tv_{j_t}b_tv_{j_1}$ where $b_i$'s with odd index belong to $A'_K\cup B'$ and those with even index belong to $A_K\cup B$.  
 Hence, to each $b_i$ we may associate the hyperedge $e_i$ of $\mathcal H_G$ corresponding to $H_{b_i} \in \mathcal H_A \cup \mathcal H_{A'}$. We claim that $\mathcal C=(j_1,e_1,j_2,..., j_t, e_t,j_1)$ is a hyper cycle of $\mathcal H_G$. Assume there exists in $\mathcal C$ an hyperedge $e_l \neq e_{i-1}, e_i$ such that $j_i \in e_l$. So, up to symmetry, $v_{j_i} \in H_{b_l}$ and then by the definition of an even template, $H$ would contain a chord $v_{j_i} b_l$, a contradiction.  So the first condition for being a hyper cycle is satisfied by $\mathcal C$. Assume now, up to symmetry, that there exist $e_i,e_l \in H_A$ that are not disjoint. Then by the definition of an even template partition, $H$ would contain a chord $b_ib_l$, a contradiction again. Hence the second condition should be satisfied and $\mathcal C$ is a hyper cycle of $\mathcal H_G$, a contradiction to the fact that $(A, A', B, B', I)$ is a strong even $\ell$-template partition of a graph $G$. 
 So, we may now conclude that $H$ contains exactly two principal paths of length $\ell-2$ and two more vertices $b_1 \in A_K \cup B$ and $b_2 \in A'_K \cup B'$. \end{proof}

\subsection{Even $\ell$-pretemplates}

 For every integer $\ell \geq 4$, an \emph{even $\ell$-pretemplate partition} of a
graph $G$ is a partition of the vertex-set of $G$ into five sets $A=A_K \cup A_S$, $B$,
$A'= A'_K \cup A'_S$, $B'$ and $I$ that satisfy the following conditions.

\begin{enumerate}[label= \arabic*.]
\item\label{b:antiC} $N(B)\subseteq A$ and $N(A\cup B) \subseteq I$. 
\item\label{b:antiCP} $N(B')\subseteq A'$ and $N(A'\cup B') \subseteq I$. 
\item\label{b:cardinality1} $|A_K|=|A'_K|=k$, $A_K = \{v_1, \dots, v_k\}$ and 
  $A' _K= \{v'_1, \dots, v'_k\}$.
  \item\label{b:cardinality2} $|A_S|=|A'_S|=s$, $A_S = \{v_{k+1}, \dots, v_{k+s}\}$ and
  $A' _S= \{v'_{k+1}, \dots, v'_{k+s}\}$ are stable sets of $G$ where  $k+s \ge 3$.
\item\label{b:path} For every $i\in \{1, \dots, k+s\}$, there exists a
  unique path $P_i$ from $v_i$ to $v'_i$ whose interior is in $I$.
\item \label{b:interI} Every vertex in $I$ has degree 2 and lies on a path from $v_i$
  to $v'_i$ for some $i\in \{1, \dots, k\}$.
\item\label{b:sp} All paths $P_1$, \dots, $P_k$ have length $\ell-1$, all paths $P_{k+1}$, \dots, $P_{k+s}$ have length $\ell-2$.
\item\label{b:cn} $G[A\cup B]$ and $G[A'\cup B']$ are both connected
  graphs. 
\item\label{b:tn} Every vertex of $B$ is in the interior of a path of
  $G[A\cup B]$ with both ends in $A$.
\item\label{b:tnP} Every vertex of $B'$ is in the interior of a path of
  $G[A'\cup B']$ with both ends in $A'$.
\end{enumerate}

We then say that $(A, A', B, B', I)$ is an even \emph{$\ell$-pretemplate
  partition} of~$G$.  

%It is easy to check that when $\ell\geq 3$, the five first elements of
%every $\ell$-partition of $G$ is an $\ell$-pretemplate partition.  The
%condition on the connectivity of $G[A\cup B]$ and $G[A'\cup B']$
%follows Lemma~\ref{l:ABuniv}.  The condition~\eqref{a:tn} follows from
%the fact for every $x\in B$, $H_x$ contains two non-adjacent vertices,
%so a vertex $x\in B$ lies on a path of length~2 with ends in $A$ and
%condition~\eqref{a:tnP} holds similarly. Conversely, we prove the
%following lemma (it is important to note that $\ell\geq 3$).
 
\begin{lemma}
  \label{evenpret}
  Let $\ell\geq 4$ be an integer and $G$ be a graph of $\mathcal C_{2\ell }$. Any even $\ell$-pretemplate partition $(A, B, A', B', I)$ of $G$ is a strong even $\ell$-template partition of $G$. \end{lemma}

\begin{proof}

%  Let $(A_K, A_S, B, A'_K, A'_S, B', I)$ be an even $\ell$-pretemplate partition of
%  $G$, we set $A=A_K \cup A_S$ and $A'=A'_K \cup A'_S$.
%  
\begin{claim}
  \label{evenl:comp}
  For all distinct $i, j \in \{1, \dots, k\}$, $v_iv_j\in E(G)$ and
  $v'_iv'_j\in E(G)$. In particular, $A_K$ and  $A'_K$ are cliques of $G$.
\end{claim}

\begin{proofclaim}
  Assume that at least one of $v_iv_j, v'_iv'_j \notin E(G)$. By condition \ref{b:cn} in the definition of an even pretemplate partition
  there exist  a path between $v_i$ and $v_j$ in $ G[A \cup B] $ and a
  path between $v'_i$ and $v'_j$ in $G[A' \cup B']$.  Together with $P_i$ and  $P_j$, these paths form a hole of length at
  least $2\ell +1$, a contradiction.\end{proofclaim}

\begin{claim}
  \label{evenexaone}
  For all distinct $i \in \{1, \dots, k\}$ and $j \in \{k+1, \dots, k+s\}$, exactly one of $v_iv_j$ and
  $v'_iv'_j$ is an edge of  $G$. 
\end{claim}

\begin{proofclaim}
  By condition \ref{b:cn} in the definition of an even pretemplate partition
  there exist  a path between $v_i$ and $v_j$ in $ G[A \cup B] $ and a
  path between $v'_i$ and $v'_j$ in $G[A' \cup B']$.  Together with the principal paths $P_i$ (of length $\ell-1)$ and  $P_j$ (of length $\ell-2)$, these paths form a hole which will be of length $2\ell$ if and only if one of the paths is of length $1$ and the other path is of length $2$.
 \end{proofclaim}

\begin{claim}
  \label{even:length2}
  Every path of $G[A\cup B]$ with both ends in $A$ is of length at
  most~2. The same holds for a path of $G[A'\cup B']$ with both ends in $A'$.
\end{claim}

\begin{proofclaim}
  Assume on the contrary that there exists a path $P$ of length at least $3$ in $G[A\cup B]$ with both ends $v_i, v_j$ in
  $A$. Then by \eqref{evenl:comp} and condition \ref{b:sp} of an even pretemplate partition, at least one of $P_i,P_j$ is of length $\ell-2$, say $P_j$. If $P_i$ is of length $\ell-1$ then by \eqref{evenexaone} $v'_i v'_j \in E(G)$ and $P$, $P_i$, $P_j$ would induce a hole of length at least~$2\ell +1$, a contradiction. Hence $P_i$ and $P_j$ are  both of length $\ell-2$ and by conditions \ref{b:cardinality2}  and \ref{b:cn} of an even pretemplate partition, $P$, $P_i$, $P_j$ and any
  path between  $v'_i$ and $v'_j$ form a hole of length at least~$2\ell +1$, a
  contradiction again. 
  The proof for $G[A'\cup B']$ is similar.
  \end{proofclaim}

\begin{claim}
  \label{evenl:A}
  $G[A]$ is a threshold graph. The same holds for  $G[A']$.
\end{claim}

\begin{proofclaim}
  $G[A]$ is obviously $C_4$-free.  By \eqref{evenl:comp} and condition \ref{b:cardinality2} of an even pretemplate partition, $A$ is partitioned into a clique and a stable set of $G$,  so by Theorem \ref{th:split} $G[A]$ is $2K_2$-free. By~\eqref{even:length2},
  $G[A]$ is $P_4$-free.  So $G[A]$ is ($P_4$, $C_4$, $2K_2$)-free and
  is therefore a threshold graph. The proof for $G[A']$ is similar.
\end{proofclaim}

We now study the structure of $G[B]$ (respectively $G[B']$) and its relation with $G[A]$ (respectively $G[A']$).

\begin{claim}
  \label{evenl:twoInA}
  For every vertex $x \in B$, $G[N_A(x)]$ has a unique anticonnected
  component of size at least~2. The same holds for $G[N_{A'}(x)]$ when $x$ is any vertex in $B'$.

\end{claim}

\begin{proofclaim}
  By condition \ref{b:tn} of an even pretemplate partition, $x$ is in the interior of a path
  $P= v_i\dots v_j$ of $G[A\cup B]$ with both ends in $A$.
  By~\eqref{even:length2}, $P$ has length~2, so $x$ is adjacent to $v_i$
  and $v_j$.  Hence $G[N_A(x)]$ has an anticonnected component of size
  at least~2.  It is unique, for otherwise $G[A]$ would contain a $C_4$.
\end{proofclaim}

For all $x\in B$ (respectively $B'$), we define $H_x$ to be the anticonnected component
of $G[N_A(x)]$ (respectively $G[N_A'(x)]$) of size at least~2 whose existence follows
from~(\ref{evenl:twoInA}).

\begin{claim}
  \label{evenl:PR}
  For every $x$ in $B$ (respectively $B'$), $H_x$ is a module of $G[A]$ (respectively $G[A']$).  
  %For every $x$ in $B'$, $H'_x$ is a module of $G[A']$. 
\end{claim}

\begin{proofclaim}
 Let $x \in B$. If the claim does not hold, since $H_x$ is by definition anticonnected, there
  exist $v_h \in A \sm H_x$ and non-adjacent $v_i, v_j\in H_x$ such
  that $v_hv_i\in E(G)$ and $v_hv_j\notin E(G)$.  Note that
  $xv_h\notin E(G)$ because otherwise, $v_h$ would be in $H_x$. Then $v_hv_ixv_j$ is a path of length $3$, a contradiction to \eqref{even:length2}. The similar proof holds for $x$ in $B'$.
\end{proofclaim}

\begin{claim}
  \label{evenl:edgexy}
  If $xy$ is an edge of $G[B]$ or $G[B']$, then $H_x \subseteq H_y$ or
  $H_y\subseteq H_x$.  
\end{claim}

\begin{proofclaim}
  Let $xy$ be an edge of $G[B].$ Up to symmetry, we may assume that $N_A(x) \subseteq N_A(y)$, for
  otherwise vertices $v_i\in N_A(x)\sm N_A(y)$ and  $v_j\in
  N_A(y)\sm N_A(x)$ either form a $C_4$ with $x$ and $y$ or a contradiction to \eqref{even:length2}.

  By~(\ref{evenl:twoInA}), $G[N_A(y)]$ has only one anticonnected component of size
  at least~2, namely $H_y$.  Since $H_x$ is anticonnected, has size
  at least~2 and is included in $N_A(y)$, it must be included in
  $H_y$. The similar proof holds for an edge  $xy$ of $G[B'].$
 \end{proofclaim}

 \begin{claim}
   \label{evenl:nonedgexy}
   If $x$ and $y$ are non-adjacent vertices of $B$ or $B'$ then $H_x$ and
   $H_y$ are disjoint.
\end{claim}

\begin{proofclaim}
  On the contrary, suppose that $x$ and $y$ are nonadjacent vertices of $B$ but
  there exists a vertex $v \in H_x \cap H_y$.  Since $H_x$ is anticonnected and
  of size at least~2, there exists $v_i\in H_x$ non-adjacent to $v$. Note that
  $v_iy\notin E(G)$, for otherwise $x$, $y$, $v_i$ and $v$ form a $C_4$.
  Similarly, there exists a vertex $v_j\in H_y$ that is non-adjacent to $v$ and
  to $x$.  If $v_iv_j\in E(G)$, then $\{x, y, v, v_i, v_j\}$ induces a $C_5$, a
  contradiction.  Otherwise, $\{x, y, v, v_i, v_j\}$ induces a $P_5$, a
  contradiction to \eqref{even:length2}. The proof for $x,y \in B'$ is similar.
\end{proofclaim}

We are now ready to define the hypergraphs $\mathcal H$ and $\mathcal H'$.  For every
$x\in B$ (respectively $B'$), we defined a set $H_x\subseteq A$ (respectively $H_x\subseteq A'$).   From~\eqref{evenl:edgexy}
and~\eqref{evenl:nonedgexy}, the sets $H_x$ for $x\in B$ form a laminar
hypergraph $\mathcal H$ (with vertex set $A$).  Symetrically, the
sets $H_x$ for $x\in B'$ form a laminar hypergraph $\mathcal H'$
(with vertex set $A'$).

 \begin{claim}
   \label{even6B}
   If $G[A]$ is not connected then at least one hyperedge of $\mathcal H$ contains all vertices of $A$.
\end{claim}

\begin{proofclaim} 
Assume $G[A]$ is not connected. By  \eqref{evenl:A}, $G[A]$ is a threshold graph, and then by Theorem~\ref{th:threshold} it contains an isolated vertex $v_i$. By the definition of an even pretemplate partition, $G[A \cup B]$ is connected and $|A|\geq 3$, so there exists a path $P$ in
$G[A\cup B]$ from $v_i$ to a vertex $u\in A\sm\{v_i\}$. By \eqref{even:length2} and since $v_i$ has no neighbor in $A$, we have that $P=uyv_i$ where $y\in B$. So, $H_y$ contains $v_i$. We may
therefore consider the hyperedge $W$ of $\mathcal H$ that contains $v_i$
and that is inclusion-wise maximal w.r.t.\ this property.
If there exists $v_j\in A \sm W$, since
$v_jv_i\notin E(G)$, we deduce as above that $\mathcal H$ has a
hyperedge $Z$ that contains $i$ and $j$.  Because of $v_j$,
$Z\subseteq W$ is impossible; because of $v_i$, $W\cap Z=\emptyset$ is
impossible; and because of the maximality of $W$, $W\subsetneq Z$ is
impossible.  Hence, $W$ and $Z$ contradict $\mathcal H$ being
laminar.  So $W = A$.
%, as claimed in
%condition~\eqref{a:universal} of templates.

%The proof is the same as the one at the end of the proof of Lemma \ref{th:ptist}.
  \end{proofclaim}
  
  At this point we can easily verify that $(A, B, A', B', I)$ is an even $\ell$-template partition of $G$:

- conditions \eqref{eventa1} \eqref{eventa2} \eqref{eventa3} and \eqref{eventa4} of an even template partition are satisfied because of conditions \ref{b:cardinality1}  \ref{b:cardinality2} \ref{b:path} \ref{b:interI} and \ref{b:sp} in the definition of an even pretemplate partition,

- \eqref{eventa5} of an even template partition is a consequence of condition \ref{b:cardinality2} of an even pretemplate partition and \eqref{evenl:comp}, \eqref{evenexaone}, \eqref{evenl:A},

- \eqref{eventa6}, \eqref{laminar2}  and \eqref{b:chooseB} of even template partition is a consequence of \eqref{evenl:twoInA}, \eqref{evenl:PR}, \eqref{evenl:edgexy}, \eqref{evenl:nonedgexy}, 

- \eqref{b:last} and \eqref{eventempe4} of an even template partition comes from \eqref{evenl:edgexy} and \eqref{evenl:nonedgexy},

- \eqref{b:makeit} and \eqref{eventempe6} of an even template partition follows easily from our previous results :
By definition of $H_x$, for every $x$ in $B$,
$N_A(x) \subseteq N_A[H_x]$.  Suppose that there exists
$u\in N_A[H_x]\sm N_A(x)$.  Since by~(\ref{evenl:PR}) $H_x$ is a module, 
it follows from Lemma~\ref{l:nxcc} that $u$ is complete to $H_x$, so
$x$ and $u$ together with two non-adjacent vertices from $H_x$ induce
a $C_4$, a contradiction. Hence, $N_A(x) = N_A[H_x]$ and
condition~\eqref{b:makeit} of an even template partition is satisfied. The proof for \eqref{eventempe6} is similar.

%By \eqref{l:A} $G[A]$ is a threshold graph, and so as in the case of odd templates we can build an extension
%$\mathcal H_A$ of $\mathcal H$ by adding more hyperedges: for every
%vertex $v\in A$, we add the hyperedge
%$H_v = N_A[v] \cap \{u \in A : u\leq_{G[A]} v\}.$ We extend similarly $\mathcal H$ by $\mathcal H_{A'}$.
%By Lemma \ref{l:abc}, $\mathcal H_A$ is a laminar hypergraph, $G[A\cup B]$ is
%  isomorphic to its line graph and in particular $G[A\cup B]$ is a
%  quasi-threshold graph. Similarly, $G[A'\cup B']$ is a
%  quasi-threshold graph.

It remains to prove the following.

%Let $\mathcal H$ be the hypergraph whose
%vertex set is $\{1, \dots, k\}$ and such that
%$H\subseteq \{1, \dots, k\}$ is a hyperedge of $\mathcal H$ if and
%only if $H=\{i : v_i\in H_x\}$ for some $x\in B$ or
%$H = \{i : v'_i\in H'_x\}$ for some $x\in B'$.

%\begin{claim}
%  \label{c:global}
%  The hypergraphs $\mathcal H$ and $\mathcal H'$ are laminar.
%\end{claim}
%
%\begin{proofclaim}
%  If $\mathcal H$ is not laminar, then there exist
%  $X, Y\in E(\mathcal H)$ such that $X\sm Y$, $Y \sm X$ and
%  $X\cap Y$ are all non-empty.  Since $\mathcal H_B$ and
%  $\mathcal H_{B'}$ are both laminar, there exists $x\in B$ such that
%  $H_x = \{v_i : i\in X\}$ and $y\in B'$ such that
%  $H'_y = \{v'_i : i\in Y\}$.
%
%  We set $H_y = \{v_i: i\in Y\}$.  Note that $H_x\sm H_y$,
%  $H_y\sm H_x$ and $H_x\cap H_y$ are all non-empty. Also, because of
%  the properties of $H'_y$ and by \eqref{l:comp}, $G[H_y]$ is connected (because $G[H'_y]$
%  is anticonnected) and $H_y$ is a module of $G[A]$.
%
%
%  Since $G[H_x]$ is anticonnected, there exist non-adjacent vertices
%  $u\in H_x\sm H_y$ and $v\in H_x\cap H_y$.  Since $G[H_y]$ is
%  connected, there exists a path from $v$ to $t \in H_y\sm H_x$ and we may assume that $vt$ is an edge.  
%  Since $H_y$ is a module of $G[A]$, $ut\notin E(G)$.  So, $t$ is adjacent
%  to $v$ and non-adjacent to $u$. This contradicts $H_x$ being a
%  module of $G[A]$.  
%\end{proofclaim}

\begin{claim}
$\mathcal H_G$ contains no hyper cycle of length greater than $2$.
\end{claim}

\begin{proofclaim}
Suppose on the contrary that $\mathcal C=(j_1,e_1,j_2,..., j_t, e_t,j_1)$ is a hyper cycle of $\mathcal H_G$ of length $t > 2$. Without loss of generality, each $e_i$ with odd index belongs to $E_A$, and since it contains $j_i$ and $j_{i+1}$ such that $v_{j_i},v_{j_{i+1}}\in A_S$ we have that $e_i$ is the set of indices of the vertices of $H_{b_i}$ for some $b_i \in A_K \cup B$. Similarly each $e_i$ with even index belongs to $E_{A'}$, contains $j_i$ and $j_{i+1}$ such that $v'_{j_i}$ and $v'_{j_{i+1}} \in A'_S$ and corresponds to $H_{b'_i}$ for some $b'_i \in A'_K \cup B'$. Hence $C=v_{j_1}b_1v_{j_2}P_{v_{j_2}}v'_{j_2}b'_2 v'_{j_3}P_{v_{j_3}}v_{j_3} \ldots v'_{j_t}b'_t v'_{j_1}P_{v_{j_1}}v_{j_1}$ is a cycle in $G$ of length $t\ell>2 \ell$.

By definition of an even template partition, there are no edges between two distinct $P_i$'s contained in $C$, and no edges between vertices in $A_K \cup B$ and vertices in $A'_K \cup B'$. So the only  chords $C$ could contain are:

- $b_ib_j$ for some distinct $i,j \in \{1,2, \ldots, t\}$ of same parity, or 

- $v_{j_l}b_i$ for some odd $i \in \{1,2, \ldots, t\}$ and $l\neq i, i+1$ (where $t+1=1$), or 

- $v'_{j_l}b_i$ for some even $i \in \{1,2, \ldots, t\}$ and $l\neq i, i+1$ (where $t+1=1$). 

\medskip

The first case is impossible since, by definition of a hyper cycle, the hyperedges $H_{b_i}$ are all disjoint. The two other cases cannot occur because else $j_l$ would be contained in the hyperedge $e_i$ where $i\neq j_{l-1}, j_l$, a contradiction to the definition of a hyper cycle.

So $C$ is a hole which has length more than $2\ell$, a contradiction to $G\in \mathcal C_{2\ell }$.  \end{proofclaim}

This ends the proof of Lemma \ref{evenpret}.
\end{proof}

Notice that as a corollary of Lemma \ref{evenpret} we get that if a graph $G$ in $\mathcal C_{2\ell }$ has an even $\ell$-pretemplate partition then it is an even $\ell$-template and any  of its even $\ell$-pretemplate partition is a strong even $\ell$-template partition.

\bigskip

We will need the following analog of Lemma \ref{l:wTwins}.

\begin{lemma}
  \label{evenl:wTwins}
  Let $G=(A, B, A', B', I,w,w')$ be an even $\ell$-template.  Two vertices $x$ and $y$ of $G$ are twins if
  and only if $x, y\in B$ and $H_x=H_y$, or $x, y\in B'$ and
  $H'_x = H'_y$.
\end{lemma}

\begin{proof}
  If $x, y\in B$ and $H_x=H_y$, or $x, y\in B'$ and $H'_x=H'_y$, then
  $x$ and $y$ are obviously twins. 
  
We claim that for all $x\in A\cup I\cup A'$, there exist two
  vertices $a,b\in N_G(x)$ such that $N[a]\cap N[b]=\{x\}$. If
  $x\in I$, choose $a$ and $b$ to be the only two neighbors of $x$. If
  $x\in A$, then let $a$ be the neighbor of $x$ in $I$ and let $b=w$ if $x\neq w$,  else let $b$ be any
  vertex of $A\sm \{x\}$ (we remind that $w$ is by definition a vertex universal in $G[A \cup B]$). In all cases, by definition of an even
  template partition, $N_G[a]\cap N_G[b]=\{x\}$.  The proof is similar when
  $x\in A'$.  So, $x$ has no twin in $G$.
\end{proof}

An even $\ell$-partition $(A, B, A', B', I)$ of an even
$\ell$-template $G$ is \emph{proper} if all universal vertices of $G[A \cup B]$ (respectively $G[A'\cup B']$) are in $B$ (respectively $B'$). 

\begin{lemma}
  \label{l:wProperEven}
  Every twinless even $\ell$-template $G=(A, B, A', B', I,w,w')$ with a proper even $\ell$-partition satisfies the following:
  \begin{itemize}
\item $w\in B$, $w$ is the unique universal vertex of $G[A \cup B]$ and
  $G[A]$ contains at least one isolated vertex $u \in A_S$,

\item $w'\in B'$, $w'$ is the unique universal vertex of $G[A' \cup B']$ and
  $G[A']$ contains at least one isolated vertex $v' \in A'_S$,

\item $\vert A_S \vert \ge 2$.
  \end{itemize}
\end{lemma}

\begin{proof}
Since the partition is proper, all universal vertices of $G[A \cup B]$ belong to $B$ and since $G$ is twinless, $B$ contains a unique vertex which is universal in $G[A \cup B]$. Hence $w$, which is by definition of $G$ a universal vertex of $G[A\cup B]$, belongs to $B$ and $G[A]$ contains no universal vertex since else, by the definition of an even template, this vertex would be universal in $G[A \cup B]$.
Then, by Theorem \ref{th:laminar}, $G[A]$ contains at least one isolated vertex say $u$. The proof for $w'$ is similar and there exists a vertex $v'$ isolated in $G[A']$. It remains to prove the third assertion. Assume that $\vert A_S \vert \le1$. Then, since $A$ contains at least three vertices, it should be that $A_S=\{u\}.$ Similarly we obtain that $A'_S=\{v'\}$, so $u'=v'.$ Then by the definition of an even template $v'$ is not isolated in $G[A']$, a contradiction.
\end{proof}

\begin{lemma}
  \label{l:TProperEven}
  For every integer $\ell \geq 4$, every even $\ell$-template
  $G$ admits an even $\ell$-partition which is proper.
\end{lemma}

\begin{proof} Let $G=(A, B, A', B', I)$ be a strong even $\ell$-partition of $G$. By Lemma \ref{evenl:ABuniv} we know that each of $G[A\cup B]$ and $G[A'\cup B']$ contains at least one universal vertex. If the lemma does not hold, we may assume up to symmetry, that there exists at least one vertex $w$ universal in $G[A\cup B]$ which is in $A$. Then from the definition of an even template it is clear that we may choose $w$ in $A_K$. 

We denote by $w_+$ the neighbor of $w$ on the principal path $P_w$ of $G$ in the partition $(A, B, A', B', I)$. Since $\ell \ge 4$, we have $w_+\in I$.
Consider now the partition $(\mathbb A_K, \mathbb A_S, \mathbb B, \mathbb A'_K, \mathbb A'_S, \mathbb B', \mathbb I)$ of $V(G)$ where   $\mathbb A_K=A_K \sm \{w\}$,  $\mathbb A_S=A_S \cup \{w_+\}$,  $\mathbb B= B \cup \{w\}$,  $\mathbb A'_K=A'_K \sm \{w'\}$, $\mathbb A'_S=A'_S \cup \{w'\}$,  $\mathbb B'=B'$, $\mathbb I=I \sm \{w_+\}$. We set $\mathbb A=\mathbb A_K \cup \mathbb A_S$ and $\mathbb A'=\mathbb A'_K \cup \mathbb A'_S$.

  Since $(A, B, A', B', I)$ is an even $\ell$-partition of $G$, it is clear that $\mathbb A_K$ and $\mathbb A'_K $ are cliques of the same cardinality. Since the only edge in $G$ between $w_+$ and $A$ is $ww_+$ we have that $\mathbb A_S=A_S \cup \{w_+\}$ is a stable set and since $w_+$ has no neighbor in $\mathbb A$, any universal vertex of $G[\mathbb A \cup \mathbb B]$ is in $\mathbb B$. As $w$ is complete to $A_S$ we get from the definition of an even template that $\mathbb A'_S=A'_S \cup \{w'\}$ is a stable set of the same cardinality as $A_S \cup \{w_+\}$. Remark that $|\mathbb A|= |A| \ge 3$. Furthermore there exists a unique path $P_{w_+}= w_+P_w w' $ between $w_+$ and $w'$ whose interior is in $\mathbb I$. This path has length $\ell-2$. It is also important to notice that $w$ is in the interior of a path $vww_+$ for any $v \in A \sm \{w\}$. Every other vertex of $\mathbb B$ is in the interior of a path of $G[\mathbb A \cup \mathbb B]$ with both ends in $\mathbb A$, path which is the same as in the initial partition since obviously this path did not contain $w$. With all these observations it is easy to conclude that $\mathbb P=(\mathbb A, \mathbb B, \mathbb A', \mathbb B', \mathbb I)$ fulfills all conditions to be a pretemplate partition of $G$. By Lemma \ref{l:HinTemplate_even} we know that $G\in \mathcal C_{2\ell }$.  So by Lemma \ref{evenpret} and the fact that $w_+$ is isolated in $G[\mathbb A]$, $\mathbb P$ is an even $\ell$-partition of $G$ such that no vertex of $ \mathbb A$ is universal in $G[\mathbb A \cup \mathbb B]$.

If $\mathbb A'$ contains no universal vertex of $G[\mathbb A' \cup \mathbb B']$ then $\mathbb P$ is a proper even $\ell$-partition of $G$, else a proof similar to the one above allows to obtain from $\mathbb P$ a proper even $\ell$-partition of $G$.
\end{proof}

\begin{lemma}
  For every integer $\ell\geq 4$, every even $\ell$-template $G$
  contains a prism or a theta.
\end{lemma}

\begin{proof}
By Lemma \ref{l:TProperEven}, $G$ admits a proper even $\ell$-partition $(A, B, A', B')$ of $G$ and by definition of an even template, $G$ contains three vertices $v_1, v_2$ and $v_3$ in $A$ and the corresponding
  vertices $v'_1, v'_2$ and $v'_3$ in $A'$. Considering all kinds of repartition of $v_1, v_2$ and $v_3$ in $A_K$ and $A_S$, all kinds of attachment between them respecting the rules of the partition and the fact that $B$ (respectively $B'$) contains a vertex universal in $G[A_K \cup A_S \cup B]$ (respectively $G[A'_K \cup A'_S \cup B']$) it is easy to verify that in each case we obtain either a prism or a theta.
\end{proof}

\subsection{Blowups and holes} \label{evendef_blowup}
We may define flat, solid and optional edges in even  $\ell$-template partitions exactly as in the case of odd $\ell$-template partitions (see subsection \ref{def_edges_temp}). In the following we will also use the notion of blowup and preblowup of an even template $\ell$-partition with the same definition as for  an odd $\ell$-template partition (see subsections \ref{def_blowup} and \ref{def_preblowup}).

The following lemma can be proved similarly as Lemma \ref{l:atMost1}. 

\begin{lemma}
  \label{evenl:atMost1}
  A hole $C$ in a blowup of a twinless even $\ell$-template contains at
  most one vertex in each blown up clique.
\end{lemma}

We remark that Lemmas \ref{c:FXsubY}, \ref{l:NuSolid} and \ref{l:removeFr} are valid for any odd or even template as they rely only on the definition of solid and optional edges. Hence, the  following  lemma has the same proof as Lemma  \ref{blowoddtemp} of the odd case, except that we use Lemma \ref{l:HinTemplate_even} instead of Lemma \ref{l:HinTemplate}, and Lemma \ref{evenl:atMost1} instead of Lemma \ref{l:atMost1}.

\begin{lemma} \label{bloweventemp}
  In a blowup $G^*$ of a twinless even $\ell$-template $G$, every hole
  has length $2\ell$.
\end{lemma}

Recall that to blowup (resp. preblowup) a template, one needs to
first fix an $\ell$-partition.  If this partition is proper, the
blowup (resp. preblowup) is \emph{proper}. 
Recall also that when $G^*$ is a preblowup of a template
$G$, the \emph{domination score} of $G$ w.r.t.\ $G^*$ is defined as (where $N$ refers to the neighborhood in $G^*$):
$$
s(G, G^*) = \sum_{x\in A\cup A' \cup I}\left| \left\{ x^*\in K_x :
    N[x^*] \subseteq N[x] \right\} \right|
$$
\begin{lemma}
  \label{l:preblowup_even}
  Let $\ell \ge 3$ and let $G^*$ be a proper preblowup of an even $\ell$-template
  with $k\geq 3$ principal paths. If $G^*\in \mathcal C_{2\ell}$,
  then $G^*$ is a proper blowup of a twinless even $\ell$-template $G$
  with $k$ principal paths (in particular, $G$ is an induced subgraph
  of $G^*$).
\end{lemma}

\begin{proof}
Among all the induced subgraphs of $G^*$ that are even
  $\ell$-templates and for which $G^*$ is a proper preblowup,
  we suppose that $G$ is one that maximizes $s(G, G^*)$. We denote
  by $(A, B, A', B', I, w, w')$ the proper $\ell$-partition of $G$ that
  is used for its preblowup and by $(A^*, B^*, A'^*, B'^*, I^*)$ the
  corresponding partition of the vertices of $G^*$.
   
  \begin{claim}\label{cb:chapeau_even}
    There exist vertices $w^* \in B^*$ and $w'^*\in B'^*$ that are complete to
    respectively $A^*$ and $A'^*.$   \end{claim}

  \begin{proofclaim}
    Since the partition is proper, $w\in B$ and then by condition~\eqref{pb:Bw} of a preblowup, there
    exists $w^*\in B^*$ that is complete to $A^*$.
    
    The proof of the statement about $w'^*$ is similar.  
  \end{proofclaim}

  \begin{claim}
    \label{c:principB}
    For every principal path $P_u=u\dots u'$ of $G$ and $u^*\in K_u$,
    there exists in $G^*$ a path $P_{u^*}$ 
    from $u^*$ to some ${u}'^*\in K_{u'}$ whose interior is in
    $\bigcup_{x\in I\cap V(P_u)} K_x$ and whose length is equal to the length of $P_u$. Moreover, the interior of $P_{u^*}$ is
    anticomplete to $V(G^*) \sm \bigcup_{v\in V(P_u)} K_v$.
  \end{claim}
  
  \begin{proofclaim}
    The existence of a path of same length as $P_u$ from $u^*$ to some $u'^*\in K_{u'}$ whose
    interior is in $\bigcup_{x\in I\cap V(P)} K_x$ follows from conditions~\eqref{pb:AI},~\eqref{pb:I} and ~\eqref{pb:II},  
    of preblowup.  The statement about its interior follows from
    conditions~\eqref{pb:A}, \eqref{pb:B} and~\eqref{pb:I} of
    preblowup.
  \end{proofclaim}

  \begin{claim}\label{cb:Anticomplet_in_AEven}
    For all $u,v\in A$ such that $uv\notin E(G)$, $K_u$ is anticomplete
    to $K_v$. A similar statement holds for $A'$.
  \end{claim}
  
  \begin{proofclaim}
  Suppose that there exist $u^*\in K_u$ and $v^*\in K_v$ such that
    $u^*v^*\in E(G^*)$. By condition~\eqref{pb:Acomp} of preblowup,
    $u\neq u^*$ and $v\neq v^*$.
    
 Since $uv\notin E(G)$ then at least one of $u,v$ is in $A_S$, say $v \in A_S$.
 
  Consider first the case where $u\in A_K.$
        The principal paths $P_u = u\dots u'$ and
    $P_v=v\dots v'$ have length respectively $\ell-1$ and $\ell-2$. Denote by $u^+$ the neighbor
    of $u$ in $P_u$ and by $v^+$ the neighbor of $v$ in $P_v$. By
    %condition \eqref{i:linkP} of templates, 
   property~\eqref{eventa5} of an even template,
    $u'v'\in E(G)$. Hence
    $uP_uu'v'P_vvv^*u^*u$ is a cycle $C$ of length $2\ell+1$. By
    conditions \eqref{pb:A} and \eqref{pb:Acomp} of preblowup, the only
    possible chords in $C$ are $u^*u^+$ and $v^*v^+$. Assume that $u^*u^+\in E(G^*)$.
    Let $P_{v^*}$ be a path of length $\ell -2$ from $v^*$ to $v'^{*}$
    as defined in~\eqref{c:principB}. Since $v'^{*}\in K_{v'}$ and
    by~\eqref{pb:Acomp} of the preblowup applied to $A'$, $v'^{*}u'\in E(G^*)$. So
    $v^*P_{v^*}v'^{*}u'P_uu^+u^*v^*$ is a hole of length $2\ell-1$, a
    contradiction. Hence $v^+v^*$ should be a chord of $C$. Let then $P_{u^*}$ be a path of length $\ell -1$ from $u^*$ to $u'^{*}$ as defined in~\eqref{c:principB}. Since $u'^{*}\in K_{u'}$ and
    by~\eqref{pb:Acomp} of the preblowup applied to $A'$, $u'^{*}v'\in E(G^*)$. So
    $u^*P_{u^*}u'^{*}v'P_vv^+v^*u^*$ is a hole of length $2\ell-1$, a
    contradiction again.
    
    It remains to consider the case where both $u$ and $v$ are in $A_S$.  Let $P_u = u\dots u'$ and $P_v=v\dots v'$ be principal paths of length $\ell-2$. Denote by $u^+$ the neighbor
    of $u$ in $P_u$ and by $v^+$ the neighbor of $v$ in $P_v$.   By property~\eqref{eventa5} of an even template,
    $u'v'\notin E(G)$, hence
    $uP_uu'w'^*v'P_vvv^*u^*u$ is a cycle $C$ of length $2\ell+1$. By conditions \eqref{pb:A} and \eqref{pb:Acomp} of preblowup, the only possible chords in $C$ are $u^+u^*$ and $v^+v^*$.
    
     Assume without loss of generality that $u^+u^* \in E(G^*)$. Let $P_{v^*}$ be a path of length $\ell -2$ from $v^*$ to $v'^{*}$
    as defined in~\eqref{c:principB}. Since $v'^{*}\in K_{v'}$ and
    by~\eqref{pb:Acomp} applied to $A'$, $v'^{*}u'\in E(G^*)$. So
    $v^*P_{v^*}v'^{*}w'^*u'P_uu^+u^*v^*$ is a hole of length $2\ell-1$, a
    contradiction. 
    %condition \eqref{i:linkP} of templates, 
     
    The result for $A'$ holds symmetrically.
  \end{proofclaim}

  \begin{claim}\label{cb:Complet_in_A}
    For all $u,v\in A$ such that $uv\in E(G)$, $K_u$ is complete to
    $K_v$.  A similar statement holds for $A'$.
  \end{claim}
  
  \begin{proofclaim}
    Suppose that there exist $u^*\in K_u$ and $v^*\in K_v$ such that
    $u^*v^*\notin E(G^*)$. Let $P_{u^*} = u^*\dots u'^*$ and
    $P_{v^*} = v^*\dots v'^*$ be defined as
    in~\eqref{c:principB}. Observe that $u'^{*}\in K_{u'}$ and
    $v'^{*}\in K_{v'}$. 
    
    Since $uv\in E(G)$ then at least one of $u,v$ is in $A_K$, say $u \in A_K$.
 
  Consider first the case where $v\in A_K$ too. Then both $P_{u^*}$ and $P_{v^*}$ have length $\ell-1$. 
  If $u'^*v'^*\in E(G^*)$ then $u^*P_{u^*} u'^*v'^*P_{v^*} v^*w^*u^*$ is a hole of length $2\ell+1$ and else $u^*P_{u^*} u'^*w'^*v'^*P_{v^*} v^*w^*u^*$ is a hole of length $2\ell+2$, so it should be that $v\in A_S$. Then $P_{u^*}$ and $P_{v^*}$ have length respectively $\ell-1$ and $\ell-2$. Moreover by property~\eqref{eventa5} of an even template,
    $u'v'\notin E(G^*)$, so by \eqref{cb:Anticomplet_in_AEven} $u'^*v'^*\notin E(G^*)$. Now $u^*P_{u^*} u'^*w'^*v'^*P_{v^*} v^*w^*u^*$ is a hole of length $2\ell+1$, a contradiction again.

%  By property~\eqref{eventa5} of an even template,
%    $u'v'\in E(G)$. 
%  
%  Furthermore $u'v'\notin E(G)$ by
%      property~\eqref{t:threshold} of templates.  Hence,
%    by~\eqref{cb:Anticomplet_in_A}, $u'^{*}v'^{*}\notin E(G^*)$.
%
%    We claim that there exists a vertex $a\in (A\cup B) \sm \{u,v\}$ that is
%    adjacent to both $u^*$ and $v^*$. If $w^* \neq u,v$, then
%    by~\eqref{cb:chapeau} and condition~\eqref{pb:AI}, we may choose
%    $a=w^*$.  Otherwise, up to symmetry, $w^*=u$.  Since the
%    $\ell$-partition of $G$ is proper, by Lemma~\ref{l:wProper}, $A$
%    contains a universal vertex $x$ distinct from $w^*=u$.  If
%    $x \neq v$, we set $a=x$. If $x=v$, then both $u$ and $v$ are
%    universal vertices of $G[A]$ and we may choose for $a$ any vertex
%    of $A\sm \{u, v\}$. This proves our claim.
%    
%    Now, $au^*P_{u^*}u'^{*}w'^*v'^{*}P_{v^*}v^*a$ is a hole of length
%    $2\ell +2$, a contradiction.  The result for $A'$ holds
%    symmetrically.
  \end{proofclaim}
  
  \begin{claim}\label{cb:Inested}
    For all $u\in I$ and $u_1, u_2\in K_u$, either
    $N[u_1]\subseteq N[u_2]$ or $N[u_2]\subseteq N[u_1]$.
  \end{claim}

  \begin{proofclaim}
    Otherwise, there exist $x^*_1\in N[u_1]\setminus N[u_2]$ and
    $x^*_2\in N[u_2] \setminus N[u_1]$.  Note that
    $x^*_1x^*_2\notin E(G^*)$ for otherwise,
    $\{x^*_1, x^*_2, u_1, u_2\}$ induces a $C_4$.  It follows that $x^*_1$ and $x^*_2$ 
    belong respectively to distinct cliques $K_{x_1}$ and $K_{x_2}$, 
    where $x_1$ and $x_2$ are the two neighbors of $u$
    along some principal path $P _v= v\dots v'$  of $G$.  Because of
    $x^*_1$, $x^*_2$ and condition~\eqref{pb:II} of
    preblowup, there exists a path $P^*$  from some
    $v^*\in K_v$ to some $v'^*\in K_{v'}$ whose interior is  in
    $\bigcup_{x\in I\cap V(P_v)} K_x$ which contains $u_1$ and $u_2$ and has a length equal to the length of $P_v$ plus one.
    
    Assume first that $v \in A_K$. Then $P^*$ has length $\ell$. 
By Lemma \ref{l:wProperEven}, since $(A, B, A', B', I, w, w')$ is a proper $\ell$-partition of $G$, $A_S$ is not empty. Let $q \neq v$ be a vertex  in $A_S
$ and $P_q = q\dots q'$ be the principal path of length $\ell-2$ joining $q$ and $q'$ in $G$. 
    Up to symmetry we may assume that $qv\notin E(G)$ and $q'v' \in E(G)$.  Now, by
    condition~\eqref{pb:I}  of preblowup
    and~\eqref{cb:chapeau}, $P^*$, $P_q$ and $w^*$ form a hole of length
    $2\ell+1$, a contradiction.
    
    So $v$ should be in $A_S$ and $P^*$ has length $\ell-1$. Since the $\ell$-partition of $G$ is proper, there exists a vertex $q \in A_S$ distinct from $v$, $P_q$ has length $\ell-2$ and there exists no edge between $P^*$ and $P_q$. Then $P^*,P_q, w^*$ and $w'^*$ induce a hole of length $2\ell+1$, a contradiction again. 
  \end{proofclaim}

  \begin{claim}\label{cb:Inestedu}
    For all $u\in I$ and $u^*\in K_u$, 
    $N[u^*]\subseteq N[u]$.
  \end{claim}

  \begin{proofclaim}
    Otherwise, by~\eqref{cb:Inested}, there exists a vertex
    $u^*\in K_u$ such that $N[u] \subsetneq N[u^*]$.  Hence
    $(V(G) \sm \{u\}) \cup \{u^*\}$ induces a subgraph $G_0$ of $G^*$
    and it is easy to verify that $G^*$ is a preblowup of $G_0$. This
    contradicts the maximality of $s(G, G^*)$.
  \end{proofclaim}
  
  By~(\ref{cb:Inested}), for every $u\in I$, the clique $K_u$ can be
  linearly ordered by the inclusion of the neighborhoods as
  $u_1, \dots, u_{k_u}$ with $u=u_{k_u}$ by~\eqref{cb:Inestedu} (so,
  for $1\leq i \leq j \leq k_u$, $N[u_i] \subseteq N[u_j]$). From 
  condition~\eqref{pb:I} of the preblowup it also follows that, in $G^*$, $u$ is complete to the cliques
  associated to its two neighbors in~$G$.

  \begin{claim}\label{cb:Anested}
    For every $u\in A$ and $u_1,u_2\in K_u$, either
    $N[u_1]\subseteq N[u_2]$ or $N[u_2]\subseteq N[u_1]$. A similar statement holds for $A'$.
  \end{claim}
  
  \begin{proofclaim} 
    Otherwise, there exist $x_1\in N[u_1]\setminus N[u_2]$ and
    $x_2\in N[u_2]\setminus N[u_1]$.  Note that $x_1x_2\notin E(G^*)$
    for otherwise, $\{x_1, x_2, u_1, u_2\}$ induces a $C_4$.
    
    Observe first that by~\eqref{cb:Anticomplet_in_A} and~\eqref{cb:Complet_in_A}, $N_{A^*}[u_1]=N_{A^*}[u_2]$. Hence
    by condition~\eqref{pb:A} of preblowup, $x_1,x_2\in B^*\cup K_{u^+}$ where $u^+$ is the neighbor of $u$ in the principal path that contains $u$. Without loss of
    generality and since $K_{u^+}$ is a clique, $x_1\in B^*$.

    By condition~\eqref{pb:BAN} of preblowup, there exist non-adjacent $a, b\in A$
    such that $x_1$ has neighbors $a^* \in K_a$ and $b^*\in K_b$, and  by~\eqref{cb:Anticomplet_in_A} $ a^*b^*\notin E(G^*).$
    Note that $a^*, b^*\neq u_2$ because $u_2x_1\notin E(G^*)$.  If
    $u_2$ is complete to $\{a^*, b^*\}$, then $\{u_2,a^*,x_1,b^*\}$
    induces a $C_4$, a contradiction.  So, up to symmetry
    $u_2a^*\notin E(G)$.  So, $a^*\notin K_u$ and
    by~\eqref{cb:Complet_in_A} and~\eqref{cb:Anticomplet_in_A}, $a^*u_1\notin E(G^*)$.
   Observe
    that $x_2a^*\notin E(G^*)$ for otherwise $\{a^*,x_1,u_1,u_2,x_2\}$
    induces a $C_5$.

    Suppose that $x_2\in B^*$. As above, we can show that $x_2$ has a neighbor
    $c^* \in A^*$ that is anticomplete to $\{u_1, u_2, x_1\}$. 
   Let 
   %$P_{u_2}= u_2\dots u'_2, 
   $P_{a^*} = a^*\dots a'^*$ and $P_{c^*} = c^*\dots c'^*$ 
    be defined as in \eqref{c:principB}. To avoid a hole $c^*x_2u_2u_1x_1a^*c^*$ of length $6$, we have
      $a^*c^*\notin E(G^*)$.  
%    By~\eqref{cb:Anticomplet_in_A} and~\eqref{cb:Complet_in_A} and
%    since $a^*c^*\notin E(G^*)$, $a'^*c'^*\in E(G^*)$. 
%    So, by conditions
%    \eqref{pb:A}, \eqref{pb:B} and \eqref{pb:I},
Since the lengths of $P_{a^*}$ and $P_{c^*}$ are at least $\ell-2$, depending whether $a'^*c'^*\in E(G^*)$,   $u_1x_1a^*P_{a^*}a'^*c'^*P_{c^*}c^*x_2u_2u_1$ or $u_1x_1a^*P_{a^*}a'^*w'^*c'^*P_{c^*}c^*x_2u_2u_1$ is a hole of length
   at least $2\ell+2$, a contradiction. 
       So $x_2\in K_{u^+}$.  Using condition~\eqref{pb:II} of
    preblowup, it is easy to verify that there exists a path $P_{u_2}$ from $u_2$ to some $u'^*\in K_{u'}$ defined similarly than in \eqref{c:principB}, which contains $x_2$. Now $u_2P_{u_2}u'^*a'^*P_{a^*}a^*x_1u_1u_2$ or $u_2P_{u_2}u'^*w'^*a'^*P_{a^*}a^*x_1u_1u_2$
    is a hole of length $2\ell +1$, a contradiction.
     
    The result for $A'$ holds symmetrically.
  \end{proofclaim}

  \begin{claim}\label{cb:Anestedu}
    For all $u\in A$ and $u^*\in K_u$, $N[u^*]\subseteq N[u]$. A
    similar statement holds for $A'$.
  \end{claim}

  \begin{proofclaim}
    Otherwise, by \eqref{cb:Anested} there exists a vertex $u^*\in K_u$ such that
    $N[u] \subsetneq N[u^*]$.  Hence, $(V(G)\sm \{u\}) \cup\{u^*\} $
    induces a subgraph $G_0$ of $G^*$ which is a template (by
    Lemma~\ref{evenpret}) with a  proper partition
    (by~\eqref{cb:Anticomplet_in_A} and~\eqref{cb:Complet_in_A}). Iit is easy to verify that
    $G^*$ is a preblowup of $G_0$. This
    contradicts the maximality of $s(G, G^*)$.  The result for $A'$
    holds symmetrically.
  \end{proofclaim}

  By~(\ref{cb:Anested}), for every $u\in A \cup A'$, the clique $K_u$ can be
  linearly ordered by the inclusion of the neighborhoods as
  $u_1, \dots, u_{k_u}$, and by~(\ref{cb:Anestedu}) $u_{k_u} =u$ (so,
  for $1\leq i \leq j \leq k_u$, $N[u_i] \subseteq N[u_j]$).

  \begin{claim}
    \label{cb:Bnested}
    If $xy$ is an edge of $G[B^*]$, then either
    $N_{A^*}(x)\subseteq N_{A^*}(y)$ or
    $N_{A^*}(y)\subseteq N_{A^*}(x)$. 
  \end{claim}
  
  \begin{proofclaim}
    Otherwise, there exist $u^*\in N_{A^*}(x)\setminus N_{A^*}(y)$
    and $v^*\in N_{A^*}(y)\setminus N_{A^*}(x)$. Note that
    $u^*v^*\notin E$ for otherwise $\{u^*,x,y,v^*\}$ induces a
    $C_4$. 
    So, for some distinct $u, v\in A$, we have $u^*\in K_u$ and
    $v^*\in K_v$. Hence, by \eqref{cb:Complet_in_A},
    $uv\notin E(G)$.  Let $P_{u^*} = u^*\dots u'^*$ and  
    $P_{v^*} = v^*\dots v'^*$ be defined 
    as in \eqref{c:principB}.  So, $xu^*P_{u^*}u'^*v'^*P_{v^*}v^*yx$ or $xu^*P_{u^*}u'^*w'^*v'^*P_{v^*}v^*yx$ forms a hole
    of length $2\ell+1$, a contradiction. 
  \end{proofclaim}

  \begin{claim}
    \label{c:nonAdjBstar}
    For every $x\in B^*$, there exist non-adjacent $u,v\in A$ such
    that $xu, xv \in E(G^*)$. 
  \end{claim}
    
  \begin{proofclaim}
    This follows from condition~\eqref{pb:BAN} of
    preblowup and from (\ref{cb:Anestedu}).
  \end{proofclaim}

  Two vertices $x, y$ in $B^*$ are \emph{equivalent} if
  $N_{A}(x) = N_{A}(y)$.

  \begin{claim}\label{cb:Bclique}
    If $x$ and $y$ are equivalent vertices of $B^*$, then
    $xy\in E(G^*)$.
  \end{claim}
  
  \begin{proofclaim}
    If $xy\notin E(G^*)$, then $x$, $y$ and two of their neighbors
    provided by~\eqref{c:nonAdjBstar} induce a $C_4$. 
  \end{proofclaim}

  Vertices of $B^*$ are partitioned into equivalence
  classes. By~\eqref{cb:Bclique}, each equivalence class is a clique
  $X$, and by~\eqref{cb:Bnested}, vertices of $X$ can be linearly
  ordered according to the inclusion of neighborhoods in $A^*$.  In each such a
  clique $X$ we choose a vertex $x$ maximal for the order and call
  $B_1$ the set of these maximal vertices.  For every $x\in B_1$, we
  denote by $K_x$ the clique of $B^*$ of all vertices equivalent to
  $x$. Remind that $w^*\in B^*$ : $w^*$ is a maximal vertex of its
  clique. Hence, we can set $w^*\in B_1$.
 
So, for every $u\in B_1$, the clique $K_u$ can be linearly
  ordered by the inclusion of the neighborhod in $A^*$ as
  $u_1, \dots, u_{k_u}$ with $u=u_{k_u}$ (so, for
  $1\leq i \leq j \leq k_u$, $N_{A^*}(u_i) \subseteq N_{A^*}(u_j)$).
  
 Statements similar to \eqref{cb:Bnested}, \eqref{c:nonAdjBstar}, \eqref{cb:Bclique}
hold for $B'^*$ and we define $B'_1$ as well. 
  
We set $G_1 = G^*[A \cup B_1 \cup A' \cup B'_1 \cup I]$ and claim that
$(A, B_1, A',B'_1, I)$ is an even $\ell$-pretemplate partition of $G_1$.
Since $G_1[A\cup I \cup A']$ is exactly $G[A\cup I \cup A']$,
conditions~\eqref{b:cardinality1}, \eqref{b:cardinality2} \eqref{b:path}, \eqref{b:interI}  and \eqref{b:sp}
hold. Adding the fact that $N_{G_1}(B_1)\subseteq A^*\cap V(G_1)=A$
by condition \eqref{pb:B} of preblowup, condition \eqref{b:antiC} for
a pretemplate holds and symmetrically also condition
\eqref{b:antiCP}. Now condition \eqref{b:cn} holds because $w^*$ and
$w'^*$ are universal in respectively $G^*[A\cup B_1]$ and $G^*[A'\cup B'_1]$. By
\eqref{c:nonAdjBstar}, the last two conditions for a pretemplate are
fulfilled by $(A, B_1, A', B'_1, I)$. Hence, by Lemma \ref{evenpret},
$G_1$ is a an even $\ell$-template.  It is twinless by
Lemma~\ref{evenl:wTwins}. We also notice that by construction $w^*$ (respectively $w'^*$) belongs to $G_1$ and  is universal in
$G_1[A \cup B_1]$ (respectively $G_1[A' \cup B'_1]$). Since $(A, B, A', B', I, w, w')$ is a proper $\ell$-partition of $G$, there exist isolated vertices in $G^*[A]$ and $G^*[A']$. Hence
$(A, B_1, A', B'_1, I,w^*,w'^*)$ is a proper even $\ell$-partition of
$G_1$.

   We now prove that
  $G^*$ is a proper blowup of $G_1$.

 By the definition of a preblowup and by~\eqref{cb:Bclique}, for all $u\in V(G_1)$, $K_u$ is a clique and $V(G^*)= \bigcup_{u\in V(G_1)} K_u$ 

  \begin{claim}\label{cb:anticomplete}
    If $u, v\in V(G_1)$ and $uv\notin E(G_1)$, then $K_u$ is anticomplete to $K_v$.
  \end{claim}

  \begin{proofclaim}
   % Suppose $u, v\in V(G_1)$ and $uv\notin E(G_1)$.  
    If $u\in I$ or
    $v\in I$, the conclusion follows directly from
    condition~\eqref{pb:I} of preblowup.  So we may assume up to
    symmetry that $u\in A\cup B_1$.  By conditions~\eqref{pb:A}
    and~\eqref{pb:B} of preblowup, we may assume $v\in A \cup B_1$.
    If $u, v\in A$, then the result follows from
  \eqref{cb:Anticomplet_in_A}, so we may assume that $v\in B_1$. 
  
  Now suppose for a contradiction that there exist $u^*\in K_u$ and
  $v^*\in K_v$ such that $u^*v^*\in E(G_1)$. By the choice of vertices
  in $B_1$, for all $v^*\in K_v$, $N[v^*]\subseteq N[v]$. So
  $u^*v\in E(G_1)$. For the same reason or by~\eqref{cb:Anestedu}, for
  all $u^*\in K_u$, $N[u^*]\subseteq N[u]$. Hence $uv \in E(G_1)$, a
  contradiction.
  \end{proofclaim}

 \begin{claim}\label{cb:cliquescompletes}
   If $uv$ is a solid edge of $G_1$ then $K_u$ is complete to $K_v$.
  \end{claim}

  \begin{proofclaim}
    Otherwise, let $u^*\in K_u$ and $v^*\in K_v$ such that $u^*v^*\notin
    E(G)$. Since $uv$ is a solid edge, up to symmetry, $u,v\in A$ or
    $u,v\in B_1$ or $u\in A$, $v\in B_1$ and in this last case $u$ is not an isolated vertex
    of $G[H_v]$.
    
    By~\eqref{cb:Complet_in_A} the case where $u$ and $v$ are in $A$
    cannot happen.  Assume then that
    $v \in B_1$.  By Lemma \ref{evenl:recoverHx}, there exist
    $a,b \in H_v$ (and hence in $A$) that are not adjacent. Assume
    that $u$ is also in $B_1$. Since $u$ and $v$ are adjacent, by \eqref{cb:Bnested} we may
    assume without loss of generality that $H_v \subseteq H_u$ and so
    $a$ and $b$ belong to $H_u$ too.  Then, by the definition of $K_u$
    and $K_v$, we get a $C_4$ induced by $\{u^*,v^*, a,b\}$, a
    contradiction.

So $u$ should be in $A,$ and to avoid a $C_4$
    induced by $\{u^*,v^*, a,b\}$, $u^*$ should be non-adjacent to at
    least one of $a$ and $b$, say $a$. In particular, $a\neq u$.
    Then, by~\eqref{cb:Complet_in_A}, $ua \notin E(G_1)$. So $u$ does
    not belong to $N(H_v)$ and since $uv$ is an edge of $G_1$, we get
    that $u \in H_v$. Since $uv$ is solid, $u$ has at least one neighbor in
    $H_v$, and we know that $u$ is not adjacent to at least one vertex in $H_v$ (namely $a$). Hence, as $H_v$ is anticonnected, there exist  non-adjacent
    vertices $c, d \in H_v$ such that $uc \notin E(G_1)$ and
    $ud \in E(G_1)$. Now $u^*P_{u^*}u'^*c'P_ccv^*du^* $ or $u^*P_{u^*}u'^*w'^*c'P_ccv^*du^* $ is a hole of
    length $2\ell+1$, a contradiction again.

  \end{proofclaim}

  \begin{claim}\label{cb:nested}
    For all $u\in V(G_1)$ and $1\leq i \leq j \leq k_u$, 
    $N[u_i]\subseteq N[u_j]$. 
  \end{claim}

  \begin{proofclaim}
    The result follows from how vertices are ordered after the proof of
    \eqref{cb:Inestedu}
    (vertices in $I$), \eqref{cb:Anestedu} (vertices in $A$ or $A'$) and
    \eqref{cb:Bclique} (vertices in $B_1$ or $B'_1$) and  from \eqref{cb:cliquescompletes}.
  \end{proofclaim}

  \begin{claim}
    \label{cb:flat} If $uv$ is a flat edge of $G_1$, then $u$ is
    complete to $K_v$ and $v$ is complete to $K_u$.
  \end{claim}
  
  \begin{proofclaim}
    By definition of a flat edge, either $u$ and $v$ are in $I$ or one
    is in $I$ and the other is in $A$ or in $A'$. The result follows
    from~\eqref{cb:Inestedu}, \eqref{cb:Anestedu}, and
    conditions~\eqref{pb:AI} (applied to $A$ or $A'$)
    and~\eqref{pb:II} of the preblowup.
  \end{proofclaim}

  \begin{claim}
    \label{c:optAB}
    If $ux$ is an optional edge of $G_1$ with $u\in A$
    and $x\in B_1$ (resp.\ $u\in A'$ and $x\in B_1'$), then $u$ is
    complete to $K_x$.
  \end{claim}

  \begin{proofclaim}
    The result follows from the definition of $K_x$ when $x \in B_1$. 
  \end{proofclaim}

  \begin{claim}\label{cb:cascade}
    If $ux$ and $uy$ are optional edges with $u\in A$, $x, y\in B_1$
    and $H_y\subsetneq H_x$ (resp.\ $u\in A'$, $x, y\in B_1'$ and
    $H'_y\subsetneq H'_x$), then every vertex of $K_u$ with a neighbor
    in $K_y$ is complete to $K_x$.
  \end{claim}

  \begin{proofclaim}
    Otherwise, let $u^*$ be a vertex in $K_u$ that has a neighbor
    $y^*$ in $K_y$ and a non-neighbor $x^*$ in $K_x$. Since
    $H_x$ and $H_y$ are not disjoint, $xy$ is a solid edge of $G_1$ and
    by~(\ref{cb:cliquescompletes}), $x^*y^*\in E(G_1)$.
    
%  Let us notice that $N_A(y)\subseteq N_A(x)$. Indeed, let $a \in A \sm N_A(x)$. By definition of a template, $a \in A \sm (H_x \cup N_A[H_x])$. Then since $H_y\subsetneq H_x$ and $H_x$ is a module of $A$ we get that $a$ is anticomplete to $H_x$ and hence to $H_y$. So $a \notin N_A(y)$.
  
  Since $x$ and $y$ are not equivalent, there exists a vertex $a$  such that $a\in N_A(y)\sm N_A(x)$ or $a\in N_A(x)\sm N_A(y)$. In the first case, by definition of a template, $a \in A \sm N_A[H_x]$. Then since $H_y\subsetneq H_x$ and $H_x$ is a module of $A$ we get that $a$ is anticomplete to $H_x$ and hence to $H_y$. So $a \notin N_A(y)$, a contradiction; we may then conclude that $a\in N_A(x)\sm N_A(y)$
  
  By definition of the cliques in $B$, $x^*a\in E(G^*)$ and $y^*a\notin E(G^*)$. Therefore, to avoid a $C_4$ induced by $\{x^*, y^*, u^*,a\}$, it should be that $u^*a\notin E(G^*)$. 

%    $x^*a\in E(G^*)$ and $y^*a\notin E(G^*)$. Therefore, to avoid a $C_4$
%    it should be that $u^*a\notin E(G^*)$.
%    $H_y\subsetneq H_x$, $N_A(y)\subsetneq N_A(x)$ and there exists
%    $a\in N_A(x)\sm N_A(y)$.
%    Since $x$ and $y$ are not equivalent, $N_A(x)\neq N_A(y)$. Since
%    $H_y\subsetneq H_x$, $N_A(y)\subsetneq N_A(x)$ and there exists
%    $a\in N_A(x)\sm N_A(y)$.  By definition of the cliques in $B$,
%    $x^*a\in E(G^*)$ and $y^*a\notin E(G^*)$. Therefore, to avoid a $C_4$
%    it should be that $u^*a\notin E(G^*)$.
    
    Now $aP_aa'u'^*P_{u^*}u^*y^*x^*a$ or $aP_aa'w'^*u'^*P_{u^*}u^*y^*x^*a$ is a hole of length $2\ell +1$ a
    contradiction.
 \end{proofclaim}

  \begin{claim}
    \label{cb:condIso} $w^*$ (resp.\ $w'^*$) is a universal vertex of
    $G^*[\bigcup_{u\in A\cup B_1} K_u]$ (resp.\
    $G^*[\bigcup_{u\in A'\cup B_1'} K_u]$).
  \end{claim}

  \begin{proofclaim}
  
    By \eqref{cb:chapeau}, $w^*$ is complete to $A^*$ and
    so to $\bigcup_{u\in A} K_u$. Furthermore, from the definition of
    $G_1$ we know that $w^*$ is complete to $B_1 \sm \{w^*\}$. Since all edges between vertices in $B_1$ are solid,
    by~(\ref{cb:cliquescompletes}), $w^*$ is complete to
    $B^*\sm \{w^*\}$.   Hence $w^*$ is a universal
    vertex of $G^*[\bigcup_{u\in A\cup B} K_u]$. The proof for $w'^*$ is
    symmetric.
  \end{proofclaim}

  From all the claims above, $G^*$ satisfies all conditions to be a
  proper blowup of $G_1$.
\end{proof}

\section{Graphs in $\mathcal C_{2\ell}$ that contain a theta or a prism}

The goal of this section is to prove the following.

\begin{lemma}
  \label{l:GPyEven}
  Let $\ell\geq 4$ be an integer.  If $G$ is a graph in
  $\mathcal C_{2\ell}$ and $G$ contains a theta or a prism, then one of the
  following holds:

  \begin{enumerate}
  \item\label{c:PyblowupEven} $G$ is a proper blowup of a twinless even
    $\ell$-template;
  \item\label{c:Pyuniv} $G$ has a universal vertex;
  \item\label{c:Pycliquecut} $G$ has a clique cutset.
  \end{enumerate}
\end{lemma}

The rest of this section is devoted to the proof of
Lemma~\ref{l:GPyEven}. So from here on, $\ell\geq 4$ is an integer and $G$
is graph in $\mathcal C_{2\ell}$ that contains a theta~$\Theta$ or a prism $\Sigma$.
By Lemma~\ref{l:holeTruemperS}, the three paths of $\Theta$  have
length~$\ell$ and those of $\Sigma$ have length $\ell-1$.  By Lemma~\ref{l:ThetaPrismTemplate}, $\Theta$ and $\Sigma$ are even $\ell$-templates.  Hence, we may define an integer $k$ and a sequence
$F_0, F_1, F_2$ of induced subgraphs of $G$ as follows.

\begin{itemize}
\item $k$ is the maximum integer such that $G$ contains
  an even $\ell$-template with $k$ principal paths. Observe that by
  Lemma~\ref{evenl:wTwins}, $G$ in fact contains a twinless even $\ell$-template with
  $k$ principal paths, because twins can be eliminated from templates
  by deleting hyperedges with equal vertex-set while there are some.

\item In $G$, pick a proper blowup $F_1$ of a twinless even
  $\ell$-template $F_0$ with $k$ principal paths.  Note that $F_0$
  exists and the proper $\ell$-partition needed for the proper blowup
  exists by Lemma~\ref{l:TProperEven}.

\item Suppose that $F_0$ and $F_1$ are chosen subject to the
  maximality of the vertex-set of $F_1$ (in the sense of
  inclusion). Note that possibly $F_0$ is not a maximal template in
  the sense of inclusion, it can be that a smaller template leads to a
  bigger blowup (but $F_0$ has $k$ principal paths).

\item $F_2$ is obtained from $F_1$ by adding all vertices of
  $G\sm F_1$ that are complete to $F_1$. 
\end{itemize}

\begin{lemma}
  \label{l:F2-F1CliqueEven}
  $V(F_2)\sm V(F_1)$ is a (possibly empty) clique that is
  complete to $F_1$.
\end{lemma}

\begin{proof}
  Otherwise, $G$ contains a $C_4$.  
\end{proof}

We now introduce some notation.  We denote by $(A, B, A', B', I, w, w')$ the
twinless proper even $\ell$-partition that is used to blow up $F_0$. We have $A= A_K \cup A_S$ and $A'= A'_K \cup A'_S$.
 When $u$ is a vertex
of $F_0$, we denote by $K_u$ the clique of $F_1$ that is blown up from
$u$.  We set $A_K^* = \bigcup_{u\in A_K} K_u$.  We use a similar notation $A_S^*$, ${A'_K}^*$, ${A'_S}^*$,
$B^*$, $B'^*$, $I ^*$, $A^*$ and $A'^*$.

\subsection{Technical lemmas}

We now prove lemmas that sum up several structural properties of $G$.

%63
\begin{lemma}\label{lb:universalevenforblow-upEven}
  If $u\in A\cup A' \cup I \cup \{w, w'\}$ and $v\in N_{V(F_0)}(u)$,
  then $u$ is complete to $K_v$. 
\end{lemma}

\begin{proof}
  We prove this lemma using the conditions from the definition of
  blowups. If $u\in \{w,w'\}$, then the result follows from
  condition~\eqref{i:condIso}. If $u\in A\cup A'$, then the conclusion
  follows from conditions~\eqref{i:Kcomp}, \eqref{i:flat} and~\eqref
  {i:optAB}.  If $u\in I$, then the conclusion follows from
  condition~\eqref {i:flat}.
\end{proof}

Very often, Lemma~\ref{lb:universalevenforblow-upEven} will be used in the
following way. Suppose there exists a principal path $P=u\dots u'$ of
$F_0$.  Suppose there exists a vertex $x$ of $P$ and $x^*\in K_x$.
Then by Lemma~\ref{lb:universalevenforblow-upEven} and
condition~\eqref{i:bKKanti} of blowups,
$\{x^*\} \cup( V(P) \sm \{x\})$ induces a path of $F_1$. If $y\neq x$
is a vertex of $P$ and $y^*\in K_y$, then
$\{x^*, y^*\} \cup (V(P) \sm \{x, y\})$ might fail to induce a path of
$F_1$, because it is possible that $xy\in E(G)$ while
$x^*y^*\notin E(G)$.  But under the assumption that $x^*y^*\in E(G)$
or $xy\notin E(G)$, we do have that
$\{x^*, y^*\} \cup (V(P) \sm \{x, y\})$ induces a path of
$F_1$. Several variant of this situation will appear soon and we will
simply justify them by refering to
Lemma~\ref{lb:universalevenforblow-upEven}.

When $u$ is a vertex in $A$, we denote by $P_u$ the unique principal
path of $F_0$ that contains $u$. Its end in $A'$ is then denoted by
$u'$.  We denote by $u^+$ the neighbor of $u$ in $P_u$. We denote by
$u^{++}$ the neighbor of $u^+$ in $P_u\sm u$.  Note that $u^+\in I$
and $u^{++}\in I\cup A'$ ($u^{++} \in A'$ if and only if $\ell=4$ and $u \in A_S$).

For any distinct $u, v\in A$, from the definition of even templates,
exactly one of 
$V(P_u) \cup V(P_v)$ or $V(P_u) \cup V(P_v) \cup \{w\}$ or 
$V(P_u) \cup V(P_v) \cup \{w'\}$ or
$V(P_u) \cup V(P_v) \cup \{w,w'\}$ induces a hole that is denoted by
$C_{u, v}$.  Such a hole is called a \emph{principal hole}.

So, there are three kinds of principal holes: those that contain
exactly one of $w$ and $w'$, those that contain none of $w$ and $w'$ and those that contain both $w$ and $w'$.  Recall that by 
Lemma~\ref{l:HinTemplate_even}, some holes of a template contain two
principal paths plus one or two extra vertices, but it may fail to be a
principal hole (because the extra vertices may fail to be $w$ or $w'$).  Though
we do not use this information formally, it is worth noting that by
Lemma~\ref{lb:universalevenforblow-upEven}, when $C$ is a principal hole,
$\bigcup_{v\in V(C)} K_v$ induces a ring. But when $C$ is a non-principal
hole, it may happen that $\bigcup_{v\in V(C)} K_v$ does not induce a
ring (because there might be in $C$ an optional edge $uv$ with
$u\in A$ and $v\in B$, and after the blowup process, there might be that no
vertex in $K_v$ is complete to $K_u$).

\begin{lemma}
  \label{evenl:twoNonAdj}
  If $u\in V(F_0)$ and $u^*\in K_u$, then $u^*$ has two 
  neighbors in $V(F_0)\sm K_u$ that are not adjacent.
\end{lemma}

\begin{proof}
  If $u\in I$, then let $P$ be the principal path that contains $u$. By
  Lemma~\ref{lb:universalevenforblow-upEven}, $u^*$ is adjacent to the two
  neighbors of $u$ in $P$.

  If $u\in A\cup A'$, say $u\in A$ up to symmetry, then we claim that
  $u$ has a neighbor $z$ in $A\cup B$.  This is clear if $u$ is not
  isolated in $A$  and otherwise we set $z=w$.  By
  Lemma~\ref{lb:universalevenforblow-upEven}, $z$ and $u^+$ are
  non-adjacent neighbors of $u^*$.

  If $u\in B$, then by the definition of a template, $H_u$ contains two non adjacent vertices $a$ and $b$ that are neighbors of $u$.  
  By Lemma~\ref{lb:universalevenforblow-upEven}, $a$ and $b$ are
  both adjacent to $u^*$. 
\end{proof}

%\begin{lemma}
%  \label{l:htabEven}
%  If $uv$ is an edge of $F_0[A\cup A' \cup I \cup \{w, w'\}]$, then
%  some principal hole of $F_0$ goes through $uv$.
%\end{lemma}
%
%\begin{proof}
%  If at least one of $u$, $v$ is in $I$, or if one of $u$, $v$ is in $A$ and the other in $A'$ then $uv$ is an edge of a
%  principal path and we know that this principal path belongs to a
%  principal hole.  Else, up to symmetry both $u$ and $v$ are
%  in $A$ or $u= w\in B$ and $v\in A$.
%
%  If $u,v \in A$ then either both $u$ and $v$ are in $A_K$ and  $P_u, P_v$ induce a principal hole containing $uv$, or one of $u$, $v$ is in $A_K$ and the other in $A_S$ and $P_u, P_v$ and $w'$ induce a principal hole containing $uv$.
%
%  If $u= w\in B$ and $v\in A$ : since the partition is proper, $G[A]$ has no
%  universal vertex and there exists $a \in A$ which is not adjacent to
%  $v$. Now $w, P_v, P_a$ and possibly $w'$ form a principal hole containing the edge
%  $uv$.
%\end{proof}

\begin{lemma}
  \label{evenl:onlyKaKb}
  If $K$ is a clique of $F_0$, $K^* = \bigcup_{v\in K} K_v$ and $D$ is a
  connected induced subgraph of $G\sm F_2$ such that
  $N_{V(F_1)}(D) \subseteq K^*$, then $N_{V(F_1)}(D)$ is a clique.
\end{lemma}

\begin{proof}
  For suppose not.  This means that there exists
  $u^*, v^*\in K^*$ and $x_u, x_v\in D$ such that
  $u^*v^*\notin E(G)$ and $x_uu^*, x_vv^*\in E(G)$ (possibly
  $x_u=x_v$).  Since $D$ is connected, there exists a path $P$ in $D$
  from $x_u$ to $x_v$.  Suppose that $u^*$, $x_u$, $v^*$, $x_v$ and
  $P$ are chosen subject to the minimality of $P$.  It follows that
  $u^*x_uPx_vv^*$ is a path, and recall that by assumption
  its interior is anticomplete to $F_1\sm K^*$.
  
  Since $u^*v^*\notin E(G)$, $u^*$ and $v^*$ are in different blown-up
  cliques. Denote by $K_u$ and $K_v$ the blown-up cliques such that
  $u^*\in K_u$ and $v^*\in K_v$. By hypothesis, $u,v\in K$ and so
  $uv\in E(G)$. Since $u^*v^*\notin E(G)$, by
  condition~\eqref{i:Kcomp} of blowups, $uv$ is not a solid edge of
  $G$.

  If $uv$ is a flat edge of $F_0$, then $uv$ is an edge of a
  principal path and we know that this principal path belongs to a
  principal hole $C$. Note that apart from $u$ and
  $v$, no vertex of $C$ is in $K$ since $K$ is a clique. By
  Lemma~\ref{lb:universalevenforblow-upEven}, in $G$,
  $(\{u^*, v^*\}) \cup V(C))\sm \{u, v\}$ induces a path $Q$ of length
  $2\ell-1$.  So $P$ and $Q$ form a hole of length at least $2\ell +1$,
  a contradiction.

  If $uv$ is an optional edge of $F_0$, say with $u\in A$ and
  $v\in B$, then $u\in H_v$, and there exists $a$ in $H_v$ such that
  $au\notin E(F_0)$.  Therefore, $P_u$, $P_a$, $v$ and possibly $w'$ form a hole
  $C^*$. By condition~\eqref{i:optAB} of blowups (if $va$ is optional),
  or by condition~\eqref{i:Kcomp} (if $va$ is solid), $a$ is complete
  to $K_v$.  By Lemma~\ref{lb:universalevenforblow-upEven} it follows that
  $(\{u^*, v^*\}) \cup V(C^*))\sm \{u, v\}$ induces a path $Q$ of length
  $2\ell-1$.  So $P$ and $Q$ form a hole of length at least $2\ell +1$,
  a contradiction again.
\end{proof}

%When $C$ is a hole of $G$, a vertex $v$ of $V(G)\sm V(C)$ is
%\emph{minor} w.r.t.\ $C$ if $N_{V(C)}(v)$ is included in a 3-vertex
%path of $C$.  A vertex of $V(G)\sm V(C)$ that is not minor w.r.t.\ $C$ is \emph{major}
%w.r.t.\ $C$.

\medskip

\begin{lemma}
  \label{evenl:xMinor}
  If $x\in V(G) \sm V(F_2)$ and $C$ is a principal hole of $F_0$, then $x$
  is minor w.r.t.\ $C$. 
  
  (We remind that $x$ is minor w.r.t.\ $C$ if the neighborhood of $x$ in $C$ is included in a $3$-vertex path of $C$.)
\end{lemma}

\begin{proof}
Otherwise let $C=C_{u, v}$ for some $u, v \in A$ such that
  $x$ is major w.r.t.\ $C$. By Lemma~\ref{l:holeTruemperS},  $x$ and $C$ form a universal wheel.
  
   \begin{claim} \label{univAI}
   $x$ is complete to all principal paths.
    \end{claim}

 \begin{proofclaim}
 
  %Suppose up to symmetry that $w\in V(C)$ and suppose $C=C_{u, v}$ for some $u, v \in A$ ({\bf $C_{u, v}$ to be defined}) .
%  If $x$ is major w.r.t.\ $C$, then $C$ and $x$  form a theta or a
%  wheel that is not a twin-wheel. So by {\bf the analog of Lemma~\ref{l:holeTruemperS} (If $\ell\geq 3$ is an integer and $G\in \mathcal C_{2\ell}$, every Truemper configuration of $G$ is a twin wheel, a universal
%  wheel, a theta whose three paths all have length $\ell$ or a prism whose three paths all have length $\ell-1$}),
%  $x$ and $C$ form a universal wheel.  
We know already that $x$ is complete to $P_u$ and $P_v$.
Let $P_t = t\dots t'$ be a
  principal path where $t\neq u, v$.
  If one of $u,v$, say $u$, is in $A_K$ then $P_u$ contains at least $4$ vertices, so $x$ is major w.r.t.\ $C_{u,t}$ and as above we may conclude that $x$ and $C_{u,t}$ form a universal wheel. Consider now the case where both $u$ and $v$ are in $A_S$. Then $C$ contains $w$ and $w'$ and $C_{u,t}$ contains at least one of $w,w'$. So $x$ is  is adjacent to at least $4$ vertices of $C_{u,t}$ and  we may again conclude that $x$ is complete to $P_t$.  
%   If $t$ is complete to $\{u, v\}$,
%  then $xt\in E(G)$ for otherwise $\{t, u, v, x\}$ induces a $C_4$.
%  Then we claim that  $x$ has at least~4 neighbors in $C_{u, t}$. Indeed, either $u \in A_K$ and $P_u$ contains at least $3$ vertices, or    so by
%  Lemma~\ref{l:holeTruemperS}, $x$ is complete to $P_t$.  If $t$ is not
%  complete to $\{u, v\}$, say $tu\notin E(G)$, then $x$ again has at
%  least~4 neighbors in $C_{u, t}$ because $w\in V(C_{u, t})$, so again
%  $x$ is complete to $P_t$.
  \end{proofclaim}

 \begin{claim} \label{univB}
   $x$ is complete to $B\cup B'$.
    \end{claim}
    
\begin{proofclaim} 
   Let $y\in B\cup B'$.  By definition of a
  template, $y$ has two neighbors $a$ and $b$, both in $A$ or both in
  $A'$, that are non-adjacent. Therefore $a$, $b$, $y$ and $x$ form a
  $C_4$, unless $x$ is adjacent to $y$.  
\end{proofclaim}

By  \eqref{univAI} and \eqref{univB}, $x$ is complete to  $I\cup A \cup A' \cup B\cup B'= V(F_0)$. 

  Let $z$ be a vertex of $F_0$ and $z^*\in K_z$.  By
  Lemma~\ref{evenl:twoNonAdj}, there exists $a,b \in V(F_0)$ such that $z^*a, z^*b \in E(G)$
  and $ab \notin E(G)$. Since there is no $C_4$ in $G$ it should be that $xz^*\in E(G)$.  
  This proves that $x$ is complete to $F_1$.  Hence, $x\in V(F_2)$, a
  contradiction.
\end{proof}

\begin{lemma}
  \label{evenl:xKaKb}
  Let $a$ and $b$ be two non-adjacent vertices of some principal hole
  $C$ of $F_0$.  If some vertex $x$ of $V(G)\sm V(F_2)$ has neighbors
  in both $K_a$ and $K_b$, then $a$ and $b$ have a common neighbor $c$
  in $C$, $x$ is adjacent to $c$, and $x$ is anticomplete to every
  $K_d$ such that $d\in V(C) \sm \{a, b, c\}$.
\end{lemma}

\begin{proof}
  Let $a^* \in K_a$ and $b^* \in K_b$ be two neighbors of $x$.
  Since $ab\notin E(G)$, by Lemma~\ref{lb:universalevenforblow-upEven},
  $\{a^*, b^*\} \cup V(C)\sm \{a, b\}$ induces a hole~$C^*$.  Since
  $x$ is adjacent to $a^*$ and $b ^*$, by Lemma~\ref{l:holeTruemperS}, $x$
  has another neighbor $c$ in $C^*$ (and in fact in $C$ since $c\neq
  a^*, b^*$).  If $c$ is not adjacent to $a^*$ or $b^*$, then $x$ is major
  w.r.t.\ $C^*$, so by Lemma~\ref{l:holeTruemperS}, $C^*$ and
 $x$ form a universal wheel.  It follows that $x$ is major w.r.t.\
 $C$, a contradiction to Lemma~\ref{evenl:xMinor}. 

 We proved that $a$ and $b$ have a common neighbor $c$ in $C$ and that
 $x$ is adjacent to $c$.  Suppose for a contradiction that $x$ has a
 neighbor $d^*\in K_d$ where $d\in V(C) \sm \{a, b, c\}$.  By the same
 argument as above, since $x$ has neighbors in $K_d$ and $K_c$, $c$
 and $d$ must have a common neighbor in $C$, and this common neighbor
 must be $a$ or $b$, say $a$ up to symmetry.  So, $x$ has neighbors in
 $K_d$ and $K_b$ while $b$ and $d$ have no common neighbors in $C$, so
 we may reach a contradiction as above. 
\end{proof}

\subsection{Connecting vertices of a template}
\label{subs:connectF0Even}

%\begin{lemma}\label{lt:cheminsentrepatatesBBPEven}
%  If $x\in B$ and $y\in B'$, then there exist in $G$ two paths $P$
%  and $Q$ of length $\ell +1$ from $x$ to $y$ such that $P$ (resp.\
%  $Q$) contains a principal path $P_0$ (resp.\ $Q_0$), and
%  $P_0\neq Q_0$.
%\end{lemma}
%
%\begin{proof}
%  We set $X=\{i\in \{1, \dots, k\} : v_i\in H_x\}$ and
%  $Y=\{i\in \{1, \dots, k\} : v'_i\in H'_{y}\}$.  So, $X$ and $Y$ are
%  hyperedges of $\mathcal H$ and since $\mathcal H$ is laminar, either
%  $X\subseteq Y$, $Y\subseteq X$ or $X\cap Y=\emptyset$.
%  
%  If $X\subseteq Y$, then let $i, j$ be distinct members of $X$ (and
%  therefore of $Y$). The paths $xv_iP_iv'_iy$ and $xv_jP_jv'_jy$ are
%  the paths we are looking for.  The proof is similar when
%  $Y\subseteq X$.
%
%  If $X\cap Y = \emptyset$, then let $i, j, q, r$ be distinct integers
%  such that $i, j\in X$ and $q, r\in Y$.  Since $G[A]$ is isomorphic
%  to the complement of $G[A']$, we may assume up to symmetry that
%  $v_iv_q\in E(G)$.  So, $v'_iv'_q\notin E(G)$.  Since $H'_y$ is a
%  module of $G[A']$, $v'_iv'_r\notin E(G)$.  It follows that
%  $v_iv_r\in E(G)$.  So, $v_r, v_q\in N_A(H_x)$.  Hence, by
%  property~\eqref{t:makeIt} of templates, $xv_r, xv_q\in E(G)$.  It
%  follows that $xv_qP_qv'_qy$ and $xv_rP_rv'_ry$ are the two paths we
%  are looking for.
%\end{proof}

\begin{lemma}\label{lt:cheminsentrepatatesEven}
  If $x\in A\cup B$ and $y\in A'\cup B'$, then there exists in $G$ a
  path $P$ of length $\ell-2$, $\ell-1$, $\ell$ or $\ell +1$ from $x$ to
  $y$ that contains a principal path.

  More specifically:
  \begin{itemize}
  \item
    If $x\in A$ and $y\in A'$, then $P$ has length $\ell-2$, $\ell-1$,
    $\ell$ or $\ell+1$.
  \item If $x\in A$ and $y\in B'$, or if $x\in B$ and $y\in A'$,
    then $P$ has length $\ell-1$, $\ell$ or $\ell+1$.
  \item If $x\in B$ and $y\in B'$, then $P$ has length
$\ell$ or $\ell+1$. Furthermore in that case, there exists another path $Q$ from $x$ to
  $y$ of length
  $\ell$ or $\ell+1$ containing a principal path and $Q$ contains no interior vertex of $P$.
  \end{itemize}
\end{lemma}

\begin{proof}
  Suppose first that $x\in A$, say $x=v_i$.  If $y\in A'$, then set
  $y=v'_j$. If $i=j$, then $P_i$ has length $\ell-2$ or $\ell-1$. Assume $i\neq j$.
  If $v_iv_j$ or $v'_iv'_j$, say $v_iv_j$ is an edge then $v_iv_jP_{v_j}v'_j$ is a path of length $\ell-1$ or $\ell$.
  If none of $v_iv_j$ or $v'_iv'_j$ is an edge then $v_i$ and $v_j$ are in $A_S$ and $v_iP_{v_i}v'_iw'v'_j$ is a path of length $\ell$.
   If $y\in B'$, then one of $v_iP_iv'_iy$ or $v_iP_iv'_iw'y$
  is the path of length at most $\ell+1$ we are looking for.  The proof is similar when
  $y\in A'$.

  We may therefore assume that $x\in B$ and $y\in B'$. Assume first that there exist $v_i$ and $v_j$ in $A$ such that $x$ is adjacent to $v_i$ and $v_j$ and $y$ is adjacent to $v'_i$ and $v'_j$. Then $xv_iP_{v_i}v'_iy$ and  $xv_jP_{v_j}v'_jy$ are two paths of length $\ell$ or $\ell+1$. Consider now the case when there exists one vertex $v_i$ in $H_x$ such that $v'_i$ belongs to $H_y$.  Since both $G[H_x]$ and $G[H_y]$ are anticonnected and contain at least two vertices, there exists $u \in H_x$ non adjacent to $v_i$ and $v' \in H_y$ non adjacent to $v'_i$. We may assume that $u'$ is not adjacent to $y$ and that $v$ is not adjacent to $x$ since else we are in the previous case. So now $v$ should be anticomplete to $H_x$ and $u'$ should be anticomplete to $H_y$. So $\{v_i, u,v\}$ and $\{v'_i, u',v'\}$ are both stable sets of $G$, this is possible if and only if $\{v_i, u,v\} \subseteq A_S$. We have then two paths $xuP_uu'w'y$ and $xwvP_vv'y$ of length $\ell+1$. It remains to consider the case where each $v_i \in H_x$ is such that $v'_i \notin H_y$. By definition each of $H_x$ and $H_y$ contains a pair of non adjacent vertices and hence there exist two distinct vertices  $u,v'$ such that $u \in H_x \cap A_S$ and $v' \in H_y \cap A'_S$.  We have then again two paths $xuP_uu'w'y$ and $xwvP_vv'y$ of length $\ell+1$.\end{proof}

\subsection{Connecting vertices of $F_1$}

We here explain how  lemmas of Subsection~\ref{subs:connectF0Even} are
extended from $F_0$ to $F_1$.

\begin{lemma}
  \label{l:patatesBlowUpEven}
  If $u^* \in A^*\cup B^*$ and $v^*\in A'^* \cup B'^*$, then there exists
  in $F_1$ a path $P^*$ of length $\ell-2$, $\ell-1$, $\ell$ or $\ell +1$ from $u^*$
  to $v^*$ that contains the interior of a principal path $P$. More specifically:
  \begin{itemize}
  \item
    If $u^*\in A^*$ and $v^*\in A'^*$, then $P^*$ has length $\ell-2$, $\ell-1$,
    $\ell$ or $\ell+1$.
  \item If $u^*\in A^*$ and $v^*\in B'^*$, or if $u^*\in B^*$ and $v^*\in A'^*$,
    then $P^*$ has length $\ell-1$, $\ell$ or $\ell+1$.
  \item If $u^*\in B^*$ and $v^*\in B'^*$, then $P^*$ has length
 $\ell$ or $\ell+1$. Furthermore in that case there exists another path $Q^*$ from $u^*$
  to $v^*$ of length
   $\ell$ or $\ell+1$ which contains the interior of a principal path  $Q \neq P$.
  \end{itemize}
\end{lemma}

\begin{proof}
  Let $u$ and $v$ be such that $u^*\in K_u$ and $v^*\in K_v$.  Let $P$
  be a path in $F_0$ like in Lemma~\ref{lt:cheminsentrepatatesEven}
  from $u$ to $v$ (so $P$ contains the interior of some principal path $Q$). By Lemma~\ref{lb:universalevenforblow-upEven},
  $\{u^*, v^*\} \cup V(P) \sm \{u, v\}$ induces a path of the same length
  as $P$ that contains the interior of $Q$.
\end{proof}

%\begin{lemma}\label{lt:patateBBPblowup}
%  If $u^*\in B^*$ and $v^*\in B'^*$, then there exist in $G$ two paths
%  $P^*$ and $Q^*$  from $u^*$ to $v^*$ both of length at most $\ell +1$ such
%  that $P^*$ (resp.\ $Q^*$) contains the interior of a principal path
%  $P$ (resp.\ $Q$), and $P\neq Q$.
%\end{lemma}
%
%\begin{proof}
%  Let $u$ and $v$ be such that $u^*\in K_u$ and $v^*\in K_v$.  Let
%  $P = u\dots v$ and $Q=u\dots v$ be as in the
%  conclusion of Lemma~\ref{lt:cheminsentrepatatesBBP}. By
%  Lemma~\ref{lb:universalevenforblow-upEven},
%  $\{u^*, v^*\} \cup V(P) \sm \{u, v\}$ and
%  $\{u^*, v^*\} \cup V(Q) \sm \{u, v\}$ are the desired paths.
%\end{proof}

% suspect path
\begin{lemma}
  \label{l:suspectEven}
   If in $G$ some vertex $x$ is adjacent to the ends of a path $P$ of length
  at most $\ell+1$ not containing $x$, then $x$ is complete to $V(P)$.  
\end{lemma}

\begin{proof}
  Otherwise, a shortest cycle in $G[V(P) \cup \{x\}]$ has length at
  least~4 and at most $\ell + 3$. Since
  $\ell\geq 4$ implies $\ell+3 < 2 \ell$, this is a contradiction.
\end{proof}

\subsection{Attaching a vertex to $F_1$}

In this subsection, we show that for all vertices $x$ of $V(G)\sm V(F_2)$,
$N_{V(F_1)}(x)$ is a clique (see Lemma~\ref{l:attachVertex10Even}).  

\begin{lemma}
  \label{l:blowupInBEven}
  If $x\in V(G)\sm V(F_2)$ has no neighbor in $I^*$, then $N_{V(F_1)}(x)$ is a 
  clique.
\end{lemma}

\begin{proof}
  Suppose for a contradiction that $N_{V(F_1)}(x)$ is not a 
  clique. 

  \begin{claim}
    \label{c:bBnoNeigh}
    We may assume that $N_{V(F_1)}(x) \subseteq A^*\cup B^*$.  
  \end{claim}

  \begin{proofclaim}
    If $x$ has neighbors in both $A^* \cup B^*$ and $A'^* \cup B'^*$,
    then consider a path $P$ as in Lemma~\ref{l:patatesBlowUpEven} from a
    neighbor of $x$ in $A^* \cup B^*$ to a neighbor of $x$ in
    $A'^* \cup B'^*$.  By Lemma~\ref{l:suspectEven}, $x$ is complete to
    $V(P)$.  This is a contradiction since $x$ has no neighbor in
    $I^*$.  Hence $x$ does not have neighbors in both $A^* \cup B^*$ and
    $A'^* \cup B'^*$, and our claim follows up to symmetry.
  \end{proofclaim}

  \begin{claim}
    \label{c:bCondPT}
    There exist non-adjacent $a, b\in A$ such that $x$ has neighbors
    in both $K_a$ and $K_b$.  
  \end{claim}
  
  \begin{proofclaim}
        
    By Lemma~\ref{evenl:onlyKaKb}, since $N_{V(F_1)}(x)$ is not a clique,
    there should exist two non-adjacent vertices $a,b \in V(F_0)$ such
    that $x$ has a neighbor $a^* \in K_a$ and a neighbor
    $b^* \in K_b.$ By~\eqref{c:bBnoNeigh}, $a,b \in A\cup B.$

    If $a, b\in A$, then our conclusion holds, so we may assume that
    $b\in B$.

    If $a\in A$, then since $ab\notin E(G)$, $H_b$ is anticomplete to
    $a$. Let $P^*_a$ be the path induced by
    $\{a^*\} \cup (V(P_a) \sm \{a\})$. Let $v\in H_b$.  We may assume
    that $xv\notin E(G)$ for otherwise our claim holds (with $a$ and
    $v$). Note that since $ab, av \notin E(G),$ by~\eqref{i:bKKanti} of blowup, $a^*b^*, a^*v \notin E(G)$.
    Now,  $a^*xb^*vP_vv'a'P_a^*a^*$ (in case one of $a,v$ belongs to $A_K$)  or $a^*xb^*vP_vv'w'a'P_a^*a^*$ (in case both $a,v$ belong to $A_S$ ) form a hole of
    length  $2\ell +1$, a contradiction.  Hence, we may assume
    $a\in B$. 

    Since $ab\notin E(G)$, by Lemma~\ref{evenlt:deuxsommetsdeB},
    $\{a\}\cup H_a$ is
    anticomplete to $\{b\} \cup H_b$.  We may assume that $x$ is
    anticomplete to $H_a\cup H_b$ for otherwise we may apply the proofs
    above. Hence, for $u\in H_a$ and $v\in H_b$, $ua^*xb^*vP_vv'u'P_uu$ or $ua^*xb^*vP_vv'w'u'P_uu$ is
    a hole of length $2\ell+2$. 
  \end{proofclaim}

  Now the sets $K_u$ for all $u\in A\cup A' \cup I$, $B^*\cup \{x\}$
  and $B'^*$ form a preblowup of $F_0$.  All conditions are easily
  checked. In particular $x$ satisfies condition~\eqref{pb:B}
  by~\eqref{c:bBnoNeigh} and~\eqref{pb:BAN} by~\eqref{c:bCondPT}).
  So, by Lemma~\ref{l:preblowup_even}, $G[V(F_1)\cup \{x\}]$ is a proper
  blowup of some $\ell$-template with $k$ principal paths. This contradicts
  the maximality of $F_1$.
\end{proof}

\begin{lemma}
  \label{l:blowupInAEven}
  If there exist $x\in V(G)\sm V(F_2)$ and $u\in A$ such that $x$
  has neighbors in both $K_u$ and $K_{u^{+}}$ and is anticomplete to
  $K_{u^{++}}$, then $N_{V(F_1)}(x)$ is a clique.
\end{lemma}

\begin{proof}
  Suppose for a contradiction that $N_{V(F_1)}(x)$ is not a clique.

  \begin{claim}
    \label{c:bAxanti}
    $x$ is anticomplete to $A'^* \cup B'^* \cup (I^*\sm K_{u^+})$.
  \end{claim}
  
  \begin{proofclaim}
    If $x$ has a neighbor $t^*$ in some $K_t$ such that
    $t\in (A' \cup I) \sm \{u^+\}$, then note that $t\neq u^{++}$ by
    assumption.  Let $C$ be a principal hole that contains $t$
    and~$u$.  By Lemma~\ref{evenl:xKaKb} applied to $t$ and $u$ (that are by definition non adjacent in $C$) we should have that $u^+$ is a common neighbor of $u$ and $t$ in $C$. This is not possible since the neighbors of $u^+$ in $C$ are $u$ and $u^{++} \neq t$. So $x$ is anticomplete to $A'^* \cup (I^*\sm K_{u^+})$.
%     because by~\eqref{i:bKKanti} of blowup 
%    $u$, $u^+$ and $t$ cannot be consecutive along $C$.

    It remains to prove that $x$ is anticomplete to $B'^*$. Otherwise,
    $x$ has a neighbor $t\in B'^*$. Consider a path $P$ from $t$ to
    a neighbor of $x$ in $K_u$ as in Lemma~\ref{l:patatesBlowUpEven} and let $Q$ be the principal
    path whose interior is contained in~$P$.  By
    Lemma~\ref{l:suspectEven}, $x$ is complete to $V(P)$ and hence to $V(Q)$.  
    This is impossible since we have shown that $x$ is anticomplete to $A'^*$.
  \end{proofclaim}

   From here on, $u^*$ and $u^{+*}$ are neighbors of $x$ in respectively $K_{u}$ and $K_{u^+}$.

  \begin{claim}\label{c:wstaricomplete}
   $x$ has a neighbor $y^* \in B^*$ that is
    complete to $A^*$.
  \end{claim}

 \begin{proofclaim}
Note that $x$ has a neighbor
  $y^*\in K_y$ for some $y \in A \cup B\sm \{u\}$, for otherwise,
  by~\eqref{c:bAxanti}, $N_{V(F_1)}(x) \subseteq K_u\cup K_{u^+}$ and
  by Lemma~\ref{evenl:onlyKaKb}, $N_{V(F_1)}(x)$ is a clique, a
  contradiction. In case $y^*=w$ we are done, so from now on we will assume that $w$ is not adjacent to $x$ and $y^*\neq w$.
  
  Let $v\in A \sm \{u\}$ and assume that there exists $v^*\in K_v$ which not adjacent to $y^*$. Then $xv^*\notin E(G)$ for
    otherwise $\{x,y^*,w,v^*\}$ induces a $C_4$. Let $P^*_v$ be the path induced by
    $\{v^*\} \cup (V(P_v) \sm \{v\})$. Either
    $xy^*wv^*P^*_vv'w'u'P_uu^{++}u^{+*}x$ 
    (in case $v'u' \notin E(G)$) or
    $xy^*wv^*P^*_vv'u'P_uu^{++}u^{+*}x$ (in case $v'u' \in E(G)$) is a hole of length at least $2\ell+1$. In both cases we get a contradiction. So we have shown that for every $v\in A \sm \{u\}$, $y^*$ is complete to $K_v$ and hence $y$ is complete to $A \sm \{u,y\}$.
    
 We also have that $u^*y^*\in E(G)$ for otherwise, $\{x,y^*,w,u^*\}$ induces a $C_4$. So, if $y^*\in A^*$ then $A$ contains no isolated vertex, a contradiction to Lemma \ref{l:wProperEven}. So $y^*\in B^*$. If $y^*$
has a non-neighbor $v^*\in A^*$ then from what precedes we have $v^*\in K_u$. Now there exists
    a path $Q$ of length 1, 2 or 3 from $x$ to $v^*$ with interior in
    $K_{u^+}$ (either $xv^*$ or $x u^{+*}v^*$, or  $xu^+v^*$ or $xu^{+*}u^+v^*$).
    Hence, $xQv^*wy^*x$ is a hole of length~4, 5 or~6, a
    contradiction. \end{proofclaim}

  \begin{claim}  
    \label{c:bAneigheq}
    $N_{A}(x) \sm \{u\} = N_{A}(u)$.
  \end{claim}
  
  \begin{proofclaim}
    If there exists $v\in N_{A}(x)\sm N_{A}[u]$, then
    $vP_vv'u'P_uu^{++}u^{+*}xv$ (in case exactly one of $u,v$ is in $A_K$) or $vP_vv'w'u'P_uu^{++}u^{+*}xv$ (in case $u,v$ are both in $A_S$)  is a hole of length $2\ell-1$, a
    contradiction.

    Conversely, suppose there exists $v\in N_{A}(u) \sm N_{A}(x)$.  
%    We
%    claim that there exists a path $Q$ of length~2 from $x$ to
%    some $z\in N_A(u)$  with interior in $(A^* \cup B^*) \sm (K_u \cup K_z)$.  
    
%    If $w\in B$, then we may choose $z=v$ and $Q=xw^*z$
%    by~\eqref{c:wstaricomplete}.

%    Otherwise, $w\in A$.  So, by Lemma~\ref{l:wProper}, $G[A]$
%    contains at least two universal vertices.  So, let
%    $t\in A\sm \{u, v\}$ be adjacent to $u$ and $v$ (if $u$ and $v$
%    are the universal vertices of $G[A]$, $t$ can be any vertex of
%    $A\sm \{u, v\}$ and otherwise choose $t$ to be a universal
%    vertex).

%    If $x$ has a neighbor $t^*$ in $K_t$, then we choose $Q=xt^*v$.
%    So, suppose $x$ is anticomplete to $K_t$ (in particular,
%    $y\neq t$).  If $x$ has a neighbor $v^*$ in $K_v$, then we choose
%    $Q=xv^*t$.  So, suppose $x$ is anticomplete to $K_v$ (in
%    particular, $y\neq v$).  Now, by the way we chose $v$ and $t$, one
%    of $v$ or $t$ is a universal vertex of $G[A]$ and therefore a
%    universal vertex of $G[A^*\cup B^*]$.  So, we may choose $Q=xy^*v$
%    or $Q=xy^*t$.
%
%    So, our claim is proved. 
Then $v'P_vvy^*xu^{+*} u^{++}P_uu'w'v'$ (in case exactly one of $u,v$ is in $A_K$) or $v'P_vvy^*xu^{+*} u^{++}P_uu'v'$ (in case $u,v$ are both in $A_K$) is
    a hole of length $2\ell+1$, a contradiction.
  \end{proofclaim}

  \begin{claim}\label{c:EtLaClique}
    $x$ is complete to $K_u$.
  \end{claim}
  
  \begin{proofclaim}
    Suppose there exists $r\in K_u$ such that $rx\notin E(G)$.  
%    We
%    claim that $x$ and $r$ have a common neighbor $z$ in $(A^*\cup B^*)
%    \sm K_u$.
%
%    If $w\in B$, then $rw^*\in E(G)$ by \eqref{c:wstaricomplete} so we
%    may choose $z=w^*$.  If $w\in A$, then by Lemma~\ref{l:wProper},
%    some vertex $z\in A\sm \{u\}$ is a universal vertex of $G[A]$, and
%    by~\eqref{c:bAneigheq}, $z$ is adjacent to $x$.  So, $z$ exists as
%    claimed.

    If $xu^+\in E(G)$ then $\{r,y^*,u^+,x\}$ induces a $C_4$, a
    contradiction. Hence $xu^+\notin E(G)$.  Now by condition
    \eqref{i:flat} of blowups, either $\{x,y^*,r,u^{+*}\}$ induces a
    $C_4$ or $\{x,y^*,r,u^+,u^{+*}\}$ induces a $C_5$.
  \end{proofclaim}

  Now, the sets $K_v$ for all $v\in (A\sm{u})\cup I\cup A'$,
  $K_u\cup \{x\}$, $B^*$ and $B'^*$ form a preblowup of $F_0$.  All
  conditions are easy to check. In particular, $K_u\cup \{x\}$ is a
  clique by \eqref{c:EtLaClique}, conditions \eqref{pb:A},
  \eqref{pb:B} and \eqref{pb:I} follows from \eqref{c:bAxanti},
  condition \eqref{pb:Acomp} from \eqref{c:bAneigheq}, condition
  \eqref{pb:Bw} from \eqref{c:wstaricomplete} and condition
  \eqref{pb:AI} from our assumptions.

  Hence by Lemma~\ref{l:preblowup_even}, $G[V(F_1)\cup \{x\}]$ is a
  proper blowup of some twinless odd $\ell$-template with $k$
  principal paths that is an induced subgraph of $G$ a contradiction
  to the maximality of $F_1$.
\end{proof}

\begin{lemma}
  \label{l:blowupInIEven}
  If $x\in V(G)\sm V(F_2)$ has no neighbor in $B^* \cup B'^*$, then
  $N_{V(F_1)}(x)$ is a clique.
\end{lemma}

\begin{proof}
  Suppose for a contradiction that $N_{V(F_1)}(x)$ is not a clique.  By
  Lemma~\ref{l:blowupInBEven}, $x$ has a neighbor in $I^*$.  So $x$ has a
  neighbor in a clique blown up from an internal vertex of some
  principal path $P_v = v\dots v'$.  Let $a$ (resp.\ $b$) be  the
  vertex of $P_v$ closest to $v$ (resp.\ to $v'$) along $P_v$ and such
  that $x$ has a neighbor in $K_a$ (resp.\ $K_{b}$).

  Suppose first that $a=b$, so $a \in I$ and $a$ is the only vertex of $P_v$ whose clique contains a neighbor of $x$.  Hence, as $N_{V(F_1)}(x)$ is not a clique, $x$ has a neighbour in
  some $K_y$ with $y \in V(F_0) \sm V(P_v)$, and since by assumption
  $x$ has no neighbor in $B^* \cup B'^*$, $y\in A\cup A' \cup I$. So,
  $y$ and $a$ are non-adjacent members of some principal hole.  By
  Lemma~\ref{evenl:xKaKb}, $x$ has a neighbor in some clique $K_{c}$
  where $c$ is adjacent to both $a$ and $y$, a contradiction to the properties implied by the supposition that $a=b$.  
  
  Suppose now that $ab\in E(G)$. Then $a$ and $b$ are the only vertices of $P_v$ whose cliques contain a neighbor of $x$.
   If both $a$ and $b$ are internal
  vertices of $P_v$, then as in the previous paragraph, we can show that no neighbour of $x$ is in
  some $K_y$ with $y \in V(F_0) \sm V(P_v)$ and hence $N_{V(F_1)}(x) \subseteq K_a \cup K_{b}$. 
  So, by Lemma~\ref{evenl:onlyKaKb}, $N_{V(F_1)}(x)$ is a clique, a
  contradiction.
  It follows that at least one of $a$ or $b$ is an
  end of $P_v$.  Up to symmetry, we may assume that $a=v$ and
  $b=v^+$.  Note that, by the definition of $b$, $x$ is then anticomplete to $K_{v^{++}}$.  Hence, by
  Lemma~\ref{l:blowupInAEven}, $N_{V(F_1)}(x)$ is a clique, a contradiction.

  Hence, $a\neq b$ and $ab\notin E(G)$.  So, by Lemma~\ref{evenl:xKaKb},
  $a$ and $b$ have a common neighbor $u$ in $P_v$. So, $a$, $u$ and
  $b$ are consecutive along $P_v$ (in particular, $u\in I$).

  \begin{claim}
    \label{c:xuInE}
    $x$ is complete to $K_u$. 
  \end{claim}

  \begin{proofclaim}
    Otherwise, let $u^*\in K_u$ be non-adjacent to $x$. There exists
    a path $Q_a$ of length 2 or 3 from $u^*$ to $x$ with interior in
    $K_a$ (either $xa^*u^*$, or $xa^*au^*$ for some $a^*$ in $K_a$).
    There exists a similar path $Q_b$.  So, $Q_a$ and $Q_b$ form a
    hole of length 4, 5 or 6, a contradiction.
  \end{proofclaim}

  \begin{claim}
    \label{c:BlowIanti}
    $x$ is anticomplete to $V(F_1) \sm (K_a \cup K_u \cup K_b)$. 
  \end{claim}

  \begin{proofclaim}
    This follows from Lemma~\ref{evenl:xKaKb} and from the assumption that $x$ is anticomplete to $B^* \cup B'^*$.
  \end{proofclaim}

  \begin{claim}
    \label{c:xInotIso}
    $x$ has neighbors in each of $K_a$, $K_b$.  
  \end{claim}

  \begin{proofclaim}
    This follows from the definition of $a$ and $b$. 
  \end{proofclaim}

  Now the sets $K_v$ for all $v\in (A \cup A' \cup I) \sm \{u\}$, $K_u\cup \{x\}$,
  $B^*$ and $B'^*$ form a preblowup of $F_0$. All conditions are
  easily checked, in particular $K_u\cup \{x\}$ is a clique
  by~\eqref{c:xuInE}, it satisfies condition~\eqref{pb:I}
  by~\eqref{c:BlowIanti} and condition~\eqref{pb:II}
  by~\eqref{c:xInotIso}.
  
  Hence, by Lemma \ref{l:preblowup_even} $G[V(F_1)\cup \{x\}]$ is a
  proper blowup of some twinless odd $\ell$-template with $k$
  principal paths that is an induced subgraph of $G$. This contradicts
   the maximality of $F_1$.
\end{proof}

\begin{lemma}
  \label{l:attachVertex10Even}
  For all vertices $x$ of $V(G)\sm V(F_2)$, $N_{V(F_1)}(x)$ is a clique. 
\end{lemma}

\begin{proof}
  Suppose for a contradiction that $N_{V(F_1)}(x)$ is not a clique.

  \begin{claim}
    \label{l:atMost1internal}
    There exists a principal path $P_u= u\dots u'$ of $F_0$ such that
    $x$ is anticomplete to $I^*\sm \bigcup_{v\in V(P_u)} K_v$.
  \end{claim}

  \begin{proofclaim}
    Otherwise, there exist two distinct principal paths $P$ and $Q$ of $F_0$ such that
    $a$ in the interior of $P$, $b$ in the interior of $Q$ and
    $x$ has neighbors in both $K_a$ and~$K_b$. Note that then $a$ and $b$ are non adjacent and do not share any neighbour. This contradicts
    Lemma~\ref{evenl:xKaKb}, applied to the principal hole $C$ of $F_0$ containing $P$ and
    $Q$.  
  \end{proofclaim}

  \begin{claim}
    \label{l:notBBP}
    We may assume that $x$ has no neighbor in $B'^*$ and has a
    neighbor $y^*\in K_y$ where  $y\in B$.
  \end{claim}

  \begin{proofclaim}
    Suppose that $x$ has a neighbor $u^*\in B^*$ and a neighbor
    $v^*\in B'^*$.  Let $P$ and $Q$ be as in
    Lemma~\ref{l:patatesBlowUpEven}.  By Lemma~\ref{l:suspectEven}, $x$ is
    complete to both $V(P)$ and $V(Q)$.  In particular, $x$ has
    neighbors in the interior of two distinct principal paths, a
    contradiction to~\eqref{l:atMost1internal}. So, up to symmetry, we
    may assume that $x$ has no neighbor in $B'^*$.  Hence, by
    Lemma~\ref{l:blowupInIEven}, $x$ has neighbors in $B^*$.
  \end{proofclaim}

  \begin{claim}
    \label{c:pathI}
    $x$ is adjacent to $u$ and $u^+$ and has a neighbor in
    $K_{u^{++}}$. Moreover, $x$ is anticomplete to
    $(A^*\cup I^* \cup A'^* \cup B'^*) \sm (K_u\cup K_{u^+} \cup
    K_{u^{++}})$.
  \end{claim}

  \begin{proofclaim}
    By Lemma~\ref{l:blowupInBEven}, $x$ has at least one neighbor in $I^*$
    and by~\eqref{l:atMost1internal}, such a neighbor is in a clique
    blown up from an internal vertex of $P_u$.  So, let $v$ be the
    vertex of $P_u$ closest to $u'$ along $P_u$ such that $x$ has a
    neighbor $v^*\in K_v$. So $v\neq u$ and $v\in A'\cup I$. We set
    $Q = y^*uP_uv$ if $y^*u\in E(G)$ and $Q=y^*wuP_uv$ otherwise.  Let
    $Q^*$ be the path induced by $\{v^*\} \cup (V(Q) \sm \{v\})$ and
    observe that $Q^*$ has length at most $\ell +1$.  By
    Lemma~\ref{l:suspectEven}, $x$ is complete to $Q^*$.  If
    $v\notin \{u^+, u^{++}\}$, then $x$ has neighbors in at least~4
    cliques blown up from vertices of $P_u$ and this contradicts
    Lemma~\ref{evenl:xKaKb}.  If $v=u^+$, $x$ is adjacent to $u$ (since
    $x$ is complete to $Q^*$) and anticomplete to $K_{u^{++}}$, so by Lemma~\ref{l:blowupInAEven},
    $N_{V(F_1)}(x)$ is a clique, a contradiction.  So, $v=u^{++}$,
    meaning that $x$ is adjacent to $u$ and $u^+$, and is anticomplete
    to $I^* \sm (K_{u^+} \cup K_{u^{++}})$
    by~\eqref{l:atMost1internal}.
    
    If $x$ has neighbors in some $K_a$ for $a\in A\sm \{u\}$ then $x$
    and $C_{u,a}$ contradict Lemma~\ref{evenl:xKaKb}. Hence $x$ is
    anticomplete to $A^*\sm \{K_u\}$.

    By~\eqref{l:notBBP}, $x$ is anticomplete to $B'^*$.  It remains to
    check that $x$ is anticomplete to $A'^*\sm K_{u^{++}}$.  So,
    suppose $x$ has a neighbor $z^*$ in some $K_z$ where
    $z\in A'\sm \{u^{++}\}$. Then a principal hole that contains $z$
    and $u$ contradicts Lemma~\ref{evenl:xKaKb}.
  \end{proofclaim}

  Let $u^{++*}$ be a neighbor of $x$ in $K_{u^{++}}$ and $P_u^*$ be
  the path induced by $(V(P_u)\sm \{u^{++}\}) \cup \{u^{++*}\}$.

  \begin{claim}
    \label{c:yTwinu}
    For every $z\in B$ such that $x$ is adjacent to some $z^*$ in
    $K_z$ we have $N_A(z) = N_A[u]$ (in particular $N_A(y) = N_A[u]$).
  \end{claim}

  \begin{proofclaim}
    Suppose there exists $v\in N_A(z) \sm N_A[u]$.  By
    condition~\eqref{i:Kcomp} or~\eqref{i:optAB} of blowups,
    $vz^*\in E(G)$.  So, by \eqref{c:pathI},
    $xz^*vP_vv'u'P^*_uu^{++*}x$ (if one of $u,v \in A_K$) or $xz^*vP_vv'w'u'P^*_uu^{++*}x$ (if both $u$ and $v$ are in $A_S$) is a hole of length $2\ell-1$, a
    contradiction.  This proves that
    $N_A(z) \subseteq N_A[u]$.  In particular, $u$ has at
    least one neighbor in $H_z$, so by condition~\eqref{b:makeit} of
    templates, $uz\in E(G)$.
    
    Suppose there exists $v\in N_A(u) \sm N_A(z)$ (so $z\neq w$). By
    condition~\eqref{i:bKKanti} of blowups, $vz^*\notin
    E(G)$. By~\eqref{c:pathI}, $xv\notin E(G)$.  Hence $xw\in E(G)$, for otherwise $xz^*wvP_vv'u'P_u^*u^{++*}x$ or 
    $xz^*wvP_vv'w'u'P_u^*u^{++*}x$ is a hole of length $2\ell +1$. Since the partition is proper there exists
   a vertex $c\in A_S$  which is isolated in $G[A]$ and, since $v \in N(u),$  $c \neq u$. Again
    by~\eqref{c:pathI}, $xc\notin E(G)$ and $xwcP_cc'u'P_u^*u^{++*}x$ (in case $u \in A_K$) or $xwcP_cc'w'u'P_u^*u^{++*}x$ (in case $u\in A_S$) 
    is a hole of length $2\ell-1$, a contradiction.
  \end{proofclaim}

  \begin{claim}
    \label{c:xanticomAB}
    $N_{F_1}(x)\subseteq K_{u^{++}}\cup K_{u^+}\cup K_u\cup K_y$
  \end{claim}
  
  \begin{proofclaim}
    By \eqref{c:pathI}
    $N_{F_1}(x)\subseteq K_{u^{++}}\cup K_{u^+}\cup K_u\cup
    B^*$. Suppose there exists $z^*\in K_z$ such that $xz^*\in E(G)$
    and $z\in B\sm \{y\}$.  By~\eqref{c:yTwinu}, $N_A(z)= N_A[u]$ and
    $N_A(y)= N_A[u]$.  So, by Lemma~\ref{evenl:wTwins}, $y$ and $z$ are
    twins of $F_0$, a contradiction.
  \end{proofclaim}

  \begin{claim}
    \label{c:yneqw}
    $y\neq w$. 
  \end{claim}

  \begin{proofclaim}
    If $y=w$, then by~\eqref{c:yTwinu},
    $N_A(w) = N_A[u]=A$ and so $u$ is a universal vertex of $G[A]$.  By Lemma~\ref{l:wProperEven}, there exists also at least one isolated vertex in $G[A]$, a contradiction to $\vert A\vert \ge 3$.
  \end{proofclaim}

  \begin{claim}\label{c:lemevoisindans1}
    $N_{K_y}(x)$ is complete to $N_A[u]$.
  \end{claim}
  
  \begin{proofclaim}
    By \eqref{c:yTwinu}, $N_A(y)=N_A[u]$. The result follows from
    conditions~\eqref{i:Kcomp} and~\eqref{i:optAB} of blowups.
  \end{proofclaim}

  \begin{claim}\label{c:Uneseuleclique1}
    $x$ is complete to $K_{u^+}$.
  \end{claim}
  
  \begin{proofclaim}
    By~\eqref{c:pathI}, $ux\in E(G)$.  Suppose for a contradiction
    that there exists $u^{+*}\in K_{u^+}$ non-adjacent to $x$. By
    condition~\eqref{i:flat} of blowups,
    $u^{+*}u, u^{+*}u^{++}\in E(G)$. Hence $xu^{++}\notin E(G)$ for
    otherwise $\{x,u^{++},u^{+*},u\}$ induces a $C_4$. But now, either
    $\{x,u^{++*},u^{+*},u\}$ induces a $C_4$ (if
    $u^{+*}u^{++*}\in E(G)$) or $\{x,u^{++*},u^{++},u^{+*},u\}$
    induces a $C_5$ (if $u^{+*}u^{++*}\notin E(G)$), a contradiction.
  \end{proofclaim}

  \begin{claim}\label{c:Uneseuleclique2}
    $K_u\cup K_y$ is a clique.
  \end{claim}
  
  \begin{proofclaim}
    Since by~\eqref{c:yTwinu} $N_A(y) = N_A[u]$, $u$ cannot be an
    isolated vertex of $H_y$.  Hence, $uy$ is a solid edge. So, by
    condition~\eqref{i:Kcomp} of blowups, $K_u$ is complete $K_y$. 
  \end{proofclaim}
  
  We define $B_0=B^*\sm N_{K_y}(x)$. 
  
  Now the sets $K_v$ for all $v\in (A\cup I\cup A')\sm \{u,u^+\}$,
  $K_u\cup N_{K_y}(x)$, $K_{u^+}\cup \{x\}$, $B_0$ and $B'^*$ form a
  preblowup of $F_0$. All conditions are easy to check. In particular,
  $K_u\cup N_{K_y}(x)$ is a clique by \eqref{c:Uneseuleclique2},
  $K_{u^+}\cup \{x\}$ is a clique by \eqref{c:Uneseuleclique1},
  conditions \eqref{pb:A}, \eqref{pb:B} and \eqref{pb:I} follows from
  \eqref{c:xanticomAB}, condition \eqref{pb:Acomp} from
  \eqref{c:lemevoisindans1}, condition \eqref{pb:AI} holds because $x$
  is complete to $N_{K_y}(x)$, condition \eqref{pb:II} follows from
  \eqref{c:pathI} and condition \eqref{pb:Bw} holds because of
  \eqref{c:yneqw}.

  Hence, by Lemma \ref{l:preblowup_even}, $G[V(F_1)\cup \{x\}]$ is a proper
  blowup of some twinless odd $\ell$-template with $k$ principal paths
  that is an induced subgraph of $G$, a contradiction to the maximality
  of $F_1$.
\end{proof}

\subsection{Attaching a component}

\begin{lemma}
  \label{l:attachCompEven}
  If $D$ is a connected component of $G\sm F_2$, then $N(D)$ is a
  clique. 
\end{lemma}

\begin{proof}
  Suppose that $N(D)$ is not a clique.  Since $D$ is a connected component of $G\sm F_2$ we have that $N(D)=N_{V(F_1)}(D)$. By Lemma~\ref{l:F2-F1CliqueEven},
  $N_{V(F_1)}(D)$ is not a clique.  So, there exist $a$ and $b$ in $D$
  such that $N_{V(F_1)}(a) \cup N_{V(F_1)}(b)$ is not a clique, and a
  path $P$ from $a$ to $b$ in $D$.  We choose $a$ and $b$ subject to
  the minimality of the length of $P$.  By
  Lemma~\ref{l:attachVertex10Even}, $a \neq b$ (so $P$ has length at
  least~1).

  We set $S^*_a = N_{V(F_1)}(a)$ and $S^*_b= N_{V(F_1)}(b)$.  By
  Lemma~\ref{l:attachVertex10Even}, $S^*_a$ and $S^*_b$ are both cliques.
  Note that possibly $S^*_a\cap S^*_b\neq \emptyset$. We denote by
  $\text{int}(P)$ the set of the internal vertices of $P$.  We set
  $S^*_{\circ} = N_{V(F_1)}(\text{int}(P))$.

  We set $S_a = \{t\in V(F_0) : S_a^*\cap K_t \neq \emptyset\}$. We
  define $S_b$ and $S_{\circ}$ similarly.  Note that $S_a$ is possibly
  not included in $S_a^*$, and the same remark holds for $S_b$ and
  $S_{\circ}$.

  \begin{claim}
    \label{c:PCL}
    There exist non-adjacent $x^*_a\in S^*_a$ and $x^*_b\in S^*_b$. Moreover,
    for all such $x^*_a$ and $x^*_b$, $x^*_aaPbx^*_b$ is a  path.
  \end{claim}

  \begin{proofclaim}
    The existence of $x^*_a$ and $x^*_b$ follows from the definition of $a$
    and $b$, and $x^*_a a P b x^*_b$ is a path because of the
    minimality of $P$.
  \end{proofclaim}

  \begin{claim}\label{c:attachpath:pas_de_voisin}
    $S^*_a\cup S^*_\circ$ and $S^*_b\cup S^*_{\circ}$ are cliques (in
    particular, $S^*_\circ$ is a (possibly empty) clique of $F_1$
    that is complete to both $S^*_a \sm S^*_{\circ}$ and
     $S^*_b \sm S^*_{\circ}$).
  \end{claim}
  
  \begin{proofclaim}
    If $S^*_a \cup S^*_{\circ}$ is not a clique, then let $x^*y^*$ be
    a non-edge in $S^*_a \cup S^*_{\circ}$.  Since $S^*_a$ is a clique
    by Lemma~\ref{l:attachVertex10Even}, we may assume
    $y^*\in S^*_{\circ}$.  By definition of $S^*_{\circ}$, $y^*$ has a
    neighbor in $\text{int}(P)$, and then $x^*,y^*$ and some subpath of $P$
    contradict the minimality of $P$.
The proof is similar for $S^*_b\cup S^*_{\circ}$.
 \end{proofclaim}

  Note that while $S^*_a\cup S^*_b$ is not a
  clique by assumption, it might be that $S_a\cup S_b$ is a clique
  (for instance when $S_a = \{u\}$, $S_b=\{v\}$ and $uv$ is an
  optional edge of $F_0$).
  
  \begin{claim}\label{long2} 
  If $S^*_{\circ} \neq \emptyset$ then any two non adjacent vertices $x \in S^*_a$ and $y \in S^*_b$ are at distance $2$ in $F_1$.
  \end{claim}
  
   \begin{proofclaim}
   Let $s^* \in S^*_\circ$. By \eqref{c:attachpath:pas_de_voisin}, $s^* \neq x, y$ and $xs^* y$ is a path in $F_1$.
   \end{proofclaim}
  
  \begin{claim}\label{c:pathF0}
    $S_a\cup S_\circ$ and $S_b\cup S_{\circ}$ are cliques of $F_0$ (in
    particular, $S_a$ and $S_b$ are (non-empty) cliques of $F_0$ and
    $S_\circ$ is a (possibly empty) clique of $F_0$ that is complete
    to both $S_a \sm S_{\circ}$ and $S_b \sm S_{\circ}$).
  \end{claim}
  
  \begin{proofclaim}
    If $S_a \cup S_{\circ}$ is not a clique, then let $xy$ be a
    non-edge of $S_a \cup S_{\circ}$.  Since
    $x\in S_a \cup S_{\circ}$, there exists
    $x^*\in K_x\cap (S^*_a \cup S^*_{\circ})$ and
    $y^*\in K_y\cap (S^*_a \cup S^*_{\circ})$.  By
    condition~\eqref{i:bKKanti} of blowups, since $xy\notin E(G)$,
    $K_x$ is anticomplete to $K_y$.  So, $x^*y^*\notin E(G)$, a
    contradiction to~\eqref{c:attachpath:pas_de_voisin}.

    The proof is similar for $S_b\cup S_{\circ}$. 
  \end{proofclaim}

For the next claim we use the path $P$ defined at the very beginning of the proof.

\begin{claim}\label{c:holeTSEven}
  If a hole $C$ of $F_1$ contains two non adjacent vertices
  $x\in S^*_a$ and $y\in S^*_b$, then $P$ and $C$ form either:
   \begin{itemize}
   \item a theta, and $S^*_a\cap V(C) = \{x\}$, $S^*_b\cap V(C) = \{y\}$ and the
      three paths, all of length $\ell$, are the two paths between $x$ and $y$ in
      $C$ and the path
      between $x$ and $y$ obtained by adding the edges $ax$ and $by$ to $P$ ; or
\item a prism, $S^*_a\cap V(C) = \{x, z\}$, $S^*_b\cap V(C) = \{y, t\}$, the triangles of the prism are $axz$ and $byt$ and the
      three disjoint paths all of length $\ell-1$, are either:
     \begin{itemize}
      \item the path $P$,  the shortest path between $x$ and $y$ in $C$ and the shortest path 
      between $z$ and $t$ in $C$ ; or
     \item the path $P$, the shortest path between  $x$ and $t$ in $C$ and the shortest path between $z$ and $y$ in $C$.

    \end{itemize}
     \end{itemize}

  \end{claim}
  
  \begin{proofclaim} 
  Note that since $S^*_a$ is a clique,
    $S^*_a\cap V(C)$ contains $x$ and at most one other vertex which
    should be adjacent to $x$. The same holds for $S^*_b$ and $y$.

    Let us assume that $S^*_\circ \cap V(C) \neq \emptyset$.  Then
    by~\eqref{long2}, there exists a unique
    vertex $t \in S^*_\circ \cap V(C)$, and $t$ is such that
    $S^*_a\cap V(C) \subseteq \{x, t\}$ and
    $S^*_b\cap V(C) \subseteq \{y, t\}$. Hence $C$ and $P$ form a
    proper wheel centered at $t$, a contradiction to
    Lemma~\ref{l:holeTruemperS}.  So,
    $S^*_\circ \cap V(C) = \emptyset$.

    If $a$ and $b$ have a common neighbor $t$ in $C$, then $x$ and $y$
    are the two neighbors of $t$ in $C$ and so, $C$ and $P$ form a
    proper wheel centered at $t$, again a contradiction to
    Lemma~\ref{l:holeTruemperS}. So the neighborhoods of $a$ and $b$
    in $C$ are disjoint.

    From the remarks above and Lemma~\ref{l:holeTruemperS}, we obtain that $C$ and $P$ form a theta whose three paths have length $\ell$ or a prism whose three paths have length $\ell-1$.  This can happen only
    if we are in one of the three cases described in \eqref{c:holeTSEven}.
    
  \end{proofclaim}

  \begin{claim}\label{c:pasI}
    $S_a\cap I = S_b \cap I = \emptyset$. 
  \end{claim}

\begin{proofclaim}
  Otherwise, up to symmetry, $S_a\cap I \neq \emptyset$. So, there
  exists a principal path $P_u=u\dots u'$ of $F_0$ whose interior
  intersects $S_a$. By~\eqref{c:pathF0}, $S_a$ is a clique, so
  $1\leq |S_a|\leq 2$ and $S_a\subseteq V(P_u)$.  
  We now break into three cases.

  \medskip
  
  \noindent{\bf Case 1:} $S_b\subseteq V(P_u)$. 
  
  By~\eqref{c:PCL} there exist vertices $x_a$ and
 $x_b$ of $P_u$ such that there exist non adjacent vertices $x^*_a\in S^*_a \cap K_{x_a}$ and $x^*_b\in S^*_b \cap K_{x_b}$.
 
 We first show that there exist such $x_a$ and $x_b$ that are not adjacent. Otherwise, and since $S_a\subseteq V(P_u)$ and $S_b\subseteq V(P_u)$, we have that  $S_a \cup S_b= \{x_a,x_b\}$. By replacing $x_a$ and $x_b$ by $x^*_a$ and $x^*_b$ in any principal hole $C$ containing $P_u$ we obtain a path $P_C$ of length $2\ell-1$, by \eqref{long2} $S^*_{\circ} \cap P_C = \emptyset$ and so $V(P_C) \cup V(P)$ induces a hole of length at least $2\ell+2$, a contradiction. So we may assume that $x_a$ and $x_b$ are not adjacent.

  %  Then $x^*_a$ and $x^*_b$ (as defined in~\eqref{c:PCL}) belong
%  respectively to $K_{x_a}$ and $K_{x_b}$ for some vertices $x_a$ and
%  $x_b$ of $P_u$. 
  Let $C$ be any principal hole of $F_0$ that contains
  $P_u$. By Lemma~\ref{lb:universalevenforblow-upEven},
  $\{x^*_a, x^*_b\} \cup (V(C) \sm \{x_a, x_b\})$ induces a hole
  $C^*$. Let us apply~\eqref{c:holeTSEven} to $C^*$, $x^*_a$ and
  $x^*_b$. We obtain that the shortest path in $C^*$ between $x^*_a$
  and $x^*_b$ has length $\ell$ or $\ell-1$. 
  However the path $P^*$  obtained from $P_u$ by replacing $x_a$ by $x^*_a$ and $x_b$ by
  $x^*_b$ is contained in $C^*$ and it has length at most $\ell-1$, 
 So wlog, $x_a=u$ and $x_b=u'$.   
  We should then be in the second situation described in \eqref{c:holeTSEven} and there should exist $z \in S^*_a\cap V(C)$, and $t \in S^*_b\cap V(C)$ such that the shortest path between them on $C^*$ has length $\ell-1$ and is disjoint from $P^*$, a contradiction to the assumption that  $S_a$ and $S_b\subseteq V(P_u)$.

  \medskip
  
  \noindent{\bf Case 2:} $S_b$ contains a vertex of some principal
  path $P_v$ distinct from $P_u$.  Up to symmetry, since $S_b$ is a
  clique (by~\eqref{c:pathF0}), we assume that $b$ is anticomplete to $K_{v'}$.

  Let $y$ be the vertex of $P_u$ closest to $u'$ such that $a$ has a
  neighbor $y^*\in K_y$.  Let $z$ be the vertex of $P_v$ closest to
  $v$ such that $b$ has a neighbor $z^*\in K_z$.  Possibly $y=u'$
  and $z=v$, but $z\neq v'$ and $y\neq u$ since $a$ has a neighbor in
  $I^*$ by assumption.  In particular, $yz\notin E(G)$.  By
  condition~\eqref{i:bKKanti} of blowups, $y^*z^*\notin E(G)$.

  Let $C$ be the principal hole of $F_0$ that contains $P_u$ and
  $P_v$.  By Lemma~\ref{lb:universalevenforblow-upEven},
  $\{y^*, z^*\} \cup (V(C) \sm \{y, z\})$ induces a hole $C^*$.
  Applying ~\eqref{c:holeTSEven} to $C^*$, $y^*$ and $z^*$, we obtain that
   $P$ has length $\ell-2$ or $\ell-1$.  We denote by $P^*_u$ the path
  obtained from $P_u$ by replacing $y$ by $y^*$, and by $P^*_v$ the
  path obtained from $P_v$ by replacing $z$ by $z^*$. Let $P^*$ be
  the path $vP^*_vz^*bPay^*P^*_uu'$ (in case $z=v$ one should replace $vP^*_vz^*$ by $z^*$, and in case $y= u'$ one should replace $y^*P^*_uu'$ by $y^*$).  The length of $P^*$ is at least
  $\ell$.

Consider now a vertex $r \in A\sm \{u,v\}.$ Depending on the adjacencies of $r$ with $v$ and of $r'$ with $u'$, one of
 $rvP^*u'w'r'P_rr$ or $rwvP^*u'w'r'P_rr$ or $rwvP^*u'r'P_rr$ or $rvP^*u'r'P_rr$ (with possibly $v$ replaced by $z^*$ in case $z=v$ and $u'$ replaced by $y^*$ in case $y=u$) is a cycle $C_r$ with at most one possible chord $br$ or $bw$ in case $z=v$. Hence the length of this cycle should be at most $2 \ell +1$. Since the length of $P_r$ is at least $\ell-2$ the case where $C_r= rwvP^*u'w'r'P_rr$ cannot occur and we get that at least one of $vr$, $u'r'$ is an edge of $G$ and so at least one of $u,v,r$ is in $A_K$.
We also notice that if the length of $P^*$ is at least $\ell +1$ then $vr$ and $u'r'$ should be edges of $G$, $r \in A_S,$ $u,v \in A_K$, $z=v$, $y=u'$ and $br$ is a chord of the cycle $rz^*P^*y^*r'P_rr$. 
This should be valid for any $r \in A\sm \{u,v\}$. However since the partition is proper, $A_S$ contains at least $2$ elements and we get a contradiction to the fact that $S_b$ is a clique.
%But then $u$ and $v$ should both belong to $A_K$ and so either the vertices of $P_u$ and $P^*$ or those of $P_v$ and $P^*$ induce a hole of length $2 \ell +1$ (remember that $S_b$ is a clique), a contradiction. 
Hence  $P^*$ has length $\ell$, $y=u'$, $z=v$ and $P$ has length $\ell-2$. 

So by~\eqref{c:holeTSEven}, $P$ and $C^*$ form a theta.   
However, since we have assumed $S_a\cap I \neq \emptyset$, $a$ has a neighbor $t^* \in K_{u^-}$ where $u^-$ is the neighbor of $u'$ in $P_u$. By Lemma~\ref{lb:universalevenforblow-upEven}, since $S^*_a$ is a clique, we have that
  $\{t^*\} \cup (V(C^*) \sm \{u^-\})$ induces a hole. This hole and $P$ form a pyramid, a contradiction to  \eqref{c:holeTSEven}.

  \medskip

  \noindent{\bf Case 3:} We are neither in Case~1 nor in Case~2.

  Since we are not in Case~1, $S_b$ contains a vertex of $F_0\sm P_u$,
  and since we are not in Case~2, this vertex must be in $B\cup B'$.
  Up to symmetry, we assume that $S_b\cap B\neq \emptyset$.  Since
  $S_b$ is a clique  (by~\eqref{c:pathF0}), $S_b\cap (B' \cup A' \cup I) = \emptyset$.  
  %Since we are not in Case~2, $S_b\cap (A \sm \{u\}) = \emptyset$.
  Hence, $S_b \subseteq B\cup \{u\}$ and there exist $x \in B \cap S_b$ and $x^* \in K_x \cap S^*_b$.

Let $u_a$ be the vertex of $S_a$ which is the closest to $u$ in
  $P_u$ and let $u'_a$ be the vertex of $S_a$ which is the closest to
  $u'$ in $P_u$. Notice that, since $S_a$ is a clique (by~\eqref{d:pathF0}), either $u_a=u'_a$ or $u_au'_a$ is an edge. So it may be that $u_a= u$ or $u'_a= u'$ but since $S_a \cap I \neq \emptyset$ we know that $u_a \neq u'$ and $u'_a \neq u$. Let now
  $u^*_a \in K_{u_a} \cap S^*_a$ and $u'^*_a \in K_{u'_a} \cap S^*_a$. We denote by $P^*_u$ the path obtained from $P_u$ by replacing $u_a$ by  $u^*_a$
  and, in case $u_a \neq u'_a$,  by replacing $u'_a$ by $u'^*_a$.
Notice that if $u^*_a \neq u'^*_a$ then
  $u^*_au'^*_a\in E(G)$ since $S_a^*$ is a clique. 
  
  Suppose that $u'_a = u^+$, where $u^+$ is the neighbor of $u$ in $P_u$. Since $H_x$ contains at least two vertices there exists $v \in H_x \setminus \{u\}$. By~\eqref{d:pathF0}, by the fact that $P$ contains at least one edge, and because $P_u$ and $P_v$ belong to a hole of $F_0$
  of length $2 \ell$ in which $u^+$ and $v$ are at distance at most $3$,  one of $aPbx^*vP_vv'u'P^*_uu'^*_aa$ or $aPbx^*vP_vv'w'u'P^*_uu'^*_aa$ is a hole of length at least $2\ell+1$, a contradiction. Hence  from now on, we may assume that $u_a \neq u$ (hence $a$ is not adjacent to $u$) and that if $u_a= u^+$ then $u'_a \neq u_a$. Now  by~\eqref{d:pathF0} we get that $S_\circ = \emptyset$.
  
%Suppose that $x$ is adjacent to $u$ in $F_0$. Depending on whether $b$ is
%adjacent to $u$ not, one of $u^*_aaPbuP^*_uu^*_a$ or
%$u^*_aaPbx^*uP^*_uu^*_a$  is a hole, implying that $P$ has length at
% least $\ell$. 
% Let us choose any vertex  $v\in H_x$ distinct from $u$ (since $H_x$ has cardinality at least 2, such a vertex do exist). 
% Then $aPbx^*vP_vv'(w')u'P_uu'^*_aa$ is a hole of length at least $2\ell +3$, a contradiction. Hence, from here on, we may assume that no vertex in
%  $B \cap S_b$ is adjacent to $u$.
%
%So $x$ is not adjacent to $u$ in $F_0$. Hence $x^*\neq w$, $u \neq w$ and $w \notin S_b$.  We also get that $u \notin H_x \cup N(H_x)$. To avoid a $C_4$ $bx^*wub$, $b$ is not adjacent to $u$
%  and $u^*_a a P b x^* w u P^*_u u^*_a$ is a hole implying that $P$ has length at
%  least $\ell-2$.  So, for any $v\in H_x$, the hole
%  $x^*bPau'^*_a P^*_uu' v' P_v v x^*$ (in case $u'_a=u'$ one should
%  replace $u'^*_a P^* u'$ by $u'^*_a$) has length at least $2\ell +2$,
%  a contradiction.

  Suppose first that  $u\in H_x$.  Let $v\in H_x$ be non-adjacent to
  $u$ (this is possible since $G[H_x]$ is anticonnected and contains more than one vertex). So, $P_u$,
  $P_v$, $x$ and possibly $w'$ form a hole $C$ (possibly not principal).  Let $z$ be the
  vertex in $I \cap S_a$ which is the closest to $u$ in $P_u$ and let
  $z^* \in K_z \cap S^*_a$.  By
  Lemma~\ref{lb:universalevenforblow-upEven},
  $\{x^*, z^*\} \cup V(C) \sm \{x, z\}$ induces a hole $C^*$ of $F_1$
  ($ux^*, vx^*\in E(G)$ by condition~\eqref{i:optAB} of blowups).  The
  distance between $x^*$ and $z^*$ in $C^*$ is at most $\ell-1$ since
  $P_u$ has length at most $\ell-1$, a contradiction to~\eqref{c:holeTSEven}
  applied to $C^*$, $x^*$ and $z^*$ (we cannot have a theta because $x^*$ and $z^*$ are too close in $C^*$ and we cannot have a prism because else $b$ would have a neighbor in  $K_v$, a contradiction to $S_b \subseteq B\cup \{u\}$). Hence, from here on, we may assume that no vertex
  $x\in B \cap S_b$ is such that $u\in H_x$.

%  Let $u_a$ be the vertex of $S_a$ which is the closest to $u$ in
%  $P_u$ and let $u'_a$ be the vertex of $S_a$ which is the closest to
%  $u'$ in $P_u$. Notice that, since $S_a$ is a clique (by~\eqref{c:pathF0}), either $u_a=u'_a$ or $u_au'_a$ is an edge. So it may be that $u_a= u$ or $u'_a= u'$ but
%  not both. Let now
%  $u^*_a \in K_{u_a} \cap S^*_a$ and $u'^*_a \in K_{u'_a} \cap S^*_a$.
%  We denote by $P^*_u$ the path obtained from $P_u$ by replacing $u_a$
%  and $u'_a$ by respectively $u^*_a$ and $u'^*_a$ (note that
%  $u^*_au'^*_a\in E(G)$ since $S_a^*$ is a clique).
  
  Suppose now that 
  $u\in N(H_x)$.   Since $S_b \subseteq B\cup \{u\}$,
  we have $S_b\cap H_x = \emptyset$.  Depending on whether $b$ is
  adjacent to $u$ or not, one of $u^*_aaPbuP^*_uu^*_a$ or
  $u^*_aaPbx^*uP^*_uu^*_a$ (remind that $u_a \neq u$) is a hole, implying that $P$ has length at
  least $\ell-1$.  Let $v\in H_x$,  then $uv$ is an edge (since $u\in N(H_x)$). So, $x^*bPau'^*_a P^*_u u' w' v' P_v v x^*$ or $x^*bPau'^*_a P^*_u u'v' P_v v x^*$
  (in case $u'^*_a=u'$ one should replace $u'^*_a P^*_u u'$ by $u'^*_a$)
  is a hole of length at least $2\ell +1$, a contradiction. 
  
  Hence, from here on, we may assume that no vertex in $B \cap S_b$ is adjacent to $u$ and so $w \notin B \cap S_b$, in particular $x \neq w$ and $x^*$ is not adjacent to $u$. Then to avoid a $C_4$ induced by $\{b,x^*,w,u\}$, $b$ is not adjacent to $u$ and $u^*_a a P b x^* w u P^*_u u^*_a$  is a hole, implying that $P$ has length at least $\ell-2$. 
So, for any $v\in H_x$, the hole
$x^*bPau'^*_a P^*_uu' v' P_v v x^*$ or the hole $x^*bPau'^*_a P^*_uu' w'v' P_v v x^*$ (in case $u'_a=u'$ one should
replace $u'^*_a P^* u'$ by $u'^*_a$) has length  $2\ell$ if and only if $P$ has length $\ell-2$, $u'_a=u'$, $u'$ is adjacent to $v'$ and $P_v$ has length $\ell-2$. This implies that $v \in A_S$ and $u \in A_K$. Furthermore, since $S_a\cap I \neq \emptyset$ and $S_a$ is a clique we get that $u_a$ is the neighbor of $u'$ on the path $P_u$. 

 Assume there exists $r \in A_S$ which is adjacent to $u$. Then $r \notin H_x$ and $rP_rr'w'u'^*_aaPbx^*wr$ is a hole of length $2\ell + 2$, a contradiction. Hence no such $r$ exists and  by setting $\mathbb A_K= (A_K \sm \{u\}) \cup \{x^*\}, \mathbb A_S= A_S \cup \{u\}, \mathbb B = \{w\}, \mathbb A'_K= (A'_K \sm \{u'\}) \cup \{a\}, \mathbb A'_S= A'_S \cup \{u^*_a\}, \mathbb B'= \{u'^*_a\}$ 
we obtain an even pretemplate partition contained in $G$ with $k+1$ paths (the new paths are $uP_uu^*_a$ of length $\ell-2$ and $x^*bPa$ of length $\ell-1$), a contradiction.

%  Since we are in Case~3 and by \eqref{c:PCL}, there exists some vertex $x\in B\cap S_b$
%  and by the two paragraphs above, $u\notin N[H_x]$ and
%  $S_b\cap N[H_x] = \emptyset$. Hence $x^*\neq w$ where
%  $x^* \in K_x \cap S^*_b$.  Then to avoid a $C_4$ $bx^*wub$, $b$ is not adjacent to $u$ and $u^*_a a P b x^* w u P^*_u u^*_a$  is a hole (in case $u_a=u$ one should
%  replace $uP^*_uu^*_a$ by $u^*_a$), implying that $P$ has length at
%  least $\ell-2$.  So, for some $v\in H_x$, the hole
%  $x^*bPau'^*_a P^*_uu' v' P_v v x^*$ or the hole $x^*bPau'^*_a P^*_uu' w'v' P_v v x^*$ (in case $u'_a=u'$ one should
%  replace $u'^*_a P^* u'$ by $u'^*_a$) has length at $2\ell$ if and only if $P$ has length $\ell-2$, $u'_a=u'$, $u'$ is adjacent to $v'$ and $P_v$ has length $\ell-2$. This implies that $u \in A_K$ and $v \in A_S$. Then in order to have the hole $u^*_a a P b x^* w u P^*_u u^*_a$ of length $2\ell$, $u_a= u^-$. 
%
%   So, by setting $\mathbb A_K= (A_K \sm \{u\}) \cup \{w,x^*\}, \mathbb A_S= A_S, \mathbb B = \emptyset, \mathbb A'_K= (A'_K \sm \{u'\}) \cup \{u'^*_a, a\}, \mathbb A'_S= A'_S, \mathbb B'=B' \cup \{u'\} $ 
%we obtain an even pretemplate partition contained in $G$ with $k+1$ paths (the new paths are $wuP_uu'^*_a$ and $x^*bPa$ each of length $\ell-1$), a contradiction.

\end{proofclaim}

\begin{claim}\label{c:attachpath:pas_trop_prche}
  We may assume that $S_a \subseteq A \cup B$ and
  $S_b \subseteq A' \cup B'$.  
\end{claim}

\begin{proofclaim}
  Otherwise, by~\eqref{c:pasI} and since $S_a$ and $S_b$ are cliques (by~\eqref{c:pathF0}),
  we may assume that $S_a, S_b \subseteq A\cup B$.

 We will show that there exists a path $Q^*$ of length at
  least $2\ell - 2$ whose union with $P$ induces a hole. This is a
  contradiction because it implies that $P$ has length at most $0$.
  So, to conclude the proof, it remains to prove the existence of
  $Q^*$.

  By~\eqref{c:PCL}, there exist non-adjacent $x^*_a\in S^*_a$ and $x^*_b\in S^*_b$ and
    for all such $x^*_a$ and $x^*_b$, $x^*_aaPbx^*_b$ is a  path.
 Let $x_a$ and $x_b$ be the vertices of $F_0$ such
  that $x^*_a \in K_{x_a}$ and $x^*_b \in K_{x_b}$.  Note that
  possibly $x_ax_b$ is an edge, but this happens only if $x_ax_b$ is
  an optional edge of $F_0$ (since $x^*_ax^*_b$ is not an edge). We break into three cases.

\medskip

  \noindent{\bf Case 1:} $x_a, x_b\in A.$
  
  Then $x_ax_b\notin E(G)$ (otherwise $x^*_a$ and $x^*_b$ would be adjacent) and at least one of $x_a,x_b$ belongs to $A_S$. So from the definition of templates, there exists a path $Q$
  of length $2\ell -2$ from $x_a$ to $x_b$ consisting in $P_{x_a}$ and $P_{x_b}$ joined either by an edge $x'_ax'_b$ or by a path $x'_aw'x'_b$. By Lemma~\ref{lb:universalevenforblow-upEven},
  $\{x^*_a, x^*_b\} \cup (V(Q) \sm \{x_a, x_b\})$ induces the path
  $Q^*$ that we are looking for.  Note that $Q^*$ and $P$ form a hole
  by~\eqref{c:attachpath:pas_de_voisin} and our assumption that
  $S_a, S_b \subseteq A\cup B$.

\medskip

\noindent{\bf Case 2:} $x_a \in A$ and $x_b \in B.$

Whether $x_ax_b$ is an optional edge or a non-edge, an immediate
consequence of the definition of a template is that there exists a
vertex $z\in H_{x_b}$ that is non-adjacent to $x_a$.  By \eqref{c:pathF0}, $S_a$ is a clique so $z\notin S_a$. We may
furthermore assume that $z\notin S_b$ since else we are in the same
situation as in Case 1. By definition of a template, there exists a
path $Q_0$ of length $2\ell -2$ between $x_a$ and $z$ consisting in $P_{x_a}$ and  $P_z$ joined either by an edge $x'_az'$ or by a path $x'_aw'z'$. Then by
Lemma~\ref{lb:universalevenforblow-upEven}, $\{x^*_a, x^*_b\} \cup (V(x_bzQ_0x_a) \sm \{x_a, x_b\})$ induces a path $Q^*$ of length
$2 \ell-1$. Note that $Q^*$ and $P$ form a hole
by~\eqref{c:attachpath:pas_de_voisin} and our assumption that
$S_a, S_b \subseteq A\cup B$, $z\notin S_b$ and $z\notin S_a$.  
\medskip

 \noindent{\bf Case 3:} $x_a, x_b\in B.$

 Then $x_ax_b\notin E(G)$ (otherwise $x^*_a$ and $x^*_b$ would be adjacent).  Hence, by Lemma~\ref{evenlt:deuxsommetsdeB},
 $H_{x_a} \cup \{x_a\}$ is anticomplete to $H_{x_b}\cup \{x_b\}$.  So,
 let $u_a\in H_{x_a}$ and $u_b\in H_{x_b}$, we may assume that $u_a, u_b \notin S_a\cup S_b$ since else we
 are in the  situation of Case 1 or 2. By definition of a template, there exists a
path $Q_0$ of length $2\ell -2$ between $u_a$ and $u_b$ consisting in $P_{u_a}$ and  $P_{u_b}$ joined either by an edge $u'_au'_b$ or by a path $u'_aw'u'_b$.  By Lemma~\ref{lb:universalevenforblow-upEven},
 $Q^*=x^*_au_aQ_0u_bx^*_b$ is also a path, it is of length $2\ell$.
  Now,
it is easy to verify that $Q^*$ and $P$ form a hole of length more than $2 \ell$, a contradiction..

   \end{proofclaim}

\begin{claim}\label{c:Pdisjoint}
  $S_\circ=\emptyset$.
\end{claim}

\begin{proofclaim}
  By~\eqref{c:pathF0} and~\eqref{c:attachpath:pas_trop_prche}, if
  $S_\circ\neq \emptyset$, then $\ell=4$, and there exists a principal
  path $P_u=ucu'$ of $F_0$ such that $S_a=\{u\},$ $S_b=\{u'\}$
  and $S_\circ = \{c\}$.  Let $u^*\in K_u\cap S^*_a$, $c^*\in S^*_\circ$ and
  $u'^*\in K_{u'} \cap S^*_b$.  Observe that by~\eqref{c:attachpath:pas_de_voisin}, each of $u^*c^*, u'^*c^*$ is an edge and by definition, $c^*$ has a neighbor in $\text{int}(P)$.

%  We claim that  $c^*u^*, c^*u'^*\in E(G)$.  If $c^*=c$ this
%  follows from condition~\eqref{i:flat} of blowups, so suppose
%  $c^*\neq c$. Then $P$, $c$, $c^*$, $u^*$ and $u'^*$ form a theta or
%  a non-twin wheel $W$ centered at $c^*$. So, by
%  Lemma~\ref{l:holeTruemperS} and the fact that $\ell = 4$, $W$ is a universal wheel and 
%  $c^*u^*, c^*u'^*\in E(G)$.

  Let $P_v=v\dots v'$ be any principal path distinct from $P_u$.  Now, $P_v$, $P$, $u^*$,
  $u'^*$, $c^*$ and possibly $w$ and/or $w'$ form a proper wheel centered at $c^*$, a
  contradiction to Lemma~\ref{l:holeTruemperS}.
\end{proofclaim}

\begin{claim}\label{c:lengthEven}
$(S_a \cup S_b) \cap (A_S \cup A'_S) = \emptyset$.
\end{claim}

\begin{proofclaim}
Notice that each of $S_a  \cap A_S$ and $S_b \cap A'_S$ contains at most one vertex since $S_a$ and $S_b$ are cliques.
\medskip

\noindent{\bf Case 1:}  $S_a$ contains $x \in A_S$ and $S_b$ contains $y' \in A'_S$. 

Let $x^* \in K_x \cap S^*_a$ and  $y'^* \in K_{y'} \cap S^*_b$. We denote by $P^*_x$ (resp. $P^*_y$) the path obtained from $P_x$ (resp. $P_y$) by replacing $x$ (resp. $y'$) by $x^*$ (resp. $y'^*$). 
If $x=y$ then $ax^*P^*_xx'bPa$ is a hole, so $P$ has length $\ell$. Since the partition is proper there exists an other vertex $r \in A_S$ and $ax^*wrP_rr'w'y'^*bPa$ is a cycle of length $2 \ell+4$ whose possible chords are $aw$ and $bw'$, in any case the cycle contains a too long hole. Hence $x \neq y$. 
 
Assume that $aw, bw' \notin E(G)$. 
Then $ax^*P^*_xx'w'y'^*bPa$ is a hole and hence $P$ has length $\ell-2$. We know that the template has at least one more principal path $P_r$. In case there exists $r \in A_S$ then $ax^*wrP_rr'w'y'^*bPa$ is a hole of length $2 \ell +2$, a contradiction. So any $r \in A \sm \{x,y\}$ belongs to $A_K$. If $rx$ and  $ry$ are in $E(G)$ then $ax^*rP_rr'w'y'^*bPa$ is a cycle of length $2 \ell +2$ whose only possible chord is $ar$, so we get a contradiction. So by symmetry we get that $r$ is adjacent to exactly one of $x$ and $y$ and $r'$ is adjacent to exactly one of $x'$ and $y'$. In case $r$ is adjacent to $y$ and not to $x$ then $ax^*wrP_rr'w'y'^*bPa$ is a hole of length $2 \ell+3$, a contradiction. So we may assume that $r$ is adjacent to $x$ and not to $y$ (and then $r'$ is adjacent to $y'$ and not to $x'$), and this is true for any $r \in A_K$. 

Let $x^+$ be the neighbour of $x$ in $P_x$ and $y^-$ be the neighbour of $y'$ in $P_y$. By setting $\mathbb A_K= A_K, \mathbb A_S= \{a,x^+,w\}, \mathbb B = \{x^*\}, \mathbb A'_K= A'_K , \mathbb A'_S= \{b,w',y^-\}, \mathbb B'=\{y'^*\}$ 
we obtain an even pretemplate partition contained in $G$ with $k+1$ paths (the new paths are $aPb$, $x^+P_xx'w'$, $wyP_yy^-$ each of length $\ell-1$), a contradiction. Hence $aw$ or $bw'$ is an edge of $G$.

\medskip

So by symmetry we may set that $aw\in E(G)$.
Then $awyP_y^*y'^*bPa$ is a hole, so $P$ has length $\ell-1$. If $bw'\notin E(G)$ then $ax^*P^*_xx'w'y'^*bPa$ is a hole of length $2\ell+1$, a contradiction. So  $bw' \in E(G)$.
Suppose $A_S$ contains a vertex $r \neq x,y$ then $awrP_rr'w'y'^*bPa$ is hole of length $2 \ell + 1$, a contradiction. Hence $A_S= \{x,y\}$ and $ar, br' \in E(G)$ for every $r \in A_K$.

So, by setting $\mathbb A_K= A_K\cup \{a\}, \mathbb A_S=\{x^*,y\}, \mathbb B =  \{w\}, \mathbb A'_K= A'_K \cup \{b\}, \mathbb A'_S= \{x',y'^*\}, \mathbb B'=\{w'\}$ 
we obtain an even pretemplate partition contained in $G$ with $k+1$ paths (the new paths are $aPb$ of length $\ell-1$ and $P^*_x, P^*_y$ of length $\ell-2$), a contradiction.

\bigskip

\noindent{\bf Case 2:} $S_a \cap A_S= \emptyset$, $S_a$ contains a vertex $x \in A_K$ and $S_b$ contains a vertex $y' \in A'_S$ (symmetric to the case where $S_b \cap A'_S= \emptyset$, $S_b$ contains a vertex $y' \in A'_K$ and $S_a$ contains a vertex $x \in A_S$). 

Let $x^* \in K_x \cap S^*_a$ and $y'^* \in K_{y'} \cap S^*_b$.
Suppose that $xy \in E(G)$. Then $ax^*yP^*_yy'^*bPa$ is a hole and hence $P$ has length $\ell-1$. Then $ax^*P^*_xx'w'y'^*bPa$ is a cycle with at most one possible chord ($bw'$) of length $2 \ell +2$, a contradiction. So $xy \notin E(G)$ and consequently  $x'y' \in E(G)$. Since $x^*$ and $y'^*$ both belong to the hole $x^*P_xx'y'^*P^*_yywx^*$, by \eqref{c:holeTSEven} we get that $P$ and this hole induce either a theta or a prism.

\medskip

- Subcase 2.1: $x^*P^*_xx'y'^*P^*_yywx^*$ and $P$ induce a theta.  

Then $P$ has length $\ell -2$ and $aw, bx' \notin E(G)$. 
%Since the partition is proper there exists $r \in A_S \sm \{y\}$. Now if $x'r' \in E(G)$, the hole $ax^*wrP_rr'x'y'^*bPa$ has length $2 \ell+2$, so it should be that $xr \in E(G)$. 
%So, by setting $\mathbb A_K= A_K\sm \{x\}, \mathbb A_S= A_S \cup \{a,x^+\}, \mathbb B = B \cup\{x^*\}, \mathbb A'_K= A'_K \sm \{x'\}, \mathbb A'_S= (A'_S \sm \{y'\}) \cup \{y'^*,b, x'\}, \mathbb B'=B'$ 
%we obtain an even pretemplate partition contained in $G$ with $k+1$ paths (the new paths are $aPb$ and $x^+x'$ each of length $\ell-2$), a contradiction.
Assume that $y$ has a neighbour $v$ in $A$. Then since $y \in A_S$ we have that $v \in A_K$ and hence $v'y'^* \notin E(G)$ and $xv, x'v' \in E(G)$. Hence to avoid a $C_4$ induced by $\{b, x', y'^*, v'\}$ it should be that $bv' \notin E(G)$. So $aPby'^*x'v'P_vvxa$ is a cycle of length $2\ell+2$ whose only possible chord is $va$, a contradiction to the fact that $G$ belongs to $C_{2\ell}$.
We may then conclude that $y$ has no neighbor in $A$ and so by setting $\mathbb A_K= (A_K\sm \{x\}) \cup \{x^*,w\}, \mathbb A_S= (A_S \sm \{y\})\cup \{a\}, \mathbb B = \emptyset, \mathbb A'_K=A'_K  \cup \{y'^*\}, \mathbb A'_S= (A'_S \sm \{y'\}) \cup \{b\}, \mathbb B'=B'$,
%we obtain an even pretemplate partition contained in $G$ with $k+1$ paths (the new paths are $aPb$ and $x^+x'$ each of length $\ell-2$), a contradiction.
we obtain an even pretemplate partition contained in $G$ with $k+1$ paths (the new paths are $aPb$ of length $\ell-2$ and $x^*P^*_xx'$ and $wyP^*_yy'^*$ of length $\ell-1$, $\mathbb A_S$ and $\mathbb A'_S$ are both stable sets), a contradiction.

\medskip

- Subcase 2.2: $x^*P^*_xx'y'^*P^*_yywx^*$ and $P$ induce a prism. Then $P$ has length $\ell-1$ and $aw, bx' \in E(G)$. 

So, by setting $\mathbb A_K= (A_K \sm \{x\}) \cup \{x^*,a\}, \mathbb A_S= A_S, \mathbb B = B, \mathbb A'_K= A'_K \cup \{b\}, \mathbb A'_S= (A'_S \sm \{y'\}) \cup \{y'^*\}, \mathbb B'=B'$ 
we obtain an even pretemplate partition contained in $G$ with $k+1$ paths (the new paths are $aPb$, $x^*P^*_xx'$ and $yP^*_yy'^*$ of lengths respectively $\ell-1$, $\ell-1$, $\ell-2$), a contradiction.

\bigskip

\noindent{\bf Case 3:} $S_a \cap A= \emptyset$, $S_a$ contains a vertex $x \in B$ and $S_b$ contains a vertex $y' \in A'_S$ (symmetric to the case where $S_b \cap A'= \emptyset$, $S_b$ contains a vertex $y' \in B'$ and $S_a$ contains a vertex $x \in A_S$).

If $xy \in E(G)$ then because of the hole $ax^*yP_yy'^*bPa$ the path $P$ has length $\ell-1$ and once again it is enough to replace $x$ by $x^*$, $y'$ by $y'^*$ and add $a$ to $A_K$ and $b$ to $A'_K$ in order to obtain an $\ell$-pretemplate partition with $k+1$ principal paths, hence $xy \notin E(G)$. So, we may assume that $a$ is not adjacent to $w$. Then because of $ax^*wyP^*_yy'^*bPa$, the path $P$ has length $\ell-2$. 

So, by setting $\mathbb A_K= A_K \cup \{x^*\}, \mathbb A_S= A_S, \mathbb B =  \{w\}, \mathbb A'_K= A'_K \cup \{b\}, \mathbb A'_S= (A'_S \sm \{y'\}) \cup \{y'^*\}, \mathbb B'=B'$ 
we obtain an even pretemplate partition contained in $G$ with $k+1$ paths (the new paths are $x^*aPb$ of length $\ell-1$ and $yP^*_yy'^*$ of length $\ell-2$), a contradiction.

\end{proofclaim}

We may now conclude the proof. By \eqref{c:lengthEven} we have that $S_a \subseteq A_K \cup B$ and $S_b \subseteq A'_K \cup B'$. Notice that if the length of $P$ is $\ell-1$ (resp.\ $\ell-2$) then we get a contradiction since we can add $a$ to $A_K$ and $b$ to $A'_K$ (resp.\  $a$ to $A_S$ and $b$ to $A'_S$) and replace any vertex $x \in S_a$ and any vertex $y' \in S_b$ by respectively  $x^* \in K_x \cap S^*_a$ and $y'^* \in K_y \cap S^*_b$ in order to get an $\ell$-pretemplate partition with one more principal path. So from now on we may assume that the length of $P$ is neither $\ell-1$ nor $\ell-2$.

Let us assume that there exist $x \in S_a \cap A_K$ and $y' \in S_b \cap A'_K$ and let $x^* \in K_x \cap S^*_a$ and  $y'^* \in K_y' \cap S^*_b$. If $x=y$ then $ax^*P^*_xy'^*bPa$ is a hole, and hence $P$ has length $\ell-1$, a contradiction. So we may now assume  that $x \neq y$, $S_a \cap K_y = \emptyset$ and $S_b \cap K_{x'} = \emptyset$. Then because of the hole $ax^*yP^*_yy'^*bPa$ the length of $P$ is $\ell-2$, a contradiction again.

%So from now on we may assume by symmetry that $S_a \cap A_K= \emptyset$. Then there exists $x^* \in K_x \cap S^*_a$ for some $x \in B$. Assume that there exists $y'^* \in K_y' \cap S^*_b$ for some $y' \in A_K.$ In case $xy \in E(G)$ then $P$ has length $\ell-2$ because of the hole $ax^*yP^*_yy^*bPa$ and we would get a contradiction. Hence $xy \notin E(G)$ and we may assume that $a$ has no neighbor in $K_w$. Now the hole $ax^*wyP^*_yy^*bPa$ implies that $P$ has length $\ell-3$. Moreover $x$ has two neighbors $r,s \in A_S$. In case $r'y' \in E(G)$ then $ax^*rP_rr'y'^*bPa$ is a hole of length $2\ell -1$, so $r'y' \notin E(G)$ and hence $ry \in E(G)$. For the same reason we have that $sy \in E(G)$ and $G$ contains a square $x^*rysx^*$, a contradiction again.

So from now on we may assume by symmetry that $S_a \cap A_K= \emptyset$. Then there exists $x^* \in K_x \cap S^*_a$ for some $x \in B$. Assume that there exists $y'^* \in K_y' \cap S^*_b$ for some $y' \in A_K.$ In case $xy \in E(G)$ then $P$ has length $\ell-2$ because of the hole $ax^*yP^*_yy'^*bPa$ and we would get a contradiction. Hence $xy, aw \notin E(G)$. Now the hole $ax^*wyP^*_yy'^*bPa$ implies that $P$ has length $\ell-3$. Then any vertex $u\in A$ which is adjacent to $x$, is adjacent to $y$: if $u \in A_K$ by definition of an even template, if $u \in A_S$ because else $ax^*uP_uu'y'^*bPa$ is hole of length $2\ell -1$. However, as any vertex in $B$, $x$ has two non adjacent neighbors $r,s \in A$ and so $G$ contains a square $x^*rysx^*$, a contradiction. So from now on we may assume that $S_b \subseteq B'$ and there exists $y'^* \in K_{y'} \cap S^*_b$ for some $y' \in B'$. By definition, $x$ has at least one neighbor $r$ in $A_S$. If $y'r' \in E(G)$ then because of the hole $ax^*rP_rr'y'^*bPa$ we get that $P$ has length $\ell-2$, a contradiction. So $r'y' \notin E(G)$ and then $y' \neq w'$, $bw' \notin E(G)$ and $P$ should have length $\ell -3$,  because of the cycle $ax^*rP_rr'w'y'^*bPa$. By symmetry, we also have $x \neq w$. 

So, by setting $\mathbb A_K= A_K \cup \{x^*\}, \mathbb A_S= A_S, \mathbb B =  \{w\}, \mathbb A'_K= A'_K \cup  \{y^*\}, \mathbb A'_S= A'_S, \mathbb B'=\{w'\}$ 
we obtain an even pretemplate partition contained in $G$ with $k+1$ paths (the new path is $x^*aPby^*$ of length $\ell-1$), a contradiction.

\end{proof}

\subsection{End of the proof}

We may now conclude the proof of Lemma~\ref{l:GPyEven}. 
 If $G\sm F_1$ is empty, then conclusion~\eqref{c:PyblowupEven}
holds.  If $G\sm F_1$ is non-empty and $G\sm F_2$ is empty, then
conclusion~\eqref{c:Pyuniv} holds.  Otherwise, we consider a connected
component $D$ of $G\sm F_2$ and apply Lemma~\ref{l:attachCompEven}.  We
then see that $G$ has a clique cutset, so
conclusion~\eqref{c:Pycliquecut} holds. 

\section{Proof of Theorem~\ref{th:structEven}}

\begin{theorem}
  \label{th:structEven}
  Let $\ell \ge 4$ be an integer.  If $G$ is a graph in $\mathcal C_{2\ell}$ then
  one of the following holds:

  \begin{enumerate}
  \item\label{c:ring} $G$ is a ring of length $2\ell$;
  \item\label{c:blowup} $G$ is a proper blowup of a twinless even
    $\ell$-template;
  \item\label{c:univ} $G$ has a universal vertex or
  \item\label{c:cliquecut} $G$ has a clique cutset.
  \end{enumerate}
\end{theorem}

\begin{proof}
  By Lemma~\ref{l:holeTruemper}, $G$ contains no unbalanced prism, no unbalanced theta, no pyramid and no
  proper wheel.  Also, clearly $G$ contains no $C_4$ and no $C_5$.
  Hence, by Theorem~\ref{th:ring}, we may assume that $G$ contains a
  prism or a theta for otherwise one of the conclusions~\eqref{c:ring},
  \eqref{c:univ} or~\eqref{c:cliquecut} holds.  The result then follows
  from Lemma~\ref{l:GPyEven}. 
\end{proof}

\end{document}